\documentclass{amsart}
\usepackage{a4}

\usepackage{hyperref}
\usepackage[italian, english]{babel}

\usepackage{mathtools} 

\theoremstyle{definition}
\newtheorem{definition}{Definition}[section]
\newtheorem{theorem}[definition]{Theorem}

\newtheorem{proposition}[definition]{Proposition}
\newtheorem{corollary}[definition]{Corollary}
\newtheorem{remark}[definition]{Remark}

\newtheorem*{assumption}{Assumption}

\usepackage{xcolor} 
\usepackage[font={small,it}]{caption}

\usepackage{tikz-cd} 
\usepackage{enumerate} 

\begin{document}

\title[Singularities and projective normality of EF 3-folds]{On the singularities and on the projective normality of some Enriques-Fano threefolds}
\author[V. Martello]{Vincenzo Martello}
\address{Dipartimento di Matematica, Universit\`{a} della Calabria, Arcavacata di Rende (CS)}
\email{vincenzomartello93@gmail.com}


\begin{abstract}
In order to find useful information to complete the classification of Enriques-Fano threefolds, we will computationally study the singularities of some known Enriques-Fano threefolds of genus 6, 7, 8, 9, 10, 13 and 17. We will also deduce the projective normality of these threefolds.
\end{abstract}

\maketitle

\section{Introduction}

An \textit{Enriques-Fano threefold} is a normal threefold $W$ endowed with a complete linear system $\mathcal{L}$ of ample Cartier divisors such that the general element $S\in \mathcal{L}$ is an Enriques surface and such that $W$ is not a \textit{generalized cone} over $S$, i.e., $W$ is not obtained by contraction of the negative section on the $\mathbb{P}^1$-bundle $\mathbb{P}(\mathcal{O}_{S}\oplus \mathcal{O}_{S}(S))$ over $S$. The linear system $\mathcal{L}$ defines a rational map $\phi_{\mathcal{L}} : W \dashrightarrow \mathbb{P}^{p}$, where $p:=\frac{S^3}{2}+1$ is called the \textit{genus} of $W$ and $2\le p\le 17$ (see \cite{KLM11} and \cite{Pro07}). 
If the elements of $\mathcal{L}$ are very ample divisors, then $W$ is embedded in $\mathbb{P}^p$ via $\phi_{\mathcal{L}}$ as a 
non-degenerate threefold whose general hyperplane section is an Enriques surface.
It is known that any Enriques-Fano threefold is singular with isolated canonical singularities (see \cite[Lemma 3.2]{CoMu85} and \cite{Ch96}). 
We will say that two distinct singular points of an Enriques-Fano threefold $W$ are \textit{associated} if the line joining them is contained in $W$.
The way in which the singular points of an Enriques-Fano threefold are associated is called \textit{configuration} and can be represented graphically: if two singular points are associated, one draws a segment joining them, otherwise not. In Appendix~\ref{app:congifuration} we will graphically describe the configurations that we will find in this paper.
\begin{definition}\label{def:similar}
The singular points of an Enriques-Fano threefold $W$ are said to be \textit{similar} if
they have the same multiplicity,
they have the same tangent cone,
and
there is an $m$ such that each singular point is associated with exactly $m$ other singular points.
\end{definition}
The first examples of Enriques-Fano threefolds were discovered by Fano (see \cite{Fa38}), under the assumption that their singularities were similar and that the blow-up at the singular points was sufficient to resolve them. 
It must be said that Fano's arguments contain many gaps, some of which have been solved by Conte and Murre (see \cite{CoMu85}). By using a sort of inverse of \cite[Theorem 7.2]{CoMu85}, Fano constructed four rational Enriques-Fano threefolds having eight similar 
singular points whose configurations are the ones in Table~\ref{tab:BS67913} of Appendix~\ref{app:congifuration}. They are: 
\begin{itemize}
\item[(i)] the Enriques-Fano threefold $W_{F}^{6}\subset \mathbb{P}^{6}$ of genus $6$ given by the image of 
the rational map defined by the linear system 
 of the septic surfaces with double points along three twisted cubics having five points in common (see \cite[\S 3]{Fa38});
\item[(ii)] the Enriques-Fano threefold $W_{F}^{7}\subset \mathbb{P}^{7}$ of genus $7$ given by the image of 
the rational map defined by the linear system 
of the sextic surfaces having double points along the six edges of a tetrahedron and containing a plane cubic curve intersecting each edge at one point (see \cite[\S 4]{Fa38});
\item[(iii)] the Enriques-Fano threefold $W_{F}^{9}\subset \mathbb{P}^{9}$ of genus $9$ given by the image of 
the rational map defined by the linear system 
of the septic surfaces having double points along the six edges of two trihedra (see \cite[\S 7]{Fa38});
\item[(iv)] the Enriques-Fano threefold $W_{F}^{13}\subset \mathbb{P}^{13}$ of genus $13$ given by the image of 
the rational map defined by the linear system 
of the sextic surfaces having double points along the six edges of a tetrahedron (see \cite[\S 8]{Fa38}).
\end{itemize}
Fano also found an ``exceptional'' example of genus $4$, which is a sextic hypersurface $W_F^4$ of $\mathbb{P}^4$ whose general hyperplane section is a sextic surface of $\mathbb{P}^3$ having double points along the six edges of a tetrahedron (see \cite[\S 10]{Fa38}). The threefold $W_F^4$ has been proved to be \textit{non-rational} (see \cite{P-BV83}).
Furthermore, as noted by Conte in \cite[p. 225]{Co82}, there is also another exceptional example of genus $3$, which is the threefold $W_F^3$ given by a quadruple $\mathbb{P}^3$ (see \cite[\S 2]{Fa38}).
We will refer to the above threefolds $W_F^{p=3,4,6,7,9,13}$ as \textit{F-EF 3-folds}.
However, Fano's classification is incomplete: indeed, it fails to include some other Enriques-Fano threefolds which have been discovered.
Under the assumption that the singularities are terminal cyclic quotients, Enriques-Fano threefolds were classified by Bayle (see \cite{Ba94}) and, in a similar and independent way, by Sano (see \cite{Sa95}). We will refer to such Enriques-Fano threefolds as \textit{BS-EF 3-folds}: they are fourteen and they have genus $2\le p \le 10$ or $p=13$.
Six of them are endowed with a linear system of very ample divisors (see \cite[Theorem A]{Ba94}): they have genus $p=6,7,8,9,10,13$ and we will denote them, respectively, by $W_{BS}^{6}$, $W_{BS}^{7}$, $W_{BS}^{8}$, $W_{BS}^{9}$, $W_{BS}^{10}$, $W_{BS}^{13}$. 
We recall that a fixed BS-EF 3-fold 
$W$ 
is the quotient of a smooth Fano threefold 
$X$ 
via an involution 
$\sigma$ 
having $8$ fixed points, and that 
$W$ 
has $8$ quadruple points whose tangent cone is a cone over a Veronese surface.
More generally, an Enriques-Fano threefold with only terminal singularities is a limit of some BS-EF 3-fold (see \cite[Main Theorem 2]{Mi99}). 
Instead, only a few examples of Enriques-Fano threefolds with non-terminal canonical singularities are known: two of genus $p=13, 17$ found by Prokhorov (see \cite[Proposition 3.2, Remark 3.3]{Pro07}) and one of genus $p=9$ found by Knutsen, Lopez and Mu\~{n}oz (see \cite[\S 13]{KLM11}).
The Enriques-Fano threefolds of genus $13$ and $17$ found by Prokhorov (shortly, \textit{P-EF 3-folds}) are obtained as the quotient of singular Fano threefolds $V$ via an involution $\tau$ having five fixed points: we will denote them by $W_P^{13}$ and $W_{P}^{17}$, respectively.
The Enriques-Fano threefold found by Knutsen-Lopez-Mu\~{n}oz (shortly, \textit{KLM-EF 3-fold}) is a threefold $W_{KLM}^9\subset \mathbb{P}^9$ given by the projection of the F-EF 3-fold $W_F^{13}\subset \mathbb{P}^{13}$ from the $\mathbb{P}^3$ spanned by a certain elliptic quartic curve $E_3\subset W_F^{13}$. 

In order to find useful information to complete the classification of Enriques-Fano threefolds, we will study the threefolds $W_{BS}^{p=6,7,8,9,10,13}$, $W_P^{p=13,17}$, $W_{KLM}^9$. We will do this from a computational point of view, with the help of the software Macaulay2. In Appendix~\ref{app:code} we will collect the input codes used in Macaulay2.

Since the rational F-EF 3-folds $W_F^{p=6,7,9,13}$ have eight quadruple points whose tangent cone is a cone over a Veronese surface (see \cite[p. 44]{Fa38}), then they only have terminal singularities (see \cite[Example 1.3]{Re87}) and therefore they are limits of the BS-EF 3-folds $W_{BS}^{p=6,7,9,13}$ (see \cite[Main Theorem 2]{Mi99}). 
We can say something more about this link between F-EF 3-folds and BS-EF 3-folds. Indeed, in \S~\ref{subsec:bayle6},~\ref{subsec:bayle7},~\ref{subsec:bayle9},~\ref{subsec:bayle13} we will prove the following results.

\begin{theorem}\label{thm:BS67913}
Let $p\in \{7,13\}$. Then the eight quadruple points of the BS-EF 3-fold $W_{BS}^p \hookrightarrow \mathbb{P}^p$ are similar and they have the same configuration as the ones of the F-EF 3-fold $W_{F}^p \subset \mathbb{P}^p$. Furthermore, we have the following facts
\begin{itemize}
\item[(i)] the ideal of $W_{BS}^6\subset \mathbb{P}^6$ is generated by cubics;
\item[(ii)] the ideal of $W_{BS}^7\subset \mathbb{P}^7$ is generated by quadrics and cubics;
\item[(iii)] the ideal of $W_{BS}^9\subset \mathbb{P}^9$ is generated by quadrics;
\item[(iv)] the ideal of $W_{BS}^{13}\subset \mathbb{P}^{13}$ is generated by quadrics.
\end{itemize}
\end{theorem}

In \S~\ref{subsec:bayle6},~\ref{subsec:bayle9} we will see that the eight quadruple points of particular examples of BS-EF 3-folds of genus $6$ (respectively, genus $9$) are also similar and they have the same configuration as the ones of the F-EF 3-fold of genus $6$ (respectively, genus $9$). It would be interesting to show it for general $W_{BS}^6$ and $W_{BS}^9$.


\begin{theorem}\label{thm:B13=F13}
The embedding of the BS-EF 3-fold of genus $13$ in $\mathbb{P}^{13}$ is a F-EF 3-fold of genus $13$.
\end{theorem} 

In \S~\ref{subsec:bayle9} we will also prove that particular BS-EF 3-folds of genus $9$ are isomorphic to F-EF 3-folds of genus $9$. It would be interesting to show it for general $W_{BS}^9$.

We have other links between F-EF 3-folds and BS-EF 3-folds. The BS-EF 3-fold of genus $4$ described in \cite[\S 6.33]{Ba94} (see also \cite[Theorem 1.1 No.5]{Sa95}) is endowed with a linear system defining a birational map onto the image, which is the F-EF 3-fold $W_F^4\subset \mathbb{P}^4$.
The BS-EF 3-fold of genus $3$ described in \cite[\S 6.1.5]{Ba94} (see also \cite[Theorem 1.1 No.2]{Sa95}) is endowed with a linear system defining a quadruple cover over $\mathbb{P}^3$, so it looks like the F-EF 3-fold $W_F^3$. This suggests that one could obtain the BS-EF 3-folds with ample (but not very ample) hyperplane sections by resuming the Fano-Conte-Murre techniques and by undermining some assumptions. Thus, re-examining the brilliant ideas of Fano with the techniques of Conte and Murre would be very interesting, since new Enriques-Fano threefolds could be found, even if this problem seems not to have been studied yet.
However, one must be careful of mistakes in resuming Fano's paper. For example, the BS-EF 3-folds $W_{BS}^{8}$ and $W_{BS}^{10}$ do not appear in the description of Fano, although they behave like the other BS-EF 3-folds with very ample hyperplane sections. It is a situation that should be understood better. The problem could be the fact that Fano stated that if the genus is greater than $7$, then there are no three mutually associated singular points (see \cite[\S 5]{Fa38}). Indeed, we will see that the singularities of both threefolds $W_{BS}^8$ and $W_{BS}^{10}$ have configurations which are in contradiction with this assertion. In \S~\ref{subsec:bayle8},~\ref{subsec:bayle10} we will prove the following result.

\begin{theorem}
Let $p\in \{8,10\}$. The ideal of $W_{BS}^p\subset \mathbb{P}^p$ is generated by quadrics and cubics. Furthermore the eight singular points of $W_{BS}^{10}$ are similar and they have the configuration in Table~\ref{tab:BS810} of Appendix~\ref{tab:BS810}.
\end{theorem}

In \S~\ref{subsec:bayle8} we will show that the eight singular points of particular examples of B-EF 3-folds of genus $8$ are also similar and they have the configuration in Table~\ref{tab:BS810} of Appendix~\ref{tab:BS810}. It would be interesting to show it for a general $W_{BS}^8$.

Furthermore, we will identify $W_{BS}^{8}$ and $W_{BS}^{10}$ as images of rational maps defined by linear systems of surfaces of $\mathbb{P}^3$, as happens for the rational F-EF 3-folds (see Theorem~\ref{thm:B8 linear system},~\ref{thm:B10linearsystem}). We will also study the singularities of the P-EF 3-folds and the KLM-EF 3-fold. It is known that these threefolds have canonical non-terminal singularities, but so far there was no information about their multiplicities and tangent cones. This analysis is very interesting because it provides some clues about the link between the non-terminality of the singularities, the non-similarity of the singularities and the fact that the blow-up at the singularities is not sufficient to resolve them. In \S~\ref{subsec:pro17},~\ref{subsec:pro13} we will prove the following result.

\begin{theorem}
Let $p\in \{13, 17\}$. Then the P-EF 3-fold of genus $p$ can be embedded in $\mathbb{P}^{p}$ and its ideal is generated by quadrics. Furthermore $W_P^{p}$ has five non-similar singular points whose configuration is the one in Table~\ref{tab:proKLM} of Appendix~\ref{app:congifuration}. Moreover, four of the these five points are quadruple points, whose tangent cone is a cone over a Veronese surface. If $p=13$, the last point is a quintuple point, whose tangent cone is a cone over the union of five planes; if $p=17$, it is a sextuple point, whose tangent cone is a cone over the union of four planes and a quadric surface.
\end{theorem}

For further details see also Propositions~\ref{prop:cones1234p17},~\ref{prop:cones1234p13} and Theorems~\ref{thm:cone5p17},~\ref{thm:cone5p13}.
As for the KLM-EF 3-fold, we refer to \S~\ref{sec:KLM}. We will find that the ideal of $W_{KLM}^9$ in $\mathbb{P}^9$ is generated by quadrics and cubics. We will see that the images of the eight quadruple points of $W_F^{13}$, via the projection map, are five singular points of $W_{KLM}^9$ such that four of them are quadruple points, whose tangent cone is a cone over a Veronese surface (see Proposition~\ref{prop:cones1234KLM}), and one is a sextuple point, whose tangent cone is a cone over the union of four planes and a quadric surface (see Theorem~\ref{thm:cone5KLM}). 

Finally, in \S~\ref{subsec:normality}, we will study the projective normality of all the above Enriques-Fano threefolds, obtaining the following result.

\begin{theorem}\label{thm:projnormKnown}
The following Enriques-Fano threefolds are projectively normal:
$$W_{KLM}^9\subset \mathbb{P}^9, \quad W_F^{p=7,9,13}\subset \mathbb{P}^p, \quad W_{BS}^{p=6,7,8,9,10,13} \xhookrightarrow{\phi_{\mathcal{L}}} \mathbb{P}^p, \quad W_{P}^{p=13,17}\subset \mathbb{P}^p.$$
\end{theorem}

We will work over the field $\mathbb{C}$ of the complex numbers. For the computational analysis we will work over a finite field (we will choose $\mathbb{F}_n := \mathbb{Z}/n\mathbb{Z}$ with $n=10000019$).

\subsection*{Acknowledgment}
The results of this paper are contained in my PhD-thesis. I would like to thank my main advisors C. Ciliberto and C. Galati and my co-advisor A.L. Knutsen for our stimulating conversations and for providing me very useful suggestions.
I would also like to acknowledge PhD-funding from the Department of Mathematics and Computer Science of the University of Calabria and funding from Research project ``Families of curves: their moduli and their related varieties'' (CUP E81-18000100005, P.I. Flaminio Flamini) in the framework of Mission Sustainability 2017 - Tor Vergata University of Rome.


\section{The BS-EF 3-fold of genus 6}\label{subsec:bayle6}

Let us study the Enriques-Fano threefold described in \cite[\S 6.2.4]{Ba94} (see also \cite[Theorem 1.1 No.9]{Sa95}). We refer to \S~\ref{code:bayle6} of Appendix~\ref{app:code} for the computational techniques we will use.  
Let us take the smooth Fano threefold $X$ given by the intersection of three divisors of bidegree $(1,1)$ of $\mathbb{P}^{3}_{[x_0:\dots :x_3]}\times \mathbb{P}^{3}_{[y_0:\dots :y_3]}$ with equations
$$\sum_{i=0}^3 \sum_{j=0}^{3} a_{ij}x_i y_j = 0,\quad \sum_{i=0}^3 \sum_{j=0}^{3} b_{ij}x_i y_j = 0,\quad\sum_{i=0}^3 \sum_{j=0}^{3} c_{ij}x_i y_j = 0,$$
where $a_{ij}=a_{ji}$, $b_{ij}=b_{ji}$, $c_{ij}=c_{ji}$, for $i,j\in\{0,1,2,3\}$.
Let $\sigma : X \to X$ be the involution of $X$ defined by the restriction on $X$ of the following map 
$$\sigma' : \mathbb{P}^3 \times \mathbb{P}^3 \to \mathbb{P}^3 \times \mathbb{P}^3, \,  \left[x_0 : \dots : x_3 \right] \times \left[y_0 : \dots : y_3 \right] \mapsto \left[y_0 : \dots : y_3 \right] \times \left[x_0 : \dots : x_3 \right].$$
We have that $\sigma$ has eight fixed points $p_1$, $p_2$, $p_3$, $p_4$, $p_5$, $p_6$, $p_7$ and $p_8$ with coordinates $\left[x_0 : x_1: x_2 : x_3 \right] \times \left[x_0 : x_1: x_2 : x_3 \right] $ such that
$$\begin{cases}
\sum_{i=0}^3 \sum_{j=0}^{3} a_{ij}x_i x_j = 0\\ 
\sum_{i=0}^3 \sum_{j=0}^{3} b_{ij}x_i x_j = 0\\
\sum_{i=0}^3 \sum_{j=0}^{3} c_{ij}x_i x_j = 0.
\end{cases}$$
The quotient map $\pi : X \to X/\sigma = : W_{BS}^{6}$ is given by the restriction on $X$ of the 
morphism $\varphi : \mathbb{P}^{3}\times \mathbb{P}^3 \to \mathbb{P}^{9}$ defined by the $\sigma'$-invariant multihomogeneous polynomials of multidegree $(1,1)$. Thus 
$\varphi : \left[x_0 : \dots : x_3 \right] \times \left[y_0 : \dots : y_3 \right] \mapsto \left[Z_0:\dots : Z_9 \right],$
where $Z_0=x_0y_0$, $Z_1=x_1y_1$, $Z_2=x_2y_2$, $Z_3=x_3y_3$, $Z_4=x_0y_1+x_1y_0$, $Z_5=x_0y_2+x_2y_0$, $Z_6=x_0y_3+x_3y_0$, $Z_7=x_1y_2+x_2y_1$, $Z_8=x_1y_3+x_3y_1$, $Z_9=x_2y_3+x_3y_2.$ 
By using Macaulay2, one can find that the image of $\mathbb{P}^{3}\times \mathbb{P}^3$ via $\varphi$ is a $6$-dimensional algebraic variety $F_{6}^{10}$ of degree $10$, whose ideal is generated by the following $10$ polynomials
$$-2Z_1Z_5Z_6+Z_4Z_6Z_7+Z_4Z_5Z_8-2Z_0Z_7Z_8+4Z_0Z_1Z_9-Z_4^2Z_9,$$
$$-2Z_2Z_4Z_6+Z_5Z_6Z_7+4Z_0Z_2Z_8-Z_5^2Z_8+Z_4Z_5Z_9-2Z_0Z_7Z_9,$$
$$-4Z_1Z_2Z_6+Z_6Z_7^2+2Z_2Z_4Z_8-Z_5Z_7Z_8+2Z_1Z_5Z_9-Z_4Z_7Z_9,$$
$$-2Z_3Z_4Z_5+4Z_0Z_3Z_7-Z_6^2Z_7+Z_5Z_6Z_8+Z_4Z_6Z_9-2Z_0Z_8Z_9,$$
$$-4Z_1Z_3Z_5+2Z_3Z_4Z_7-Z_6Z_7Z_8+Z_5Z_8^2+2Z_1Z_6Z_9-Z_4Z_8Z_9,$$
$$-4Z_2Z_3Z_4+2Z_3Z_5Z_7+2Z_2Z_6Z_8-Z_6Z_7Z_9-Z_5Z_8Z_9+Z_4Z_9^2,$$
$$-4Z_1Z_2Z_3+Z_3Z_7^2+Z_2Z_8^2-Z_7Z_8Z_9+Z_1Z_9^2,$$
$$-4Z_0Z_2Z_3+Z_3Z_5^2+Z_2Z_6^2-Z_5Z_6Z_9+Z_0Z_9^2,$$
$$-4Z_0Z_1Z_3+Z_3Z_4^2+Z_1Z_6^2-Z_4Z_6Z_8+Z_0Z_8^2,$$
$$-4Z_0Z_1Z_2+Z_2Z_4^2+Z_1Z_5^2-Z_4Z_5Z_7+Z_0Z_7^2.$$
Let us observe that $W_{BS}^{6}= \varphi(X) = F_{6}^{10}\cap H_6$, where $H_6$ is the $6$-dimensional projective subspace of $\mathbb{P}^{9}$ given by the zero locus of the following three polynomials
$$a_{00}Z_0+a_{11}Z_1+a_{22}Z_2+a_{33}Z_3+2a_{01}Z_4+2a_{02}Z_5+2a_{03}Z_6+2a_{12}Z_7+2a_{13}Z_8+2a_{23}Z_9,$$
$$b_{00}Z_0+b_{11}Z_1+b_{22}Z_2+b_{33}Z_3+2b_{01}Z_4+2b_{02}Z_5+2b_{03}Z_6+2b_{12}Z_7+2b_{13}Z_8+2b_{23}Z_9,$$
$$c_{00}Z_0+c_{11}Z_1+c_{22}Z_2+c_{33}Z_3+2c_{01}Z_4+2c_{02}Z_5+2c_{03}Z_6+2c_{12}Z_7+2c_{13}Z_8+2c_{23}Z_9.$$
Therefore we have $\pi = \varphi |_{X} : X \to W_{BS}^{6}=\varphi (X) \subset H_6 \cong \mathbb{P}^{6}$.
What follows has been proved for fixed values of $a_{ij}$, $b_{ij}$ and $c_{ij}$, in order to simplify the computational analysis.

\subsection{Example}\label{ex:bayle6}
Let us take
$$(a_{ij})=\begin{pmatrix}
1 & 0 & 0 & 0\\
0 & -7 & 0 & 0\\ 
0 & 0 & 4 & 0 \\
0 & 0 & 0 & 2
\end{pmatrix}, \, 
(b_{ij})=\begin{pmatrix}
1 & 0 & 0 & 0\\
0 & -6 & 0 & 0\\ 
0 & 0 & 2 & 0 \\
0 & 0 & 0 & 3
\end{pmatrix}, \, 
(c_{ij})=\begin{pmatrix}
1 & 0 & 0 & 0\\
0 & -1 & 0 & 0\\ 
0 & 0 & -7 & 0 \\
0 & 0 & 0 & 7
\end{pmatrix}.$$
Then the eight fixed points of $\sigma : X \to X$ are 
$$p_1 = \left[1:1:1:1\right] \times \left[ 1:1:1:1\right],\,\,
p_2 = \left[-1:1:1:1\right] \times \left[-1:1:1:1\right],$$
$$p_3 = \left[1:-1:1:1\right] \times \left[1:-1:1:1\right],\,
p_4 = \left[-1:-1:1:1\right] \times \left[-1:-1:1:1\right],$$
$$p_5 = \left[1:1:-1:1\right] \times \left[1:1:-1:1\right],\,
p_6 = \left[-1:1:-1:1\right] \times \left[-1:1:-1:1\right],$$
$$p_7 = \left[1:-1:-1:1\right] \times \left[1:-1:-1:1\right],\, 
p_8 = \left[-1:-1:-1:1\right] \times \left[-1:-1:-1:1\right].$$
Furthermore, we have
$$H_6= \{Z_0-7Z_1+4Z_2+2Z_3=0, Z_0-6Z_1+2Z_2+3Z_3=0, Z_0-Z_1-7Z_2+7Z_3=0\}=$$
$$=\{Z_2-Z_3=0,\,Z_1-Z_3=0,\,Z_0-Z_3=0\},$$
which is the $\mathbb{P}^{6}_{\left[w_0:\dots : w_6 \right]}$ embedded in $\mathbb{P}^{9}_{\left[Z_0:\dots :Z_{9} \right]}$ via the morphism such that
\begin{center}
$Z_i = w_0,\, i=0,1,2,3, \quad \quad Z_j = w_{j-3},\, j = 4,\dots ,9.$
\end{center}
By using Macaulay2, we find that the quotient map $\pi : X \to W_{BS}^{6} \subset H_6 \cong \mathbb{P}^{6}$ is given by the restriction on $X$ of the morphism $\varphi' : \mathbb{P}^{3}\times \mathbb{P}^3 \to \mathbb{P}^{6}$ defined by
$\left[ x_{0} : x_1 : x_2 : y_3 : y_4 : y_5 \right] \mapsto \left[w_0:\dots :w_6 \right]$,
where 
$$w_0=x_3y_3,\, w_1=x_0y_1+x_1y_0,\, w_2=x_0y_2+x_2y_0,\, w_3=x_0y_3+x_3y_0,$$ 
$$w_4=x_1y_2+x_2y_1, \, w_5=x_1y_3+x_3y_1, \, w_6=x_2y_3+x_3y_2.$$
Thanks to Macaulay2, we obtain that this BS-EF 3-fold $W_{BS}^{6}\subset \mathbb{P}^{6}$ has ideal generated by the following $10$ polynomials
$$-2w_0w_2w_3+w_1w_3w_4+w_1w_2w_5-2w_0w_4w_5+4w_0^2w_6-w_1^2w_6,$$
$$-2w_0w_1w_3+w_2w_3w_4+4w_0^2w_5-w_2^2w_5+w_1w_2w_6-2w_0w_4w_6,$$
$$-4w_0^2w_3+w_3w_4^2+2w_0w_1w_5-w_2w_4w_5+2w_0w_2w_6-w_1w_4w_6,$$
$$-2w_0w_1w_2+4w_0^2w_4-w_3^2w_4+w_2w_3w_5+w_1w_3w_6-2w_0w_5w_6,$$
$$-4w_0^2w_2+2w_0w_1w_4-w_3w_4w_5+w_2w_5^2+2w_0w_3w_6-w_1w_5w_6,$$
$$-4w_0^2w_1+2w_0w_2w_4+2w_0w_3w_5-w_3w_4w_6-w_2w_5w_6+w_1w_6^2,$$
$$-4w_0^3+w_0w_4^2+w_0w_5^2-w_4w_5w_6+w_0w_6^2,\quad -4w_0^3+w_0w_2^2+w_0w_3^2-w_2w_3w_6+w_0w_6^2,$$
$$-4w_0^3+w_0w_1^2+w_0w_3^2-w_1w_3w_5+w_0w_5^2,\quad -4w_0^3+w_0w_1^2+w_0w_2^2-w_1w_2w_4+w_0w_4^2.$$
Furthermore, $W_{BS}^6$ has eight singular points $P_i:=\pi (p_i)$, for $1\le i \le 8$; they are
$$P_1 = \left[1: 2: 2: 2: 2: 2: 2\right],\, P_2 = \left[1: -2:-2:-2: 2: 2: 2\right],$$
$$P_3 = \left[1: -2: 2: 2:-2:-2: 2\right],\,
P_4 = \left[1: 2:-2:-2:-2:-2: 2\right],$$
$$P_5 = \left[1: 2:-2: 2: -2: 2:-2\right],\,
P_6  = \left[1: -2: 2:-2: -2: 2:-2\right],$$
$$P_7 = \left[1: -2:-2: 2: 2:-2:-2\right],\, P_8 = \left[1: 2: 2:-2: 2:-2:-2\right].$$
One can verify that all the lines joining the points $P_i$ and $P_j$, for $1\le i<j \le 8$, are contained in $W_{BS}^6$. So we can say that each one of the eight singular points of $W_{BS}^{6}$ is associated with all the other $m=7$ points, as in 
Table~\ref{tab:BS67913}
of Appendix~\ref{app:congifuration}. Thus, the singularities of the BS-EF 3-fold $W_{BS}^6$ have the same configuration as the ones of the F-EF 3-fold $W_{F}^{6}$.

\section{The BS-EF 3-fold of genus 7 }\label{subsec:bayle7}

Let us study the Enriques-Fano threefold described in \cite[\S 6.4.1]{Ba94} (see also \cite[Theorem 1.1 No.11]{Sa95}). We refer to \S~\ref{code:bayle7} of Appendix~\ref{app:code} for the computational techniques we will use.
Let $X$ be the smooth Fano threefold given by a divisor of 
$\mathbb{P}^{1}_{\left[x_0 : x_1\right]}\times \mathbb{P}^{1}_{\left[y_0 : y_1\right]}\times \mathbb{P}^{1}_{\left[z_0 : z_1\right]}\times \mathbb{P}^{1}_{\left[t_0 : t_1\right]}$
of type
$$\sum_{i+j+k+l\,\, odd} a_{ijkl}x_iy_jz_kt_l = 0.$$
Let $\sigma : X \to X$ be the involution of $X$ defined by the restriction on $X$ of the map $\sigma' : \mathbb{P}^{1}\times \mathbb{P}^{1}\times \mathbb{P}^{1}\times \mathbb{P}^{1} \to \mathbb{P}^{1}\times \mathbb{P}^{1}\times \mathbb{P}^{1}\times \mathbb{P}^{1}$ given by
$$\left[x_0 : x_1\right] \times \left[y_0 : y_1\right] \times \left[z_0 : z_1\right] \times \left[t_0 : t_1\right] \mapsto \left[x_0 : -x_1\right] \times \left[y_0 : -y_1\right] \times \left[z_0 : -z_1\right] \times \left[t_0 : -t_1\right].$$
Then $\sigma : X \to X$ has the following eight fixed points
$$p_1 = \left[0 : 1\right] \times \left[0 : 1\right] \times \left[0 : 1\right] \times \left[0 : 1\right], \quad 
p_1' = \left[1 : 0\right] \times \left[1 : 0\right] \times \left[1 : 0\right]\times \left[1 : 0\right],$$ 
$$p_2 = \left[0 : 1\right] \times \left[1 : 0\right] \times \left[1 : 0\right] \times \left[0 : 1\right], \quad 
p_2' = \left[1 : 0\right] \times \left[0 : 1\right] \times \left[0 : 1\right] \times \left[1 : 0\right],$$ 
$$p_3 = \left[1 : 0\right] \times \left[1 : 0\right] \times \left[0 : 1\right]\times \left[0 : 1\right], \quad
p_3' = \left[0 : 1\right] \times \left[0 : 1\right] \times \left[1 : 0\right] \times \left[1 : 0\right],$$ 
$$p_4 = \left[1 : 0\right] \times \left[0 : 1\right] \times \left[1 : 0\right]\times \left[0 : 1\right], \quad
p_4' = \left[0 : 1\right] \times \left[1 : 0\right] \times \left[0 : 1\right]\times \left[1 : 0\right].$$
The quotient map $\pi : X \to X/\sigma = : W_{BS}^{7}$ is given by the restriction on $X$ of the 
morphism $\varphi : \mathbb{P}^1\times\mathbb{P}^1\times\mathbb{P}^1\times\mathbb{P}^1 \to \mathbb{P}^7$, defined by the $\sigma'$-invariant multihomogeneous polynomials of multidegree $(1,1,1,1)$. In particular we have
$$\varphi : \left[x_0 : x_1\right] \times \left[y_0 : y_1\right] \times \left[z_0 : z_1\right]\times \left[t_0 : t_1\right] \mapsto \left[w_0 : w_1: w_2: w_3: w_4: w_5: w_6: w_7 \right]
$$
where $w_0 = x_1y_1z_1t_1$, $w_1 = x_1y_0z_0t_1$, $w_2 = x_0y_0z_1t_1$, $w_3 = x_1y_0z_1t_0$, $w_4 = x_0y_0z_0t_0$, $w_5 = x_0y_1z_1t_0$, $w_6 = x_1y_1z_0t_0$, $w_7 = x_0y_1z_0t_1$. 
By fixing (random) values for $a_{0001}$, $a_{0010}$, $a_{0100}$, $a_{1000}$, $a_{1110}$, $a_{1101}$, $a_{1011}$ and $a_{0111}$, one can verify, with Macaulay2, that the ideal of the BS-EF 3-fold $W_{BS}^{7}$ is generated by the following $11$ polynomials of degree $2$ or $3$:
\begin{center}
${w}_{2} {w}_{6}-{w}_{3} {w}_{7},\quad {w}_{1} {w}_{5}-{w}_{3}{w}_{7}, \quad {w}_{0} {w}_{4}-{w}_{3} {w}_{7},$\\
\end{center}

\begin{center}
$a_{1110}{w}_{0} {w}_{5} {w}_{6}+a_{1011}{w}_{0}{w}_{3} {w}_{7}+a_{0111}{w}_{0} {w}_{5} {w}_{7}+a_{0010}{w}_{3} {w}_{5} {w}_{7}+a_{1101}{w}_{0}{w}_{6} {w}_{7}+a_{1000}{w}_{3} {w}_{6} {w}_{7}+a_{0100}{w}_{5} {w}_{6} {w}_{7}+a_{0001}{w}_{3}{w}_{7}^{2},$
\end{center}

\begin{center}
$a_{1000}{w}_{1} {w}_{4} {w}_{6}+a_{1011}{w}_{1} {w}_{3} {w}_{7}+a_{0001}{w}_{1}{w}_{4} {w}_{7}+a_{0010}{w}_{3} {w}_{4} {w}_{7}+a_{1101}{w}_{1} {w}_{6} {w}_{7}+a_{1110}{w}_{3}{w}_{6} {w}_{7}+a_{0100}{w}_{4} {w}_{6} {w}_{7}+a_{0111}{w}_{3} {w}_{7}^{2},$
\end{center}

\begin{center}
$a_{0010}{w}_{3}{w}_{4} {w}_{5}+a_{1000}{w}_{3} {w}_{4} {w}_{6}+a_{1110}{w}_{3} {w}_{5} {w}_{6}+a_{0100}{w}_{4}{w}_{5} {w}_{6}+a_{1011}{w}_{3}^{2} {w}_{7}+a_{0001}{w}_{3} {w}_{4} {w}_{7}+a_{0111}{w}_{3}{w}_{5} {w}_{7}+a_{1101}{w}_{3} {w}_{6} {w}_{7},$
\end{center}

\begin{center}
$a_{0010}{w}_{2} {w}_{4} {w}_{5}+a_{1011}{w}_{2}{w}_{3} {w}_{7}+a_{0001}{w}_{2} {w}_{4} {w}_{7}+a_{1000}{w}_{3} {w}_{4} {w}_{7}+a_{0111}{w}_{2}{w}_{5} {w}_{7}+a_{1110}{w}_{3} {w}_{5} {w}_{7}+a_{0100}{w}_{4} {w}_{5} {w}_{7}+a_{1101}{w}_{3}{w}_{7}^{2},$
\end{center}

\begin{center}
$a_{1011}{w}_{1} {w}_{2} {w}_{3}+a_{0001}{w}_{1} {w}_{2} {w}_{4}+a_{1000}{w}_{1}{w}_{3} {w}_{4}+a_{0010}{w}_{2} {w}_{3} {w}_{4}+a_{1101}{w}_{1} {w}_{3} {w}_{7}+a_{0111}{w}_{2}{w}_{3} {w}_{7}+a_{1110}{w}_{3}^{2} {w}_{7}+a_{0100}{w}_{3} {w}_{4} {w}_{7},$
\end{center}

\begin{center}
$a_{1011}{w}_{0}{w}_{2} {w}_{3}+a_{0111}{w}_{0} {w}_{2} {w}_{5}+a_{1110}{w}_{0} {w}_{3} {w}_{5}+a_{0010}{w}_{2}{w}_{3} {w}_{5}+a_{1101}{w}_{0} {w}_{3} {w}_{7}+a_{0001}{w}_{2} {w}_{3}{w}_{7}+a_{1000}{w}_{3}^{2} {w}_{7}+a_{0100}{w}_{3} {w}_{5} {w}_{7},$
\end{center}

\begin{center}
$a_{1011}{w}_{0} {w}_{1}{w}_{3}+a_{1101}{w}_{0} {w}_{1} {w}_{6}+a_{1110}{w}_{0} {w}_{3} {w}_{6}+a_{1000}{w}_{1} {w}_{3}{w}_{6}+a_{0111}{w}_{0} {w}_{3} {w}_{7}+a_{0001}{w}_{1} {w}_{3} {w}_{7}+a_{0010}{w}_{3}^{2}{w}_{7}+a_{0100}{w}_{3} {w}_{6} {w}_{7},$
\end{center}

\begin{center}
$a_{1011}{w}_{0} {w}_{1} {w}_{2}+a_{1101}{w}_{0} {w}_{1}{w}_{7}+a_{0111}{w}_{0} {w}_{2} {w}_{7}+a_{0001}{w}_{1} {w}_{2} {w}_{7}+{a_{1110}w}_{0} {w}_{3}{w}_{7}+a_{1000}{w}_{1} {w}_{3} {w}_{7}+a_{0010}{w}_{2} {w}_{3} {w}_{7}+a_{0100}{w}_{3}{w}_{7}^{2}.$
\end{center}
The eight singular points of $W_{BS}^7$ are
$P_i := \pi (p_i) = \{w_k = 0 | k\ne i-1\}$ and $P_i' := \pi (p_i') = \{w_k = 0| k \ne 3+i\}$, for $1\le i\le 4.$
Let $l_{i,j}$ be the line joining $P_i$ and $P_j$ with $i,j \in \{1,2,3,4,1',2',3',4'\}$ and $i \ne j$.
We have that $W_{BS}^{7}$ does not contain the lines
$l_{1,1'}$, $l_{2,2'}$, $l_{3,3'}$ and $l_{4,4'}$,
while it contains the others. So each one of the eight singular points of $W_{BS}^{7}$ is associated with $m=6$ of the other singular points, as in 
Table~\ref{tab:BS67913} of Appendix~\ref{app:congifuration}. Thus, the singularities of the BS-EF 3-fold $W_{BS}^7$ have the same configuration as the ones of the F-EF 3-fold $W_{F}^{7}$.

\section{The BS-EF 3-fold of genus 9}\label{subsec:bayle9}

Let us study the Enriques-Fano threefold described in\cite[\S 6.1.4]{Ba94} (see also \cite[Theorem 1.1 No.12]{Sa95}). 
We refer to \S~\ref{code:bayle9} of Appendix~\ref{app:code} for the computational techniques we will use.
Let us take two quadric hypersurfaces of $\mathbb{P}^{5}_{[x_0:x_1:x_2:y_3:y_4:y_5]}$,
$$Q_1 :=\{ s_1(x_0,x_1,x_2)+ r_1(y_3,y_4,y_5) = 0 \}, \,\, Q_2 :=\{ s_2(x_0,x_1,x_2)+ r_2(y_3,y_4,y_5) = 0\},$$
where $s_1$, $s_2$, $r_1$, $r_2$ are the following quadratic homogeneous forms:
$$s_1(x_0,x_1,x_2) := \sum_{i,j \in \{0,1,2\}}a_{i,j}x_ix_j,\quad s_2(x_0,x_1,x_2) := \sum_{i,j \in \{0,1,2\}}a_{i,j}'x_ix_j,$$
$$r_1(y_3,y_4,y_5) := \sum_{i,j \in \{3,4,5\}}b_{i,j}y_iy_j,\quad r_2(y_3,y_4,y_5) := \sum_{i,j \in \{3,4,5\}}b_{i,j}'y_iy_j.$$
Let us consider the smooth Fano threefold $X := Q_1 \cap Q_2$ and the involution $\sigma$ of $X$ defined by the restriction on $X$ of the following morphism 
$$\sigma' : \mathbb{P}^5  \to \mathbb{P}^5,\, \left[ x_{0} : x_1 : x_2 : y_3 : y_4 : y_5 \right] \mapsto] \left[ x_0 : x_1 : x_2 : -y_3 : -y_4 : -y_5 \right].$$
Then $\sigma :X \to X$ has eight fixed points $p_1$, $p_2$, $p_3$, $p_4$, $p_1'$, $p_2'$, $p_3'$, $p_4'$ such that
$$\{p_1, p_2, p_3, p_4\} = X \cap \{y_3=y_4=y_5=0\}, \,\, \{p_1', p_2', p_3', p_4'\} = X \cap \{x_0=x_1=x_2=0\}.$$
The quotient map $\pi : X \to X/\sigma = : W_{BS}^{9}$ is given by the restriction on $X$ of the 
morphism defined by the linear system of the $\sigma$-invariant quadric hypersurfaces of $\mathbb{P}^{5}$, that is the morphism
$\varphi : \mathbb{P}^{5} \to \mathbb{P}^{11}_{\left[Z_0 : \dots : Z_{11}\right]}$ such that
$$\begin{tikzcd}
\left[ x_{0} : x_1 : x_2 : y_3 : y_4 : y_5 \right] \arrow[d, mapsto]\\
 \left[x_0^2: x_1^2: x_2^2: x_0x_1: x_0x_2: x_1x_2: y_3^2: y_4^2: y_5^2: y_3y_4: y_3y_5: y_4y_5 \right].
\end{tikzcd}$$
By using Macaulay2, one can find that the image of $\mathbb{P}^{5}$ via $\varphi$ is a $5$-dimensional algebraic variety $F_5^{16}$ of degree $16$, whose ideal is generated by the following $12$ polynomials
$${Z}_{9} {Z}_{10}-{Z}_{6} {Z}_{11},\,\, {Z}_{7}{Z}_{10}-{Z}_{9} {Z}_{11},\,\, {Z}_{8} {Z}_{9}-{Z}_{10} {Z}_{11}, \,\, {Z}_{7}{Z}_{8}-{Z}_{11}^{2},\,\, {Z}_{6} {Z}_{8}-{Z}_{10}^{2},\,\, {Z}_{6}{Z}_{7}-{Z}_{9}^{2},$$
$${Z}_{3} {Z}_{4}-{Z}_{0} {Z}_{5}, \,\, {Z}_{1}{Z}_{4}-{Z}_{3} {Z}_{5},\,\, {Z}_{2} {Z}_{3}-{Z}_{4} {Z}_{5},\,\, {Z}_{1}{Z}_{2}-{Z}_{5}^{2},\,\, {Z}_{0} {Z}_{2}-{Z}_{4}^{2},\,\, {Z}_{0}{Z}_{1}-{Z}_{3}^{2}.$$
We observe that $W_{BS}^{9}= \varphi(X) = F_{5}^{16}\cap H_9$, where $H_9$ is the following $9$-dimensional projective subspace of $\mathbb{P}^{11}$
$$H_9:=\{a_{00}Z_{0}+a_{11}Z_{1}+a_{22}Z_{2}+(a_{01}+a_{10})Z_{3}+(a_{02}+a_{20})Z_{4}+(a_{12}+a_{21})Z_{5}+$$
$$+b_{33}Z_{6}+b_{44}Z_{7}+b_{55}Z_{8}+(b_{34}+b_{43})Z_{9}+(b_{35}+b_{53})Z_{10}+(b_{45}+b_{54})Z_{11} = 0,$$
$$a_{00}'Z_{0}+a_{11}'Z_{1}+a_{22}'Z_{2}+(a_{01}'+a_{10}')Z_{3}+(a_{02}'+a_{20}')Z_{4}+(a_{12}'+a_{21}')Z_{5}+$$
$$+b_{33}'Z_{6}+b_{44}'Z_{7}+b_{55}'Z_{8}+(b_{34}'+b_{43}')Z_{9}+(a_{35}'+b_{53}')Z_{10}+(b_{45}'+b_{54}')Z_{11} = 0\}.$$
Therefore we have $\pi = \varphi |_{X} : X \to W_{BS}^{9} = \varphi (X) \subset H_9 \cong \mathbb{P}^{9}$.
What follows has been proved for fixed values of $a_{ij}$, $b_{ij}$, $a_{ij}'$ and $b_{ij}'$, in order to simplify the computational analysis.

\subsection{Example}\label{ex:bayle9}
If
$(a_{ij})=\begin{pmatrix}
1 & 0 & 0 \\
0 & -3 & 0\\ 
0 & 0 & 2
\end{pmatrix} = (b_{ij}')$ and
$(b_{ij})=\begin{pmatrix}
3 & 0 & 0 \\
0 & -8 & 0\\ 
0 & 0 & 5 
\end{pmatrix} = (a_{ij}'),$
then we obtain
$$p_1 = \left[1:1:1:0:0:0\right], \, p_1' \left[0:0:0:1:1:1\right]$$
$$p_2 = \left[-1:1:1:0:0:0\right], \, p_2' = \left[0:0:0:-1:1:1\right]$$
$$p_3 = \left[1:-1:1:0:0:0\right], \, p_3' = \left[0:0:0:1:-1:1\right]$$
$$p_4 = \left[1:1:-1:0:0:0\right], \, p_4' = \left[0:0:0:1:1:-1\right].$$
Furthermore, we have
$$H_9=\{Z_{0}-3Z_{1}+2Z_{2}+3Z_{6}-8Z_{7}+5Z_{8}= 0,\, 3Z_{0}-8Z_{1}+5Z_{2}+Z_{6}-3Z_{7}+2Z_{8} = 0\}=$$
$$=\{{Z}_{1}-{Z}_{2}-8 {Z}_{6}+21 {Z}_{7}-13{Z}_{8}=0,\,
{Z}_{0}-{Z}_{2}-21 {Z}_{6}+55 {Z}_{7}-34 {Z}_{8}=0\},$$
which is the $\mathbb{P}^{9}_{\left[w_0:\dots : w_9 \right]}$ embedded in $\mathbb{P}^{11}_{\left[Z_0:\dots :Z_{11} \right]}$ via the morphism such that
$$Z_0 = w_0+21w_4-55w_5+34w_6,\,\, Z_1 = w_0+8w_4-21w_5+13w_6, \,\, Z_{i+2} = w_i,\, i=0,\dots 9.$$
By using Macaulay2, we find that the quotient map $\pi : X \to W_{BS}^{9} \subset H_9 \cong \mathbb{P}^{9}$ is given by the restriction on $X$ of the morphism $\varphi' : \mathbb{P}^{5} \to \mathbb{P}^{9}$ such that
$$\begin{tikzcd}
\left[ x_{0} : x_1 : x_2 : y_3 : y_4 : y_5 \right] \arrow[d,mapsto] \\
\left[x_2^2: x_0x_1: x_0x_2: x_1x_2: y_3^2: y_4^2: y_5^2: y_3y_4: y_3y_5: y_4y_5 \right].
\end{tikzcd}$$
In particular we obtain a BS-EF 3-fold $W_{BS}^{9}\subset \mathbb{P}^{9}$ whose ideal is generated by the following $12$ quadratic polynomials
\begin{center}
${w}_{7} {w}_{8}-{w}_{4} {w}_{9},\,\, {w}_{5} {w}_{8}-{w}_{7}{w}_{9},\,\, {w}_{6} {w}_{7}-{w}_{8} {w}_{9}, \,\, {w}_{5}{w}_{6}-{w}_{9}^{2},\,\, {w}_{4} {w}_{6}-{w}_{8}^{2},\,\, {w}_{4}{w}_{5}-{w}_{7}^{2},$

${w}_{2}^{2}-{w}_{3}^{2}-13 {w}_{0} {w}_{4}+34 {w}_{0}{w}_{5}-21 {w}_{0} {w}_{6},$

${w}_{1} {w}_{2}-{w}_{0} {w}_{3}-21 {w}_{3}{w}_{4}+55 {w}_{3} {w}_{5}-34 {w}_{3} {w}_{6},$

${w}_{0} {w}_{2}-{w}_{1}{w}_{3}+8 {w}_{2} {w}_{4}-21 {w}_{2} {w}_{5}+13 {w}_{2}{w}_{6},$

\begin{small}
${w}_{1}^{2}-{w}_{3}^{2}-21 {w}_{0} {w}_{4}-168 {w}_{4}^{2}+55{w}_{0} {w}_{5}-1155 {w}_{5}^{2}-34 {w}_{0} {w}_{6}-442 {w}_{6}^{2}+881{w}_{7}^{2}-545 {w}_{8}^{2}+1429 {w}_{9}^{2},$
\end{small}
${w}_{0} {w}_{1}-{w}_{2}{w}_{3},\quad {w}_{0}^{2}-{w}_{3}^{2}+8 {w}_{0} {w}_{4}-21 {w}_{0} {w}_{5}+13{w}_{0} {w}_{6}.$
\end{center}
This threefold $W_{BS}^9$ 
has singular points at $P_i:=\pi (p_i)$ and $P_i':=\pi (p_i')$, for $1\le i \le 4$, that is
$$P_1 = \left[1: 1: 1: 1: 0: 0: 0: 0: 0: 0 \right],\,
P_1' = \left[0: 0: 0: 0: 1: 1: 1: 1: 1: 1\right],$$
$$P_2 = \left[1: -1: -1: 1: 0: 0: 0: 0: 0: 0\right],\, P_2' = \left[0: 0: 0: 0: 1: 1: 1: -1: -1: 1\right],$$
$$P_3 = \left[1: -1: 1: -1: 0: 0: 0: 0: 0: 0\right],\, P_3' = \left[0: 0: 0: 0: 1: 1: 1: -1: 1: -1\right],$$
$$P_4 = \left[1: 1: -1: -1: 0: 0: 0: 0: 0: 0\right],\, P_4' = \left[0: 0: 0: 0: 1: 1: 1: 1: -1: -1\right].$$
Let $l_{i,j}$ be the line joining $P_i$ and $P_j$ for $i,j \in \{1,2,3,4,1',2',3',4'\}$ and $i \ne j$. It follows that $W_{BS}^{9}$ contains the lines
$l_{1,1'}$, $l_{1,2'}$, $l_{1,3'}$, $l_{1,4'}$, $l_{2,1'}$, $l_{2,3'}$, $l_{2,3'}$, $l_{2,4'}$, $l_{3,1'}$, $l_{3,2'}$, $l_{3,3'}$, $l_{3,4'}$, $l_{4,1'}$, $l_{4,3'}$, $l_{4,3'}$, $l_{4,4'},$
but it does not contain the others. So each one of the eight singular points of $W_{BS}^{9}$ is associated with $m=4$ of the other singular points, as in 
Table~\ref{tab:BS67913}
of Appendix~\ref{app:congifuration}. Thus, the singularities of the BS-EF 3-fold $W_{BS}^9$ have the same configuration as the ones of the F-EF 3-fold $W_{F}^{9}$.
We can say something more, namely that $W_{BS}^{9} = W_F^{9}$. Let us see how.
Let us project $\mathbb{P}^9$ from the $\mathbb{P}^5$ spanned by the singular points $P_2$, $P_3$, $P_4$, $P_2'$, $P_3'$, $P_4'$ of 
$W_{BS}^9$.
By using Macaulay2, we obtain the rational map $\rho : \mathbb{P}^{9} \dashrightarrow \mathbb{P}^{3}_{\left[z_0:\dots : z_3\right]}$ such that
$\left[w_0 : \dots : w_{9} \right] \mapsto \left[{w}_{0}+{w}_{1}+{w}_{2}+{w}_{3}: -{w}_{4}+{w}_{5}: -{w}_{4}+{w}_{6}: {w}_{4}+{w}_{7}+{w}_{8}+{w}_{9}\right].$
The restriction $\rho|_{ W_{BS}^{9}} : W_{BS}^{9} \dashrightarrow \mathbb{P}^{3}$ is a birational map (it can be verified through Macaulay2), whose inverse map
is the rational map
$\nu : \mathbb{P}^{3} \dashrightarrow  W_{BS}^{9} \subset \mathbb{P}^{9}$
defined by the linear system 
of the septic surfaces of $\mathbb{P}^{3}$ double along the six edges of the two trihedra
$$T :=\{(z_0  - 21z_1  + 13z_2 )z_0(z_0  - 55z_1  + 34z_2) = 0\},\, T':=\{ (z_2  + z_3)(z_1+z_3)z_3 = 0\},$$
and containing the lines given by the intersection of a face of $T$ and one of $T'$.


\section{The BS-EF 3-fold of genus 13}\label{subsec:bayle13}

Let us study the Enriques-Fano threefold described in \cite[\S 6.3.2]{Ba94} (see also \cite[Theorem 1.1 No.14]{Sa95}).
We refer to \S~\ref{code:bayle13} of Appendix~\ref{app:code} for the computational techniques we will use. 
Let us take the smooth Fano threefold $X := \mathbb{P}^{1}\times \mathbb{P}^{1}\times \mathbb{P}^{1}$ and the map $\sigma : X \to  X$ defined by
$$\left[x_0 : x_1\right] \times \left[y_0 : y_1\right] \times \left[z_0 : z_1\right] \mapsto \left[x_0 : -x_1\right] \times \left[y_0 : -y_1\right] \times \left[z_0 : -z_1\right].$$
Then $\sigma$ is an involution of $X$ having the following eight fixed points
$$p_1' = \left[0 : 1\right] \times \left[1 : 0\right] \times \left[1 : 0\right], \quad
p_1 = \left[1 : 0\right] \times \left[0 : 1\right] \times \left[0 : 1\right],$$
$$p_2' = \left[0 : 1\right] \times \left[0 : 1\right] \times \left[0 : 1\right], \quad 
p_2 = \left[1 : 0\right] \times \left[1 : 0\right] \times \left[1 : 0\right],$$ 
$$p_3 = \left[0 : 1\right] \times \left[1 : 0\right] \times \left[0 : 1\right], \quad
p_3' = \left[1 : 0\right] \times \left[0 : 1\right] \times \left[1 : 0\right],$$ 
$$p_4 = \left[0 : 1\right] \times \left[0 : 1\right] \times \left[1 : 0\right], \quad 
p_4' = \left[1 : 0\right] \times \left[1 : 0\right] \times \left[0 : 1\right].$$ 
The $\sigma$-invariant multihomogeneous polynomials of multidegree $(2,2,2)$
define the coordinates of the quotient map $\pi : X \to X / \sigma = : W_{BS}^{13} \subset \mathbb{P}^{13}$, i.e.
$$\pi : \left[x_0 : x_1\right] \times \left[y_0 : y_1\right] \times \left[z_0 : z_1\right] \mapsto \left[w_0 : \dots : w_{13} \right], \text{ where}$$
$$w_0 = x_0^2y_0^2z_0^2,\, w_1 = x_0^2y_0^2z_1^2,\, w_2 = x_0^2y_0y_1z_0z_1, \, w_3 = x_0^2y_1^2z_0^2,\, w_4 = x_0^2y_1^2z_1^2,$$ 
$$w_5 =  x_0x_1y_0^2z_0z_1,\, w_6 = x_0x_1y_0y_1z_0^2,\, w_7 = x_0x_1y_0y_1z_1^2,\, w_8 = x_0x_1y_1^2z_0z_1,$$
$$w_9 = x_1^2y_0^2z_0^2,\, w_{10} = x_1^2y_0^2z_1^2,\, w_{11} = x_1^2y_0y_1z_0z_1,\, w_{12} = x_1^2y_1^2z_0^2,\, w_{13} = x_1^2y_1^2z_1^2.$$
The use of Macaulay2 enables us to find that the BS-EF 3-fold $W_{BS}^{13}$ has ideal generated by the following $42$ quadratic polynomials 
\begin{center}
${w}_{10} {w}_{12}-{w}_{9} {w}_{13},\quad
{w}_{7}{w}_{12}-{w}_{6} {w}_{13},\quad
{w}_{4} {w}_{12}-{w}_{3} {w}_{13},\quad
{w}_{1}{w}_{12}-{w}_{0} {w}_{13},$\\
${w}_{11}^{2}-{w}_{9} {w}_{13},\quad
{w}_{8}{w}_{11}-{w}_{6} {w}_{13},\quad
{w}_{7} {w}_{11}-{w}_{5} {w}_{13},\quad
{w}_{6}{w}_{11}-{w}_{5} {w}_{12},$\\
${w}_{4} {w}_{11}-{w}_{2} {w}_{13},\quad
{w}_{3}{w}_{11}-{w}_{2} {w}_{12},\quad
{w}_{2} {w}_{11}-{w}_{0} {w}_{13},\quad
{w}_{8}{w}_{10}-{w}_{5} {w}_{13},$\\
${w}_{6} {w}_{10}-{w}_{5} {w}_{11},\quad
{w}_{4}{w}_{10}-{w}_{1} {w}_{13},\quad
{w}_{3} {w}_{10}-{w}_{0} {w}_{13},\quad
{w}_{2}{w}_{10}-{w}_{1} {w}_{11},$\\
${w}_{8} {w}_{9}-{w}_{5} {w}_{12},\quad
{w}_{7}{w}_{9}-{w}_{5} {w}_{11},\quad
{w}_{4} {w}_{9}-{w}_{0} {w}_{13},\quad
{w}_{3}{w}_{9}-{w}_{0} {w}_{12},$\\
${w}_{2} {w}_{9}-{w}_{0} {w}_{11},\quad
{w}_{1}{w}_{9}-{w}_{0} {w}_{10},\quad
{w}_{8}^{2}-{w}_{3} {w}_{13},\quad
{w}_{7}{w}_{8}-{w}_{2} {w}_{13},$\\
${w}_{6} {w}_{8}-{w}_{2} {w}_{12},\quad
{w}_{5}{w}_{8}-{w}_{0} {w}_{13},\quad
{w}_{7}^{2}-{w}_{1} {w}_{13},\quad
{w}_{6}{w}_{7}-{w}_{0} {w}_{13},$\\
${w}_{5} {w}_{7}-{w}_{1} {w}_{11},\quad
{w}_{3}{w}_{7}-{w}_{2} {w}_{8},\quad
{w}_{2} {w}_{7}-{w}_{1}{w}_{8},\quad
{w}_{6}^{2}-{w}_{0} {w}_{12},$\\
${w}_{5} {w}_{6}-{w}_{0}{w}_{11},\quad
{w}_{4} {w}_{6}-{w}_{2} {w}_{8},\quad
{w}_{2} {w}_{6}-{w}_{0}{w}_{8},\quad
{w}_{1} {w}_{6}-{w}_{0} {w}_{7},$\\
${w}_{5}^{2}-{w}_{0}{w}_{10},\quad
{w}_{4} {w}_{5}-{w}_{1} {w}_{8},\quad
{w}_{3} {w}_{5}-{w}_{0}{w}_{8},\quad
{w}_{2} {w}_{5}-{w}_{0} {w}_{7},$\\
${w}_{1} {w}_{3}-{w}_{0}{w}_{4},\quad
{w}_{2}^{2}-{w}_{0} {w}_{4}.$
\end{center}
The above threefold $W_{BS}^{13}$ has the following eight singular points
$$P_1 := \pi (p_1) = \{ w_i = 0 | i\ne 4\}, \quad P_1' := \pi (p_1') = \{w_i = 0 | i\ne 9\}$$
$$P_2 := \pi (p_2) = \{ w_i = 0 | i\ne 0\}, \quad P_2' := \pi (p_2') = \{ w_i = 0 | i\ne 13 \},$$
$$P_3 := \pi (p_3) = \{ w_i = 0 | i\ne 10\}, \quad P_3' := \pi (p_3') = \{ w_i = 0 | i\ne 3 \},$$ 
$$P_4 := \pi (p_4) = \{ w_i = 0 | i\ne 12\}, \quad  P_4' := \pi (p_4') = \{ w_i = 0 | i\ne 1\}.$$
Let $l_{i,j}$ be the line joining $P_i$ and $P_j$ with $i,j \in \{1,2,3,4,1',2',3',4'\}$ and $i \ne j$.
We see that $W_{BS}^{13}$ contains the lines
$l_{1,2'}$, $l_{1,3'}$, $l_{1,4'}$, $l_{2,1'}$, $l_{2,3'}$, $l_{2,4'}$, $l_{3,1'}$, $l_{3,2'}$, $l_{3,4'}$, $l_{4,1'}$, $l_{4,2'}$, $l_{4,3'},$ 
while it does not contain the others. So each one of the eight singular points of $W_{BS}^{13}$ is associated with $m=3$ of the other singular points, as in 
Table~\ref{tab:BS67913}
of Appendix~\ref{app:congifuration}. Thus, the singularities of the BS-EF 3-fold $W_{BS}^{13}$ have the same configuration as the ones of the F-EF 3-fold $W_{F}^{13}$.
We can say something more, namely that $W_{BS}^{13} = W_F^{13}$. Let us see how.
Let $\mathcal{S}$ be the linear system of the sextic surfaces of $\mathbb{P}^{3}$ having double points along the six edges of a fixed tetrahedron $T\subset \mathbb{P}^{3}$. 
Up to a change of coordinates, we can take the tetrahedron $T:=\{t_0t_1t_2t_3=0\}\subset \mathbb{P}^{3}_{\left[t_0:\dots : t_3\right]}$. Then $\mathcal{S}$ defines a rational map
$\nu_{\mathcal{S}} : \mathbb{P}^{3} \dashrightarrow \mathbb{P}^{13}$ given by $
\left[t_0 : t_1 : t_2 : t_3 \right] \mapsto  \left[w_0 : \dots : w_{13}\right],$
where
\begin{center}
$w_{0} = {t}_{0} {t}_{1}^{3} {t}_{2} {t}_{3},$
$w_{1} = {t}_{0}^{2} {t}_{1}^{2} {t}_{2}^{2},$
$w_{2} = {t}_{0}^{2} {t}_{1}^{2} {t}_{2} {t}_{3},$
$w_{3} = {t}_{0}^{2} {t}_{1}^{2} {t}_{3}^{2},$
$w_{4} = {t}_{0}^{3} {t}_{1} {t}_{2} {t}_{3},$

$w_{5} = {t}_{0} {t}_{1}^{2} {t}_{2}^{2} {t}_{3},$
$w_{6} = {t}_{0} {t}_{1}^{2} {t}_{2} {t}_{3}^{2},$
$w_{7} = {t}_{0}^{2} {t}_{1} {t}_{2}^{2} {t}_{3},$
$w_{8} = {t}_{0}^{2} {t}_{1} {t}_{2} {t}_{3}^{2},$
$w_{9} = {t}_{1}^{2} {t}_{2}^{2} {t}_{3}^{2},$

$w_{10} = {t}_{0} {t}_{1} {t}_{2}^{3} {t}_{3},$
$w_{11} = {t}_{0} {t}_{1} {t}_{2}^{2} {t}_{3}^{2},$
$w_{12} = {t}_{0} {t}_{1} {t}_{2} {t}_{3}^{3},$
$w_{13} = {t}_{0}^{2} {t}_{2}^{2} {t}_{3}^{2}$
\end{center}
(see \cite[p.635]{GH}). The above map is birational onto the image, which is the F-EF 3-fold $W_F^{13}\subset \mathbb{P}^{13}$ of genus $13$ (see \cite[\S 8]{Fa38}). Thanks to Macaulay2, we see that the threefold $W_F^{13}\subset \mathbb{P}^{13}$ coincides with the threefold $W_{BS}^{13}\subset \mathbb{P}^{13}$. 
Thus, we have the assertion of Theorem~\ref{thm:B13=F13} for $p=13$.
%


\section{The BS-EF 3-fold of genus 8}\label{subsec:bayle8}

Let us study the Enriques-Fano threefold described in \cite[\S 6.4.2]{Ba94}. Sano erroneously omits it (see \cite[p. 378]{Sa95}). We refer to \S~\ref{code:bayle8} of Appendix~\ref{app:code} for the computational techniques we will use. Let us take the hyperplane $\{x_4=0\}\subset \mathbb{P}^4_{\left[x_0:x_1:x_2:x_3:x_4\right]}$ and two quadric surfaces $Q,R\subset \{x_4=0\} \cong \mathbb{P}^3_{\left[x_0:x_1:x_2:x_3\right]}$, respectively with equations 
$$Q(x_0,x_1,x_2,x_3):=q_{00}x_0^2+q_{11}x_1^2+q_{22}x_2^2+q_{33}x_3^2+q_{01}x_0x_1+q_{23}x_2x_3=0,$$ 
$$R(x_0,x_1,x_2,x_3):=r_{00}x_0^2+r_{11}x_1^2+r_{22}x_2^2+r_{33}x_3^2+r_{01}x_0x_1+r_{23}x_2x_3=0.$$
Let $C:=Q\cap R$ be the elliptic quartic curve given by the complete intersection of the above quadrics, and let $Y\subset \mathbb{P}^4$ be the 
cone over $Q$ with vertex at the point $v:=\left[0:0:0:0:1\right]$. 
Let $bl : X \to Y$ be the 
blow-up of $Y$ at the point $v$ and along the curve $C$. We have that $X$ is a smooth Fano threefold. Let us explain this. Let us consider the blow-up of $\mathbb{P}^4$ at $v$ and along $C$, that is the map $bl' : \operatorname{Bl}_{v\cup C}\mathbb{P}^4 \to \mathbb{P}^4$ with exceptional divisors $E_v := bl'^{-1}(v)$ and $E_C := bl'^{-1}(C)$. By definition, we have that $X$ is the strict transform of $Y$ on $\operatorname{Bl}_{v\cup C}\mathbb{P}^4$ and that $bl=bl'|_{X}$. If $H$ denotes the pullback of the hyperplane class $h$ of $\mathbb{P}^4$, we have that $X\sim 2H-2E_v-E_C$. By the adjunction formula, we have that $-K_X = -(K_{\operatorname{Bl}_{v\cup C}\mathbb{P}^4}+X)|_{X} \sim (3H-E_v-E_C)|_{X}$. We want to show that $-K_X$ is ample. Let $\mathcal{C}$ be the linear system of the cubic hypersurfaces of $\mathbb{P}^4$ containing the curve $C$ and passing through the point $v$, and let us fix a general hyperplane $h_v \subset \mathbb{P}^4$ passing through $v$. Thus, $\mathcal{C}$ contains a sublinear system $\overline{\mathcal{C}} \subset \mathcal{C}$ whose fixed part is given by $h_v\cup \{x_4=0\}$. 
Since the movable part of $\overline{\mathcal{C}}$ is given by the hyperplanes of $\mathbb{P}^4$, then we obtain the ampleness of $\mathcal{C}$ at least outside $v\cup C$.
So we have the ampleness of $-K_X$ at least outside $E_v\cap X$ and $E_C\cap X$, since $|-K_X|$ coincides with the restriction on $X$ of the strict transform of $\mathcal{C}$.
Furthermore, the movable part of $\overline{\mathcal{C}}$
also contains the hyperplanes of $\mathbb{P}^4$ through $v$, whose strict transforms are very ample on $E_v$: indeed, we have $|\mathcal{O}_{E_v}(H-E_v)|=|\mathcal{O}_{E_v}(-E_v^2)|\cong|\mathcal{O}_{\mathbb{P}^3}(1)|$ (see \cite[Chap 4, \S 6]{GH}).
Therefore, the ampleness of $-K_X$ along $E_v\cap X$ follows by the fact that $E_v\cap X$ is a smooth quadric surface in $E_v\cong \mathbb{P}^3$.
It remains to show the ampleness of $-K_X$ along $S':=E_C\cap X$, which is a $\mathbb{P}^{1}$-bundle 
over $C$, identified with the projectification $\mathbb{P}(\mathcal{N}_{C|Y})$ of the normal bundle of $C$ in $Y$ (see \cite[Chap 4, \S 6]{GH}).
Since $C$ is the complete intersection of a hyperplane section and of a quadric section of $Y$, then 
$S' = \mathbb{P}(\mathcal{N}_{C|Y})\cong \mathbb{P}(\mathcal{O}_{C}(h)\oplus \mathcal{O}_{C}(2h))$ (see \cite[Example 10.2]{CS89}).
In particular we have that the class $S'|_{S'}$ is the class of the tautological bundle on $S'$ (see \cite[Chap 4, \S 6]{GH}). Thus, $-E_C|_{S'}=-S'|_{S'}$ is ample on $S'$, and so $(-K_X)|_{S'}=(3H-E_C)|_{S'}$ is ample too.

Let $\sigma : X \to X$ be now the morphism defined by the birational map $\sigma' : Y \dashrightarrow Y$ given by
$\left[x_0 : x_1 : x_2:x_3:x_4\right] \mapsto \left[x_4x_0 : x_4x_1 : -x_4x_2: -x_4x_3 : R(x_0,x_1,x_2,x_3)\right].$
The map $\sigma$ is an involution of $X$ with eight fixed points, which are the preimages via $bl : X \to Y$ of the eight points $p_1,p_2,p_3,p_4,p_1',p_2',p_3',p_4'\in Y$ such that
$$\{p_1,p_1',p_2,p_2'\}=Y\cap\{x_2=0,\,x_3=0,\,x_4^2-R(x_0,x_1,x_2,x_3)=0\},$$ 
$$\{p_3,p_3',p_4,p_4'\}=Y\cap\{x_0=0,\,x_1=0,\,x_4^2+R(x_0,x_1,x_2,x_3)=0\}.$$
The $\sigma '$-invariant elements of $\mathcal{C}$
define the rational map $\varphi: Y \dashrightarrow \mathbb{P}^9$
given by $\left[x_0:\dots :x_4\right] \mapsto \left[Z_0:\dots :Z_9\right]$, where
\begin{center}
$Z_0 = x_4^2x_0+x_0R(x_0,x_1,x_2,x_3),\quad Z_1 = x_4^2x_1+x_1R(x_0,x_1,x_2,x_3),$\\
$Z_2 = x_4^2x_2-x_2R(x_0,x_1,x_2,x_3),\quad Z_3 = x_4^2x_3-x_3R(x_0,x_1,x_2,x_3),$\\ 
$Z_4 = x_4x_0^2,\,\,\, Z_5 = x_4x_1^2,\,\,\, Z_6 = x_4x_2^2,\,\,\, Z_7 = x_4x_3^2,\quad Z_8 = x_4x_0x_1,\,\,\, Z_9 =x_4x_2x_3.$
\end{center}
Let us observe that $\varphi(Y)$ is contained in the hyperplane
$$H_8 := \{q_{00}Z_4+q_{11}Z_5+q_{22}Z_6+q_{33}Z_7+q_{01}Z_8+q_{23}Z_9=0\}\cong \mathbb{P}^8 \subset \mathbb{P}^9.$$
Therefore the rational map $\varphi$ defines the quotient map $\pi : X \to X / \sigma = : W_{BS}^{8}$, thanks to the following commutative diagram
$$\begin{tikzcd}
X \arrow[d, "bl"] \arrow[dr, "\pi"] & \\
Y  \arrow[r, dashrightarrow, "\varphi"] & \varphi(Y) = \pi(X) = W_{BS}^8\subset H_8 \cong \mathbb{P}^{8}.
\end{tikzcd}$$
What follows has been proved for fixed values of $q_{ij}$ and $r_{ij}$, in order to simplify the computational analysis.

\subsection{Example}\label{ex: ideal bayle8}
Let us take
$$Q(x_0,x_1,x_2,x_3) = x_0^2-x_1^2 -x_2^2+x_3^2 \quad \text{and} \quad
R(x_0,x_1,x_2,x_3) = 2x_0^2-x_1^2  -3x_2^2+2x_3^2.$$ 
Then $\varphi(Y)$ is contained in the hyperplane
$H_8=\{Z_4-Z_5-Z_6+Z_7\}$, which we can see as the image of the morphism $i : \mathbb{P}^8 \hookrightarrow \mathbb{P}^9$ such that
$$\begin{tikzcd}
\left[w_0 : \dots : w_8 \right] \arrow[r, mapsto, "i"] & \left[w_0:w_1:w_2:w_3:w_4+w_5-w_6:w_4:w_5:w_6:w_7:w_8\right].
\end{tikzcd}$$
Thanks to Macaulay2, one can verify that we obtain a BS-EF 3-fold $W_{BS}^{8} \subset H_8 \cong \mathbb{P}^8$ whose ideal is generated by the following $11$ polynomials of degree $2$ or $3$
\begin{center}
${w}_{5} {w}_{6}-{w}_{8}^{2},\quad 
{w}_{2} {w}_{6}-{w}_{3}{w}_{8}, \quad
{w}_{3} {w}_{5}-{w}_{2} {w}_{8},\quad
{w}_{4}^{2}+{w}_{4}{w}_{5}-{w}_{4} {w}_{6}-{w}_{7}^{2},$

${w}_{1} {w}_{4}+{w}_{1}{w}_{5}-{w}_{1} {w}_{6}-{w}_{0} {w}_{7},\quad
{w}_{0} {w}_{4}-{w}_{1}{w}_{7},$

${w}_{0}^{2}-{w}_{1}^{2}-{w}_{2}^{2}+{w}_{3}^{2}-4 {w}_{4}{w}_{5}+4 {w}_{5}^{2}+4 {w}_{4} {w}_{6}-4 {w}_{8}^{2},$

${w}_{2} {w}_{3}{w}_{7}-{w}_{0} {w}_{1} {w}_{8}+4 {w}_{4} {w}_{7} {w}_{8}-4 {w}_{5}{w}_{7} {w}_{8},$

${w}_{0} {w}_{1} {w}_{6}-{w}_{3}^{2} {w}_{7}-4 {w}_{4}{w}_{6} {w}_{7}+4 {w}_{7} {w}_{8}^{2},$

${w}_{3}^{2} {w}_{4}-{w}_{1}^{2}{w}_{6}+4 {w}_{4} {w}_{6}^{2}+4 {w}_{6} {w}_{7}^{2}-8 {w}_{4}{w}_{8}^{2},$

${w}_{2} {w}_{3} {w}_{4}-{w}_{1}^{2} {w}_{8}-8 {w}_{4} {w}_{5}{w}_{8}+4 {w}_{4} {w}_{6} {w}_{8}+4 {w}_{7}^{2} {w}_{8}.$
\end{center}
This threefold $W_{BS}^{8}$ 
has the following eight singular points:
$$P_1 := i^{-1}(\varphi (p_1)) = \varphi (\left[1 : 1: 0:0:1\right]) = 
\left[2:2:0:0:1:0:0:1:0\right],$$
$$P_2 := i^{-1}(\varphi (p_2)) = \varphi (\left[-1 : -1: 0:0:1\right]) =
\left[-2:-2:0:0:1:0:0:1:0\right],$$
$$P_3 := i^{-1}(\varphi (p_3)) = \varphi (\left[0 : 0: 1:1:1\right]) =
\left[0:0:2:2:0:1:1:0:1\right],$$
$$P_4 := i^{-1}(\varphi (p_4)) = \varphi (\left[0 : 0: -1:1:1\right]) =
\left[0:0:-2:-2:0:1:1:0:1\right],$$
$$P_1' := i^{-1}(\varphi (p_1')) = \varphi (\left[-1 : 1: 0:0:1\right]) =
\left[2:-2:0:0:-1:0:0:1:0\right],$$
$$P_2' := i^{-1}(\varphi (p_2')) = \varphi (\left[1 : -1: 0:0:1\right]) =
\left[-2:2:0:0:-1:0:0:1:0\right],$$
$$P_3' := i^{-1}(\varphi (p_3')) = \varphi (\left[0 : 0: -1:-1:1\right]) =
\left[0:0:2:-2:0:-1:-1:0:1\right],$$
$$P_4' := i^{-1}(\varphi (p_4')) = \varphi (\left[0 : 0: 1:-1:1\right]) =
\left[0:0:-2:2:0:-1:-1:0:1\right].$$
Let $l_{i,j}$ be the line joining $P_i$ and $P_j$ for $i,j \in \{1,2,3,4,1',2',3',4'\}$ and $i \ne j$. We have that $W_{BS}^{8}$ does not contain the lines
$l_{1,1'}$, $l_{1,2'}$, $l_{2,1'}$, $l_{2,2'}$, $l_{3,3'}$, $l_{3,4'}$, $l_{4,3'}$, $l_{4,4'},$ 
while it contains the others. So each one of the eight singular points of $W_{BS}^{8}$ is associated with $m=5$ of the other singular points, as in 
Table~\ref{tab:BS810} of Appendix~\ref{app:congifuration}. Hence there exist three mutually associated points (for example $P_1$, $P_2$ and $P_3$). This case had been excluded by Fano for $p>7$ (see \cite[\S 5]{Fa38}). So this suggests that in Fano's paper there are other gaps to be discovered.

\begin{theorem}\label{thm:B8 linear system}
Let $T$ be a trihedron with edges $l_0,l_1,l_2$ and vertex $v$ as in Figure~\ref{fig:baseL(N)-bayle8}. Let us choose a general point $q_1 \in l_1$, a general point $q_2\in l_2$, three distinct points $a_r,\,a_s,\,a_t\in l_0$, a general point $b_1\in r_1 :=\left\langle q_1, a_r\right\rangle$ and a general point $b_2\in r_2 :=\left\langle q_2, a_r\right\rangle$. Let us take a general conic $C$ through the points $q_1$, $q_2$, $b_1$, $b_2$, in the plane spanned by the three points $a_r$, $q_1$, $q_2$. Finally let us consider the lines
$s_1 :=\left\langle q_1, a_s\right\rangle$, $s_2 :=\left\langle q_2, a_s\right\rangle$, $t_1 :=\left\langle b_1, a_t\right\rangle$, $t_2 :=\left\langle b_2, a_t\right\rangle$ and the lines $l_1' :=\left\langle q_1', q_2\right\rangle$ and $l_2' :=\left\langle q_2', q_1\right\rangle$, where $q_1'$ is a general point on $t_1$ and $q_2'$ a general point on $t_2$. Then the BS-EF 3-fold $W_{BS}^{8}$ can be obtained as the image of $\mathbb{P}^3$ via the rational map 
$\nu_{\mathcal{N}} : \mathbb{P}^3 \dashrightarrow \mathbb{P}^8$ defined by the linear system $\mathcal{N}$ of the septic surfaces of $\mathbb{P}^3$ which are quadruple at the points $q_1$ and $q_2$, triple at the vertex $v$, and double along the lines $l_0$, $l_1$, $l_2$, $l_1'$, $l_2'$, along the conic $C$ and at the points $c_1:=t_1\cap s_1$ and $c_2:=t_2 \cap s_2$. Furthermore, a general element of $\mathcal{N}$ contains the lines $t_1$, $t_2$, $r_1$, $r_2$, $s_1$, $s_2$ and $e_0 := \left\langle q_1,q_2\right\rangle$.
\begin{figure}[ht]
\centering
\includegraphics[scale=0.4]{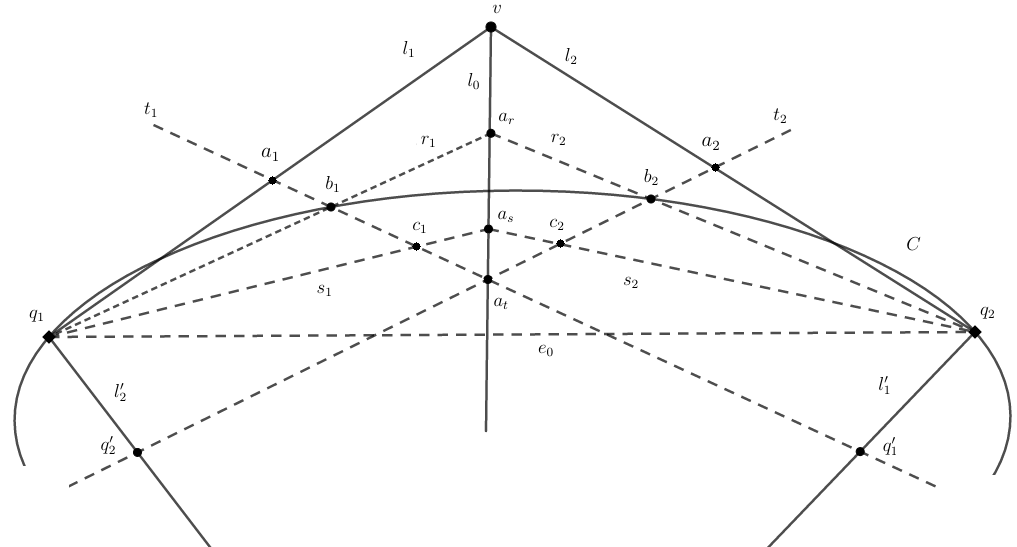}
\caption{\label{fig:baseL(N)-bayle8}\scriptsize{Base locus of the linear system $\mathcal{N}$.}}
\end{figure}
\end{theorem}
\begin{proof}
Let us project $\mathbb{P}^8$ from the $\mathbb{P}^4$ spanned by the singular points $P_1$, $P_1'$, $P_2$, $P_3$ and $P_3'$ of the BS-EF 3-fold $W_{BS}^8$ of Example~\ref{ex: ideal bayle8}. 
By using Macaulay2, we obtain the rational map 
$\rho: \mathbb{P}^{8} \dashrightarrow \mathbb{P}^{3}_{\left[z_0:\dots :z_3 \right]}$ such that
$$\left[w_0 : \dots : w_{13} \right] \mapsto \left[w_2-2w_8: w_5-w_6: w_3-2w_6: w_0-w_1+2w_4-2w_7 \right].$$
One can verify, via Macaulay2, that the restriction $\rho|_{ W_{BS}^{8}} : W_{BS}^{8} \dashrightarrow \mathbb{P}^{3}$ is birational and that its inverse map is given by the rational map
$\nu_{\mathcal{N}} : \mathbb{P}^{3} \dashrightarrow  W_{BS}^{8} \subset \mathbb{P}^{8}$ defined by the linear system 
$\mathcal{N}$ 
of the septic surfaces
\begin{itemize}
\item[(i)] quadruple at $q_1=\left[1:0:-2:0\right]$ and $q_2=\left[1:0:2:0\right]$;
\item[(ii)] triple at the vertex $v=\left[0:0:0:1\right]$ of the following trihedron 
$$T=\{z_1(2z_0+z_2)(2z_0-z_2)=0\};$$
\item[(iii)] double at the points
$c_1=\left[1:-2:-2:0\right]$ and
$c_2=\left[1:2:2:0\right]$; double along the line
$l_1' = \{z_3=2z_0+2z_1-z_2=0\}\ni
q_1'=\left[1:-2:-2:0\right]$
and along the line $l_2'=\{z_3=2z_0-2z_1+z_2=0\}\ni q_2'=\left[1:2:2:0\right]$;
double along the edges $l_0=\{z_0=z_2=0\}$, $l_1=\{z_1=2z_0+z_2=0\}$, $l_2=\{z_1=2z_0-z_2=0\}$ of the trihedron $T$; double along the conic
$C=\{2z_1+z_3=4z_0^2-z_2^2-2z_2z_3-2z_3^2=0\}$ passing through $q_1$, $q_2$,
$b_1=\left[1:-1:-2:2\right]$ and
$b_2=\left[1:1:2:-2\right]$;
\item[(iv)] containing the lines
$r_1=\{2z_1+z_3=2z_0+z_2=0\}$,
$r_2=\{2z_1+z_3=2z_0-z_2=0\}$,
$s_1=\{z_3=2z_0+z_2=0\}$,
$s_2=\{z_3=2z_0-z_2=0\}$,
$t_1=\{2z_1-2z_2-z_3=2z_0+z_2=0\}$,
$t_2=\{2z_1-2z_2-z_3=2z_0-z_2=0\}$.
\end{itemize}

\end{proof}


\section{The BS-EF 3-fold of genus 10}\label{subsec:bayle10}

Let us study the Enriques-Fano threefold described in \cite[\S 6.5.1]{Ba94} (see also \cite[Theorem 1.1 No.13]{Sa95}). We refer to \S~\ref{code:bayle10} of Appendix~\ref{app:code} for the computational techniques we will use. Let us take the smooth Fano threefold $X := \mathbb{P}^{1}\times S_6$, where 
$S_6$ is a smooth sextic Del Pezzo surface. We recall that $S_6$ is the image of $\mathbb{P}^2$ via the rational map $\lambda :  \mathbb{P}^2 \dashrightarrow \mathbb{P}^{6}$ defined by the linear system of the plane cubic curves passing through three fixed points $a_1$, $a_2$, $a_3$ in general position. Up to a change of coordinates, we may assume 
$a_1=[1:0:0]$, $a_2=[0:1:0]$, $a_3=[0:0:1]$ and 
$$\lambda : \left[u_0 : u_1 : u_2\right] \mapsto \left[ u_1^2 u_2 : u_1 u_2^2 : u_0^2 u_2 : u_0 u_2^2 : u_0^{2} u_1 : u_0 u_1^2 : u_0 u_1 u_2 \right].$$
Thanks to Macaulay2, 
we can say that $S_6=\lambda\left(\mathbb{P}^{2}\right)\subset \mathbb{P}^{6}_{[x_0:x_1:x_2:x_3:x_4:x_5:x_6]}$ has ideal generated by the following polynomials 
$$x_3x_5-x_6^2,\quad x_2x_5-x_4x_6,\quad x_1x_5-x_0x_6,\quad x_3x_4-x_2x_6,\quad x_1x_4-x_6^2,$$
$$x_0x_4-x_5x_6,\quad x_0x_3-x_1x_6,\quad x_0x_2-x_3x_6,\quad x_0x_2-x_6^2.$$
The quadratic transformation $q_{a_1,a_2,a_3} : \mathbb{P}^2 \dashrightarrow \mathbb{P}^2$, given by the linear system of the conics passing through $a_1$, $a_2$ and $a_3$, defines an involution of the above sextic Del Pezzo surface.
Indeed, we have
\begin{footnotesize}
$$\begin{tikzcd}
\left[u_0 : u_1 : u_2\right] \arrow[r, mapsto, "q_{a_1,a_2,a_3}"] \arrow[d, mapsto, "\lambda"] & \left[\frac{1}{u_0} : \frac{1}{u_1} : \frac{1}{u_2}\right] \arrow[dd, mapsto, "\lambda"]\\
\hspace{-1cm}\left[u_1^2 u_2 : u_1 u_2^2 : u_0^2 u_2 : u_0 u_2^2 : u_0^{2} u_1 : u_0 u_1^2 : u_0 u_1 u_2\right]  & \\
& \hspace{-0.7cm}\left[u_0^2 u_2 : u_1 u_0^2 : u_1^2 u_2 : u_0 u_1^2 : u_1 u_2^2 : u_0 u_2^2 : u_0 u_1 u_2\right],
\end{tikzcd}$$
\end{footnotesize}
and then we obtain the involution $t'$ of $\mathbb{P}^{6}$ given by
$$\begin{tikzcd}
\hspace{-1cm}\left[x_0 : x_1 : x_2 : x_3 : x_4 : x_5 : x_6\right] \arrow[r, mapsto, "t' "] & \left[x_2 : x_4 : x_0 : x_5 : x_1 : x_3 : x_6\right].
\end{tikzcd}$$
The locus of $t'$-fixed points of $\mathbb{P}^{6}$ consists of two projective subspaces
$$F_1 := \{x_0+x_2 = x_1+x_4 = x_3+x_5 = x_6 = 0\}\cong \mathbb{P}^{2},$$
$$F_{2} := \{x_0-x_2 = x_1-x_4 = x_3-x_5= 0\} \cong \mathbb{P}^{3}.$$
In particular we have $F_1\cap S_6 = \emptyset$ and $F_{2}\cap S_6 = \{d_1,d_2,d_3,d_4\}$, where
$$d_1 := \left[1:1:1:1:1:1:1\right], d_2 :=\left[1:-1:1:-1:-1:-1:1\right],$$
$$d_3 := \left[-1:1:-1:-1:1:-1:1\right], d_4 := \left[-1:-1:-1:1:-1:1:1\right].$$
Thus, $\sigma_2 := t' |_{S_6}$ is an involution of $S_6$ with four fixed points.  
We also consider the involution of $\mathbb{P}^1$ with two fixed points $\left[0:1\right]$ and $\left[1:0\right]$, that is the map $\sigma_1 : \mathbb{P}^{1} \to \mathbb{P}^{1}$ given by
$\left[y_0 : y_1\right] \mapsto \left[y_0 : -y_1 \right]$. 
Therefore the map $\sigma:= (\sigma_1 \times \sigma_2) : X \to X$ is an involution of $X$ having eight fixed points $p_1$, $p_2$, $p_3$, $p_4$, $p_1'$, $p_2'$, $p_3'$, $p_4'$, where 
$$p_i := \left[0 : 1\right] \times d_i, \quad
p_i' = \left[1 : 0\right] \times d_i,\quad i=1,2,3,4.$$ 
The quotient map $\pi : X \to X / \sigma = : W_{BS}^{10}$ is given by the restriction on $X$ of the morphism $\varphi : \mathbb{P}^{1}\times \mathbb{P}^6 \to \mathbb{P}^{10}_{\left[w_0 : \dots : w_{10}\right]}$ defined by the $(\sigma_{1} \times t')$-invariant multihomogeneous polynomials of multidegree $(2,1)$, i.e.
$\varphi : \left[y_0 : y_1\right] \times \left[x_0 : \dots : x_6\right] \mapsto \left[w_0 : \dots : w_{10}\right]$
where
$$w_0 = y_0^2x_6,\, w_1 = y_0^2(x_0+x_2),\, w_2 = y_0^2(x_1+x_4),\, w_3 = y_0^2(x_3+x_5),$$
$$w_4 = y_1^2x_6,\,w_5 = y_1^2(x_0+x_2),\, w_6 = y_1^2(x_1+x_4),\, w_7 = y_1^2(x_3+x_5),$$ 
$$w_8 = y_0y_1(x_0-x_2),\, w_9 =y_0y_1(x_1-x_4),\, w_{10} = y_0y_1(x_3-x_5).$$
Thanks to Macaulay2, one can find that the BS-EF 3-fold $W_{BS}^{10}$ has ideal generated by the following $20$ polynomials of degree $2$ or $3$
\begin{center}
${w}_{7} {w}_{8}-2 {w}_{4} {w}_{9}+{w}_{5}{w}_{10},\quad
{w}_{6} {w}_{8}-{w}_{5} {w}_{9}+2 {w}_{4} {w}_{10},\quad
2{w}_{4}{w}_{8}-{w}_{7} {w}_{9}+{w}_{6} {w}_{10},$

${w}_{3} {w}_{8}-2 {w}_{0}{w}_{9}+{w}_{1} {w}_{10},\quad
{w}_{2} {w}_{8}-{w}_{1} {w}_{9}+2 {w}_{0}{w}_{10},\quad
2 {w}_{0} {w}_{8}-{w}_{3} {w}_{9}+{w}_{2} {w}_{10},$

${w}_{3}{w}_{6}-{w}_{2} {w}_{7},\quad
{w}_{2} {w}_{6}-{w}_{3}{w}_{7}-{w}_{9}^{2}+{w}_{10}^{2},\quad
{w}_{1} {w}_{6}-2 {w}_{0}{w}_{7}-{w}_{8} {w}_{9},$

$2 {w}_{0} {w}_{6}-{w}_{1} {w}_{7}-{w}_{8}{w}_{10},\quad
{w}_{3} {w}_{5}-{w}_{1} {w}_{7},\quad
{w}_{2} {w}_{5}-2 {w}_{0}{w}_{7}-{w}_{8} {w}_{9},$

${w}_{1} {w}_{5}-{w}_{3}{w}_{7}-{w}_{8}^{2}+{w}_{10}^{2},\quad
2 {w}_{0} {w}_{5}-{w}_{2}{w}_{7}+{w}_{9} {w}_{10},\quad
{w}_{3} {w}_{4}-{w}_{0} {w}_{7},$

$2 {w}_{2}{w}_{4}-{w}_{1} {w}_{7}-{w}_{8} {w}_{10},\quad
2 {w}_{1} {w}_{4}-{w}_{2}{w}_{7}+{w}_{9} {w}_{10},\quad
4 {w}_{0} {w}_{4}-{w}_{3} {w}_{7}+{w}_{10}^{2},$

$4{w}_{4}^{3}-{w}_{4} {w}_{5}^{2}-{w}_{4} {w}_{6}^{2}+{w}_{5} {w}_{6}{w}_{7}-{w}_{4} {w}_{7}^{2},\quad 
4 {w}_{0}^{3}-{w}_{0} {w}_{1}^{2}-{w}_{0}{w}_{2}^{2}+{w}_{1} {w}_{2} {w}_{3}-{w}_{0} {w}_{3}^{2}$.
\end{center}
The eight singular points of $W_{BS}^{10}$ are $P_i:=\pi (p_i)$ and $P_i':=\pi (p_i')$, for $1\le i \le 4$, that is
$$P_1 = \left[0:0:0:0:1:2:2:2:0:0:0\right],\, 
P_1' = \left[1:2:2:2:0:\dots:0\right]$$
$$P_2 = \left[0:0:0:0:1:2:-2:-2:0:0:0\right],\,
P_2' = \left[1:2:-2:-2:0:\dots :0\right],$$
$$P_3 = \left[0:0:0:0:1:-2:2:-2:0:0:0\right],\,
P_3' = \left[1:-2:2:-2:0:\dots :0\right],$$
$$P_4 = \left[0:0:0:0:1:-2:-2:2:0:0:0\right],\,
P_4' = \left[1:-2:-2:2:0:\dots :0\right].$$
Let $l_{i,j}$ be the line joining $P_i$ and $P_j$ for $i,j \in \{1,2,3,4,1',2',3',4'\}$ and $i \ne j$. 
We have that $W_{BS}^{10}$ contains the lines
$l_{1,2}$, $l_{1,3}$, $l_{1,4}$, $l_{1,1'}$, $l_{2,3}$, $l_{2,4}$, $l_{2,2'}$, $l_{3,4}$, $l_{3,3'}$, $l_{4,4'}$, $l_{1',2'}$, $l_{1',3'}$, $l_{1',4'}$, $l_{2',3'}$, $l_{2',4'}$, $l_{3',4'}$, 
while it does not contain the others. So each one of the eight singular points of $W_{BS}^{10}$ is associated with $m=4$ of the other singular points, as in 
Table~\ref{tab:BS810} of Appendix~\ref{app:congifuration}. Hence there exist three mutually associated points (for example $P_1$, $P_2$ and $P_3$). This case had been excluded by Fano for $p>7$ (see \cite[\S 5]{Fa38}). So this suggests that in Fano's paper there are other gaps to be discovered.

\begin{theorem}\label{thm:B10linearsystem}
Let $T\subset \mathbb{P}^3$ be a tetrahedron with faces $f_i$ and edges $l_{ij}:=f_i\cap f_j$ for $0\le i<j\le 3$. Let $v_i$ be the vertex opposite to the face $f_i$, for $0\le i\le 3$. 
Let $\pi$ be a plane through the vertex $v_0$, which intersects the face $f_i$ along a line $r_i$, for $1\le i\le 3$, and let us define the point $q_i:=r_i\cap l_{0i}$ (see Figure~\ref{fig:baseL(M)-bayle10}). Then $W_{BS}^{10}$ can be obtained as the image of $\mathbb{P}^3$ via the rational map $\nu_{\mathcal{M}} : \mathbb{P}^3 \dashrightarrow \mathbb{P}^{10}$ defined by the linear system $\mathcal{M}$ of the sextic surfaces of $\mathbb{P}^3$ which are quadruple at the vertex $v_0$, triple at the other three vertices $v_1$, $v_2$, $v_3$, and double along the three lines $r_1$, $r_2$, $r_3$. Furthemore a general element of $\mathcal{M}$ also contains the six edges of $T$.
\begin{figure}[ht]
\centering
\includegraphics[scale=0.5]{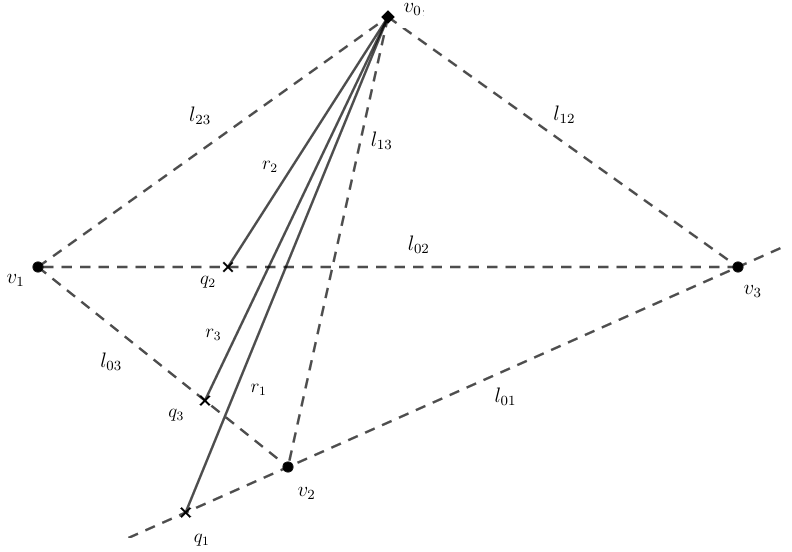}
\caption{\label{fig:baseL(M)-bayle10}\scriptsize{Base locus of the linear system $\mathcal{M}$.}}
\end{figure}
\end{theorem}
\begin{proof}
Let us project $\mathbb{P}^{10}$ from the $\mathbb{P}^6$ spanned by the singular points $P_1$, $P_2$, $P_3$, $P_4$, $P_1'$, $P_2'$ and $P_3'$ of $W_{BS}^{10}$. 
By using Macaulay2, we obtain the rational map
$$\rho: \mathbb{P}^{10} \dashrightarrow \mathbb{P}^{3}_{\left[z_0:\dots :z_3 \right]}, \quad \left[w_0 : \dots : w_{13} \right] \mapsto \left[-2w_0+w_1+w_2-w_3: w_8: w_9: w_{10} \right].$$
Thanks to Macaulay2, we see that the restriction $\rho|_{ W_{BS}^{10}} : W_{BS}^{10} \dashrightarrow \mathbb{P}^{3}$ is birational. 
We 
also compute its inverse map, which is 
the rational map
$\nu_{\mathcal{M}} : \mathbb{P}^{3} \dashrightarrow  W_{BS}^{10} \subset \mathbb{P}^{10}$
defined by the linear system 
$\mathcal{M}$ 
of the sextic surfaces
\begin{itemize}
\item[(i)] containing the six edges
$l_{23} =\{z_1=z_2-z_3=0\}$, $l_{13} =\{z_3=z_1+z_2=0\}$, $l_{12} =\{z_2=z_1+z_3=0\}$,
$l_{01} =\{z_0=z_1+z_2+z_3=0\}$, $l_{03} =\{z_0=z_1-z_2+z_3=0\}$ and $l_{02} =\{z_0=z_1+z_2-z_3=0\}$
of the tetrahedron $T$
with faces
$f_{0} =\{z_0=0\}$, $f_{1} =\{z_1+z_2+z_3=0\}$, $f_2=\{z_1-z_2+z_3=0\}$ and $f_{3} =\{z_1+z_2-z_3=0\}$;
\item[(ii)] double along the lines 
$r_1 =\{z_1=z_2+z_3=0\}$, $r_2 =\{z_3=z_1-z_2=0\}$ and $r_3 =\{z_2=z_1-z_3=0\}$ contained in the plane $\pi =\{z_1-z_2-z_3=0\},$
and obviously double at the points $q_1=\left[0:0:-1:1\right]$, $q_2=\left[0:1:1:0\right]$, $q_3=\left[0:1:0:1\right];$
\item[(iii)] triple at the following vertices of $T$
$$v_1=\left[0:0:1:1\right], \, v_2=\left[0:1:-1:0\right], \, v_3=\left[0:1:0:-1\right];$$
\item[(iv)] and quadruple at the vertex
$v_0=\left[1:0:0:0\right].$
\end{itemize}
\end{proof}


\section{The P-EF 3-fold of genus 17}\label{subsec:pro17}

Let us study the Enriques-Fano threefold described in \cite[\S 3]{Pro07}. We refer to \S~\ref{code:pro17} of Appendix~\ref{app:code} for the computational techniques we will use.
Let $P$ be the octic Del Pezzo surface given by the image of the anticanonical embedding of $\mathbb{P}^{1}\times \mathbb{P}^{1}$ in $\mathbb{P}^{8}$, which is defined by the linear system of the divisors of bidegree $(2,2)$, i.e.
$\left[u_0 : u_1\right] \times \left[v_0 : v_1\right]  \mapsto  \left[y_{0,0}:y_{0,1}:y_{0,2},y_{1,0}:y_{1,1}:y_{1,2}:y_{2,0}:y_{2,1}:y_{2,2}\right],$
where $y_{i,j}:=u_0^{i}u_1^{2-i}v_0^{j}v_1^{2-j}$. Let us consider $\mathbb{P}^8$ as the hyperplane $\{ x=0 \}$ in $\mathbb{P}^9_{[y_{0,0}:y_{0,1}:y_{0,2}:y_{1,0}:y_{1,1}:y_{1,2}:y_{2,0}:y_{2,1}:y_{2,2}:x]}$. Let $V$ be the cone over $P$ with vertex at the point $v:=[0:0:0:0:0:0:0:0:0:1]$; then $V$ is a singular Fano threefold
(see \cite[Lemma 3.1]{Pro07}).
By using Macaulay2, we 
see that the ideal of $V$ is generated by the following polynomials
\begin{center}
$ {y}_{2,1}^{2}-{y}_{2,0}{y}_{2,2},\quad 
{y}_{1,2} {y}_{2,1}-{y}_{1,1}{y}_{2,2},\quad 
{y}_{1,1} {y}_{2,1}-{y}_{1,0}{y}_{2,2}, \quad 
{y}_{0,2} {y}_{2,1}-{y}_{0,1}{y}_{2,2},$

${y}_{0,1} {y}_{2,1}-{y}_{0,0}{y}_{2,2},\quad 
{y}_{1,2} {y}_{2,0}-{y}_{1,0}{y}_{2,2}, \quad
{y}_{1,1} {y}_{2,0}-{y}_{1,0}{y}_{2,1}, \quad
{y}_{0,2} {y}_{2,0}-{y}_{0,0}{y}_{2,2},$

${y}_{0,1} {y}_{2,0}-{y}_{0,0}{y}_{2,1}, \quad
{y}_{1,2}^{2}-{y}_{0,2} {y}_{2,2},\quad
{y}_{1,1}{y}_{1,2}-{y}_{0,1} {y}_{2,2}, \quad
{y}_{1,0}{y}_{1,2}-{y}_{0,0} {y}_{2,2},$

${y}_{1,1}^{2}-{y}_{0,0}{y}_{2,2},\quad
{y}_{1,0} {y}_{1,1}-{y}_{0,0}{y}_{2,1},\quad
{y}_{0,2} {y}_{1,1}-{y}_{0,1}{y}_{1,2}, \quad
{y}_{0,1} {y}_{1,1}-{y}_{0,0}{y}_{1,2},$

${y}_{1,0}^{2}-{y}_{0,0} {y}_{2,0},\quad
{y}_{0,2}{y}_{1,0}-{y}_{0,0} {y}_{1,2},\quad
{y}_{0,1}{y}_{1,0}-{y}_{0,0} {y}_{1,1}, \quad
{y}_{0,1}^{2}-{y}_{0,0}{y}_{0,2}.$
\end{center}
Let us consider the involution $t$ of $\mathbb{P}^9$ given by
$t(x)=-x$ and $t(y_{i,j})=(-1)^{i+j}y_{i,j}$.
Let $v_{i,j}:=\{x=0, y_{k,l}=0 | (k,l)\ne (i,j)\}$. The locus of $t$-fixed points in $\mathbb{P}^{9}$ consists of two projective subspaces
$F_{1} = \{y_{0,0}=y_{0,2}=y_{1,1}=y_{2,0}=y_{2,2} = 0\} \cong \mathbb{P}^{4}$ and
$F_2 = \{y_{0,1}=y_{1,0}=y_{1,2}=y_{2,1}=x = 0\}\cong \mathbb{P}^{4}$
such that $F_{1}\cap V = \{v\}$ and $F_2\cap V = \{v_{0,0},v_{0,2},v_{2,0},v_{2,2}\}$.
Thus, $t$ defines an involution $\tau := t |_{V} : V \to V$ of $V$ with five fixed points. 
The quotient of $V$ via the involution $\tau$ is an Enriques-Fano threefold of genus $17$ (see \cite[Proposition 3.2]{Pro07}).
In particular, the quotient map $\pi : V \to V/\tau =: W_{P}^{17}$ is defined by the restriction on $V$ of the linear system $\mathcal{Q}$ of the quadric hypersurfaces of $\mathbb{P}^{9}$ of type
$$q_{1}(y_{0,0},y_{0,2},y_{1,1},y_{2,0},y_{2,2})+q_2(y_{0,1},y_{1,0},y_{1,2},y_{2,1},x)=0,$$
where $q_1$ and $q_2$ are quadratic homogeneous forms. 
Let us observe that the linear system $\mathcal{Q}$ defines a morphism $\varphi : \mathbb{P}^{9} \to \mathbb{P}^{29}$, $\left[y_{0,0}:\dots :y_{2,2}:x\right] \mapsto \left[Z_0:\dots :Z_{29} \right]$,
where
\begin{center}
$Z_{0} = {y}_{1,1}^{2}, \,\,
Z_{1} = {y}_{0,0}^{2}, \,\,
Z_{2} = {y}_{0,2}^{2}, \,\,
Z_{3} = {y}_{2,0}^{2}, \,\,
Z_{4} = {y}_{2,2}^{2}, \,\,
Z_{5} = x^{2}, \,\,
Z_{6} = {y}_{0,1}^{2},$

$Z_{7} = {y}_{1,0}^{2}, \,\, 
Z_{8} = {y}_{1,2}^{2}, \,\, 
Z_{9} = {y}_{2,1}^{2}, \,\, 
Z_{10} = {y}_{0,1} x, \,\,
Z_{11} = {y}_{1,0} x, \,\,
Z_{12} = {y}_{1,2} x,$

$Z_{13} = {y}_{2,1}x, \,\,
Z_{14} = {y}_{0,0}{y}_{1,1},\,\,
Z_{15} = {y}_{0,2} {y}_{1,1},\,\,
Z_{16} = {y}_{2,0} {y}_{1,1}, \,\,
Z_{17} = {y}_{2,2} {y}_{1,1},$

$Z_{18} = {y}_{0,1} {y}_{1,0},\,\,
Z_{19} = {y}_{0,1}{y}_{1,2}, \,\,
Z_{20} = {y}_{1,0} {y}_{2,1}, \,\, 
Z_{21} = {y}_{1,2}{y}_{2,1},$

$Z_{22} = {y}_{0,0} {y}_{0,2}, \,\,
Z_{23} = {y}_{0,0}{y}_{2,0},\,\,
Z_{24} = {y}_{0,2}{y}_{2,2}, \,\,
Z_{25} = {y}_{2,0}{y}_{2,2},$

$Z_{26} = {y}_{0,1} {y}_{2,1}, \,\,
Z_{27} = {y}_{0,0} {y}_{2,2}, \,\,
Z_{28} = {y}_{0,2} {y}_{2,0},\,\,
Z_{29} = {y}_{1,0}{y}_{1,2}.$
\end{center}
By looking at the ideal of $V\subset \mathbb{P}^9$, we deduce that $\varphi (V)$ is contained in a $17$-dimensional projective subspace of $\mathbb{P}^{29}$ given by 
\begin{center}
$H_{17}:= \{ Z_{18} = Z_{14}, \, 
Z_{19} = Z_{15}, \,
Z_{20} = Z_{16}, \,
Z_{21} = Z_{17},\,
Z_{22} = Z_{6}, \,
Z_{23} = Z_{7},$

$\quad \quad \quad Z_{24} = Z_{8}, \,
Z_{25} = Z_{9},\,
Z_{26} = Z_0, \,
Z_{27} = Z_0, \,
Z_{28} = Z_0, \,
Z_{29} = Z_0\}.$
\end{center}
Hence we obtain the morphism $\pi : V \to W_{P}^{17}=\varphi(V) \subset H_{17} \cong \mathbb{P}^{17}$ defined by 
$\left[y_{0,0}:y_{0,1}:y_{0,2}:y_{1,0}:y_{1,1}:y_{1,2}:y_{2,0}:y_{2,1}:y_{2,2}:x\right] \mapsto \left[z_0:z_1:\dots :z_{16}:z_{17} \right]$, where
\begin{center}
$z_{0} = {y}_{1,1}^{2}$,
$z_{1} = {y}_{0,0}^{2}$,
$z_{2} = {y}_{0,2}^{2}$,
$z_{3} = {y}_{2,0}^{2}$,
$z_{4} = {y}_{2,2}^{2}$,
$z_{5} = x^{2}$,
$z_{6} = {y}_{0,1}^{2}$, 
$z_{7} = {y}_{1,0}^{2}$, 
$z_{8} = {y}_{1,2}^{2}$,
$z_{9} = {y}_{2,1}^{2}$,
$z_{10} = {y}_{0,1} x$,
$z_{11} = {y}_{1,0} x$,
$z_{12} = {y}_{1,2} x$,
$z_{13} = {y}_{2,1}x$,
$z_{14} = {y}_{0,0}{y}_{1,1}$,
$z_{15} = {y}_{0,2} {y}_{1,1}$,
$z_{16} = {y}_{2,0} {y}_{1,1}$,
$z_{17} = {y}_{2,2} {y}_{1,1}.$
\end{center}
Thanks to Macaulay2 we find that the threefold $W_{P}^{17}$ has ideal generated by the following $88$ quadratic polynomials
\begin{center}
${z}_{15} {z}_{16}-{z}_{14} {z}_{17}, \quad
{z}_{12}{z}_{16}-{z}_{11} {z}_{17}, \quad
{z}_{9} {z}_{16}-{z}_{3} {z}_{17}, \quad
{z}_{8}{z}_{16}-{z}_{0} {z}_{17}, \quad
{z}_{6} {z}_{16}-{z}_{1} {z}_{17}, $

${z}_{4}{z}_{16}-{z}_{9} {z}_{17}, \quad
{z}_{2} {z}_{16}-{z}_{6} {z}_{17}, \quad
{z}_{0}{z}_{16}-{z}_{7} {z}_{17}, \quad
{z}_{13} {z}_{15}-{z}_{10} {z}_{17}, \quad
{z}_{9}{z}_{15}-{z}_{0} {z}_{17}, $

${z}_{8} {z}_{15}-{z}_{2} {z}_{17}, \quad
{z}_{7}{z}_{15}-{z}_{1} {z}_{17}, \quad
{z}_{4} {z}_{15}-{z}_{8} {z}_{17}, \quad
{z}_{3}{z}_{15}-{z}_{7} {z}_{17}, \quad
{z}_{0} {z}_{15}-{z}_{6} {z}_{17},$

${z}_{13}{z}_{14}-{z}_{10} {z}_{16}, \quad
{z}_{12} {z}_{14}-{z}_{11} {z}_{15}, \quad
{z}_{9}{z}_{14}-{z}_{7} {z}_{17}, \quad
{z}_{8} {z}_{14}-{z}_{6} {z}_{17}, \quad
{z}_{7}{z}_{14}-{z}_{1} {z}_{16}, $

${z}_{6} {z}_{14}-{z}_{1} {z}_{15}, \quad
{z}_{4}{z}_{14}-{z}_{0} {z}_{17}, \quad
{z}_{3} {z}_{14}-{z}_{7} {z}_{16}, \quad
{z}_{2}{z}_{14}-{z}_{6} {z}_{15}, \quad
{z}_{0} {z}_{14}-{z}_{1} {z}_{17},$

${z}_{12}{z}_{13}-{z}_{5} {z}_{17}, \quad
{z}_{11} {z}_{13}-{z}_{5} {z}_{16}, \quad
{z}_{8}{z}_{13}-{z}_{12} {z}_{17}, \quad
{z}_{7} {z}_{13}-{z}_{11} {z}_{16}, \quad
{z}_{6}{z}_{13}-{z}_{11} {z}_{15},$

${z}_{2} {z}_{13}-{z}_{12} {z}_{15}, \quad
{z}_{1}{z}_{13}-{z}_{11} {z}_{14}, \quad
{z}_{0} {z}_{13}-{z}_{11} {z}_{17}, \quad
{z}_{11} {z}_{12}-{z}_{10} {z}_{13},\quad
{z}_{10} {z}_{12}-{z}_{5} {z}_{15},$

${z}_{9} {z}_{12}-{z}_{13} {z}_{17}, \quad
{z}_{7} {z}_{12}-{z}_{10} {z}_{16}, \quad
{z}_{6}{z}_{12}-{z}_{10} {z}_{15}, \quad
{z}_{3} {z}_{12}-{z}_{13} {z}_{16}, \quad
{z}_{1}{z}_{12}-{z}_{10} {z}_{14},$

${z}_{0} {z}_{12}-{z}_{10} {z}_{17}, \quad
{z}_{10}{z}_{11}-{z}_{5} {z}_{14}, \quad
{z}_{9} {z}_{11}-{z}_{13} {z}_{16}, \quad
{z}_{8}{z}_{11}-{z}_{10} {z}_{17}, \quad
{z}_{6} {z}_{11}-{z}_{10} {z}_{14}, $

${z}_{4}{z}_{11}-{z}_{13} {z}_{17}, \quad
{z}_{2} {z}_{11}-{z}_{10} {z}_{15}, \quad
{z}_{0} {z}_{11}-{z}_{10} {z}_{16}, \quad
{z}_{9} {z}_{10}-{z}_{11} {z}_{17}, \quad
{z}_{8}{z}_{10}-{z}_{12} {z}_{15},$

${z}_{7} {z}_{10}-{z}_{11} {z}_{14}, \quad
{z}_{4}{z}_{10}-{z}_{12} {z}_{17}, \quad
{z}_{3} {z}_{10}-{z}_{11} {z}_{16}, \quad
{z}_{0}{z}_{10}-{z}_{11} {z}_{15}, \quad
{z}_{8} {z}_{9}-{z}_{17}^{2},$

${z}_{7} {z}_{9}-{z}_{16}^{2}, \quad
{z}_{6} {z}_{9}-{z}_{14} {z}_{17}, \quad
{z}_{5}{z}_{9}-{z}_{13}^{2}, \quad
{z}_{2} {z}_{9}-{z}_{15} {z}_{17}, \quad
{z}_{1}{z}_{9}-{z}_{14} {z}_{16},$

${z}_{0} {z}_{9}-{z}_{16} {z}_{17}, \quad
{z}_{7}{z}_{8}-{z}_{14} {z}_{17}, \quad
{z}_{6} {z}_{8}-{z}_{15}^{2}, \quad
{z}_{5}{z}_{8}-{z}_{12}^{2}, \quad
{z}_{3} {z}_{8}-{z}_{16} {z}_{17},$

${z}_{1}{z}_{8}-{z}_{14} {z}_{15}, \quad
{z}_{0} {z}_{8}-{z}_{15} {z}_{17}, \quad
{z}_{6}{z}_{7}-{z}_{14}^{2}, \quad
{z}_{5} {z}_{7}-{z}_{11}^{2}, \quad
{z}_{4}{z}_{7}-{z}_{16} {z}_{17},$

${z}_{2} {z}_{7}-{z}_{14} {z}_{15}, \quad
{z}_{0}{z}_{7}-{z}_{14} {z}_{16}, \quad
{z}_{5} {z}_{6}-{z}_{10}^{2}, \quad
{z}_{4}{z}_{6}-{z}_{15} {z}_{17}, \quad
{z}_{3} {z}_{6}-{z}_{14} {z}_{16},$

${z}_{0}{z}_{6}-{z}_{14} {z}_{15}, \quad
{z}_{0} {z}_{5}-{z}_{10} {z}_{13}, \quad
{z}_{3}{z}_{4}-{z}_{9}^{2}, \quad
{z}_{2} {z}_{4}-{z}_{8}^{2}, \quad
{z}_{1} {z}_{4}-{z}_{14}{z}_{17},$

${z}_{0} {z}_{4}-{z}_{17}^{2}, \quad
{z}_{2} {z}_{3}-{z}_{14}{z}_{17}, \quad
{z}_{1} {z}_{3}-{z}_{7}^{2}, \quad
{z}_{0}{z}_{3}-{z}_{16}^{2}, \quad
{z}_{1} {z}_{2}-{z}_{6}^{2},$

${z}_{0}{z}_{2}-{z}_{15}^{2}, \quad
{z}_{0} {z}_{1}-{z}_{14}^{2}, \quad
{z}_{0}^{2}-{z}_{14}{z}_{17}.$
\end{center}
Furthermore, $W_{P}^{17}$ has the following five singular points
$$P_1 := \pi (v_{0,0}),\,\,\,
P_2 := \pi (v_{0,2}),\,\,\,
P_3 := \pi (v_{2,0}),\,\,\,
P_4 := \pi (v_{2,2}),\,\,\,
P_5 := \pi (v),$$
that is $P_i=\{z_k=0 | k\ne i\}$ for $1\le i\le 5$.

\begin{proposition}\label{prop:cones1234p17}
If $i=1,2,3,4$, the tangent cone $TC_{P_i} W_{P}^{17}$ to $W_{P}^{17}$ at the point $P_{i}$ is a cone over a Veronese surface.
\end{proposition}
\begin{proof}
Each point $P_i$, $i=1,2,3,4$, can be viewed as the origin of the open affine set $U_i := \{z_i \ne 0\}$. The ideal of the tangent cone $TC_{P_i}(W_{P}^{17}\cap U_i)$ is generated by the minimal degree homogeneous parts of all the polynomials in the ideal of $W_{P}^{17} \cap U_i$.
By using Macaulay2 we obtain the following tangent cones.

$TC_{P_1}(W_{P}^{17}\cap U_1)$ has ideal generated by ${z}_{17},{z}_{16},{z}_{15},{z}_{13},{z}_{12},{z}_{9},{z}_{8},{z}_{4},{z}_{3},{z}_{2},{z}_{0},$
${z}_{10}{z}_{11}-{z}_{5}{z}_{14},\, {z}_{6} {z}_{11}-{z}_{10} {z}_{14},\, {z}_{7} {z}_{10}-{z}_{11}{z}_{14},\, {z}_{6} {z}_{7}-{z}_{14}^{2},\, {z}_{5}{z}_{7}-{z}_{11}^{2},\, {z}_{5} {z}_{6}-{z}_{10}^{2}.$
Hence $TC_{P_1}W_{P}^{17}$ is a cone with vertex $P_1$ over a Veronese surface in the $\mathbb{P}^{5}$ defined by $\{ z_i = 0 | i = 0,1,2,3,4,8,9,12,13,15,16,17 \}$.

$TC_{P_2}(W_{P}^{17}\cap U_2)$ has ideal generated by ${z}_{17},{z}_{16},{z}_{14},{z}_{13},{z}_{11},{z}_{9},{z}_{7},{z}_{4},{z}_{3},{z}_{1},{z}_{0},$
${z}_{10} {z}_{12}-{z}_{5}{z}_{15}, \, {z}_{6} {z}_{12}-{z}_{10} {z}_{15}, \, {z}_{8} {z}_{10}-{z}_{12}{z}_{15}, \, {z}_{6} {z}_{8}-{z}_{15}^{2}, \, {z}_{5}{z}_{8}-{z}_{12}^{2}, \, {z}_{5} {z}_{6}-{z}_{10}^{2}.$
Hence $TC_{P_2}W_{P}^{17}$ is a cone with vertex $P_2$ over a Veronese surface in the $\mathbb{P}^{5}$ defined by $\{ z_i = 0 | i = 0,1,2,3,4,7,9,11,13,14,16,17 \}$.

$TC_{P_3}(W_{P}^{17}\cap U_3)$ has ideal generated by ${z}_{17},{z}_{15},{z}_{14},{z}_{12},{z}_{10},{z}_{8},{z}_{6},{z}_{4},{z}_{2},{z}_{1},{z}_{0},$
${z}_{11} {z}_{13}-{z}_{5}{z}_{16}, \, {z}_{7} {z}_{13}-{z}_{11} {z}_{16}, \, {z}_{9} {z}_{11}-{z}_{13}{z}_{16},\, {z}_{7} {z}_{9}-{z}_{16}^{2}, \, {z}_{5}{z}_{9}-{z}_{13}^{2}, \, {z}_{5} {z}_{7}-{z}_{11}^{2}.$ 
Hence $TC_{P_3}W_{P}^{17}$ is a cone with vertex $P_3$ over a Veronese surface in the $\mathbb{P}^{5}$ defined by $\{ z_i = 0 | i = 0,1,2,3,4,6,8,10,12,14,15,17 \}$.

$TC_{P_4}(W_{P}^{17}\cap U_4)$ has ideal generated by ${z}_{16},{z}_{15},{z}_{14},{z}_{11},{z}_{10},{z}_{7},{z}_{6},{z}_{3},{z}_{2},{z}_{1},{z}_{0},$
${z}_{12} {z}_{13}-{z}_{5}{z}_{17},\, {z}_{8} {z}_{13}-{z}_{12} {z}_{17},\, {z}_{9} {z}_{12}-{z}_{13}{z}_{17},\, {z}_{8} {z}_{9}-{z}_{17}^{2},\, {z}_{5}{z}_{9}-{z}_{13}^{2},\ {z}_{5} {z}_{8}-{z}_{12}^{2}.$
Hence $TC_{P_4}W_{P}^{17}$ is a cone with vertex $P_4$ over a Veronese surface in the $\mathbb{P}^{5}$ defined by $\{ z_i = 0 | i = 0,1,2,3,4,6,7,10,11,14,15,16 \}$. 
\end{proof}


\begin{theorem}\label{thm:cone5p17}
The tangent cone $TC_{P_5} W_{P}^{17}$ to $W_{P}^{17}$ at the point $P_{5}$ is a cone over a reducible sextic surface $M_6 \subset \mathbb{P}^{7} \subset \mathbb{P}^{17}$, which is given by the union of four planes $\pi_{1},\pi_{2},\pi_{1}',\pi_{2}'$ and a quadric surface $Q \subset \mathbb{P}^{3} \subset \mathbb{P}^{7}$. In particular each one of the planes $\pi_{1},\pi_{2},\pi_{1}',\pi_{2}'$ intersects the quadric $Q$ respectively along a line $l_{1}$ , $l_{2}$, $l_{1}'$, $l_{2}'$, where $l_i$ is disjoint from $l_i'$, for $i=1,2$. In the other cases the intersections of two of these lines identify four points, i.e.
$q_{1,2} := l_1 \cap l_2,\, q_{1,2'} := l_1 \cap l_2', \, q_{1',2} := l_1' \cap l_2, \, q_{1',2'} := l_1' \cap l_2'$.
\begin{figure}[h]
\centering
\includegraphics[scale=0.5]{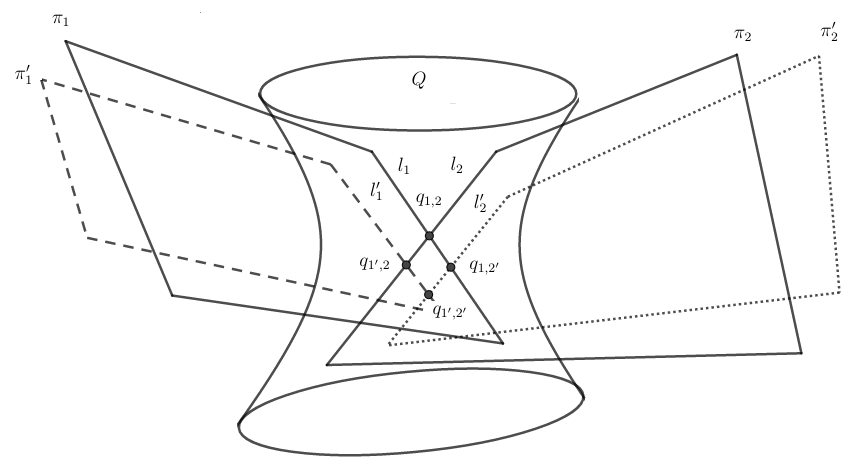}
\caption{The reducible sextic surface $M_6\subset \mathbb{P}^{7}$ given by the union of four planes $\pi_{1}$, $\pi_{2}$, $\pi_{1}'$, $\pi_{2}'$ and a quadric surface $Q\subset \mathbb{P}^{3}\subset \mathbb{P}^{7}$, which intersect as in the statement of Theorems~\ref{thm:cone5p17},~\ref{thm:cone5KLM}.}
\label{fig:sexticM}
\end{figure}
\end{theorem}
\begin{proof}
The point $P_5$ can be viewed as the origin of the open affine set given by $U_5 := \{z_5 \ne 0\}$. The ideal of the tangent cone $TC_{P_5}(W_{P}^{17}\cap U_5)$ is generated by the minimal degree homogeneous parts of all the polynomials in the ideal of $W_{P}^{17} \cap U_5$. 
By using Macaulay2 we find that $TC_{P_5}(W_{P}^{17}\cap U_5)$ has ideal generated by the following polynomials
$${z}_{17},\,{z}_{16},\,{z}_{15},\,{z}_{14},\,{z}_{9},\,{z}_{8},\,{z}_{7},\,{z}_{6},\,{z}_{0},\quad
{z}_{11}{z}_{12}-{z}_{10} {z}_{13},$$
$${z}_{2} {z}_{13}, {z}_{1} {z}_{13}, {z}_{3} {z}_{12}, {z}_{1} {z}_{12}, {z}_{4}{z}_{11}, {z}_{2} {z}_{11}, {z}_{4} {z}_{10}, {z}_{3} {z}_{10}, {z}_{3}{z}_{4}, {z}_{2} {z}_{4}, {z}_{1} {z}_{4}, {z}_{2} {z}_{3}, {z}_{1}{z}_{3}, {z}_{1} {z}_{2}.$$
Hence $TC_{P_5}W_{P}^{17}$ is a cone with vertex $P_5$ over a surface $M_6$ contained in the $\mathbb{P}^{7}$ defined by $\{ z_i = 0 | i = 0,5,6,7,8,9,14,15,16,17 \}$. This surface $M_6$ is the union of four planes $\pi_{1}, \pi_{2}, \pi_{1}', \pi_{2}'$ and a quadric surface $Q$, where

\begin{center}
$\pi_{1} := \{z_i = 0 | i=0,1,3,4,5,6,7,8,9,11,13,14,15,16,17 \},$

$\pi_{2} := \{z_i = 0 | i=0,2,3,4,5,6,7,8,9,12,13,14,15,16,17 \},$

$\pi_{1}' := \{z_i = 0 | i=0,1,2,4,5,6,7,8,9,10,12,14,15,16,17 \},$

$\pi_{2}' := \{z_i = 0 | i=0,1,2,3,5,6,7,8,9,10,11,14,15,16,17 \},$

$Q := \{z_i = 0 | i=0,1,2,3,4,5,6,7,8,9,14,15,16,17 \} \cap \{z_{11}z_{12}-z_{10}z_{13}=0\}.$
\end{center}
We give an idea of $M_6$ in Figure~\ref{fig:sexticM}.
\end{proof}



By Proposition~\ref{prop:cones1234p17} and Theorem~\ref{thm:cone5p17} we have that $W_P^{17}$ has five non-similar points. For completeness, let us find their configuration.
Let $l_{i,j}:=\{z_k=0|k\ne i,j\}$ be the line joining the singular points $P_i$ and $P_j$ with $1\le i<j \le 5$. By looking at the ideal of $W_P^{17}$, we deduce that the lines $l_{1,5}$, $l_{2,5}$, $l_{3,5}$, $l_{4,5}$ are contained in $W_{P}^{17}$, while the lines $l_{1,2}$, $l_{1,3}$, $l_{1,4}$, $l_{2,3}$, $l_{2,4}$, $l_{3,4}$ are not. 
Hence the five singular points $P_1$, $P_2$, $P_3$, $P_4$, $P_5$ of $W_{P}^{17}$ are associated as in Table~\ref{tab:proKLM}
of Appendix~\ref{app:congifuration}.

\begin{remark}\label{rem:singpro17NOCM2}
Let $bl: \operatorname{Bl}_{P_i=1,\dots 5}\mathbb{P}^{17} \to \mathbb{P}^{17}$ be the blow-up of $\mathbb{P}^{17}$ at the five singular points of $W_P^{17}$ and let $\widetilde{W}$ be the strict transform of $W_P^{17}$. Then $\widetilde{W}$ intersects the exceptional divisor $bl^{-1}(P_5)$ along a surface which is isomorphic to $M_6$ and which has four singular points locally given by the intersection of three planes of $\mathbb{P}^4$, such that two of them intersect the third along two lines and intersect each other at a point which is the intersection of these two lines. 
Consequently, $\widetilde{W}$ is not a desingularization of $W_P^{17}$, since there are four singular points infinitely near to $P_5$. Therefore it is not enough to blow-up $W_P^{17}$ at $P_1$, $P_2$, $P_3$, $P_4$, $P_5$ to solve the singularities of $W_P^{17}$, as instead implicitly assumed by Fano and explicitly by Conte-Murre (see \cite[\S 3.10]{CoMu85}).


\end{remark}

\section{The P-EF 3-fold of genus 13}\label{subsec:pro13}

Let us study the Enriques-Fano threefold mentioned in \cite[Remark 3]{Pro07}. We refer to \S~\ref{code:pro13} of Appendix~\ref{app:code} for the computational techniques we will use.
Let us take the smooth sextic Del Pezzo surface $S_6\subset \mathbb{P}^{6}_{[x_0:x_1:x_2:x_3:x_4:x_5:x_6]}$ defined in \S~\ref{subsec:bayle10}.
%
Let us consider $\mathbb{P}^{6}$ as the hyperplane $\{y_0=0\} \subset \mathbb{P}^7_{[x_0:x_1:x_2:x_3:x_4:x_5:x_6:y]}$ and let $V$ be the cone over $S_6$ with vertex $v:=[0:0:0:0:0:0:0:1]$. With a similar proof as the one of \cite[Lemma 3.1]{Pro07}, one can see that $V$ is a singular Fano threefold. 
Let $t$ be the involution of $\mathbb{P}^7$ defined by 
$\left[x_0 : \dots : x_6: y\right] \mapsto \left[x_2 : x_4 : x_0 : x_5 : x_1 : x_3 : x_6: -y\right].$
The locus of $t$-fixed points in $\mathbb{P}^{7}$ consists of two projective subspaces
$$F_1 = \{x_0+x_2 = x_1+x_4 = x_3+x_5 = x_6 = 0\}\cong \mathbb{P}^{3},$$
$$F_{2} = \{x_0-x_2 = x_1-x_4 = x_3-x_5 = y = 0\} \cong \mathbb{P}^{3}.$$
In particular we have that $F_1\cap V = \{v\}$ and $F_{2}\cap V = \{v_1,v_2,v_3,v_4\}$, where
$$v_1 := \left[1:1:1:1:1:1:1:0 \right], \quad v_2 :=\left[1:-1:1:-1:-1:-1:1:0\right],$$
$$v_3 := \left[-1:1:-1:-1:1:-1:1:0\right],\quad v_4 := \left[-1:-1:-1:1:-1:1:1:0\right].$$
Thus, $t$ induces an involution $\tau := t |_{V}$ of $V$ with five fixed points.  
The quotient of $V$ via the involution $\tau$ is an Enriques-Fano threefold of genus $13$, which we will denote by $W_{P}^{13}$: one can deduce this by using a similar proof as the one of \cite[Proposition 3.2]{Pro07}.
In this case the quotient map $\pi : V \to V/\tau = W_{P}^{13}$ is defined by the restriction on $V$ of the linear system $\mathcal{Q}$ of the quadric hypersurfaces of $\mathbb{P}^{9}$ of type
$$q_1(x_0+x_2,x_1+x_4,x_3+x_5,x_6)+q_{2}(x_0-x_2,x_1-x_4,x_3-x_5,y)=0,$$
where $q_1$ and $q_2$ are quadratic homogeneous forms. 
The linear system $\mathcal{Q}$
defines a morphism $\varphi : \mathbb{P}^{7} \to \mathbb{P}^{19}$ such that 
$\left[x_0:\dots:x_6:y\right] \mapsto \left[Z_0:\dots :Z_{19} \right]$
where
\begin{center}
$Z_0 = x_6^2$,
$Z_1 = x_0^2+x_2^2$, 
$Z_2 = x_1^2+x_4^2$, 
$Z_3 = x_3^2+x_5^2$,
$Z_4 = (x_0+x_2)x_6$,
$Z_5 = (x_1+x_4)x_6$,
$Z_6 = (x_3+x_5)x_6$,
$Z_7 = x_0x_1+x_2x_4$,
$Z_8 = x_2x_3+x_0x_5$, 
$Z_9 = x_1x_3+x_4x_5$,
$Z_{10} = (x_0-x_2)y$,
$Z_{11} = (x_1-x_4)y$,
$Z_{12} = (x_3-x_5)y$,
$Z_{13} = y^2$,
$Z_{14} = 2x_0x_2$,
$Z_{15} = 2x_1x_4$,
$Z_{16} = 2x_3x_5$,
$Z_{17} = x_4x_3+x_1x_5$,
$Z_{18} = x_0x_3+x_2x_5$,
$Z_{19} = x_1x_2+x_0x_4$.
\end{center}
By looking at the ideal of $V\subset \mathbb{P}^7$, we observe that the threefold $\varphi (V)$ is contained in a $13$-dimensional projective subspace of $\mathbb{P}^{19}$ given by 
$$H_{13}:= \{Z_{14} = 2Z_0, \,\,
Z_{15} = 2Z_0, \,\,
Z_{16} = 2Z_0, \,\,
Z_{17} = Z_4, \,\,
Z_{18} = Z_5, \,\,
Z_{19} = Z_6\}.$$
%
%
%
Thus, we obtain the morphism $\pi : V \to W_{P}^{13} = \varphi (V) \subset H_{13} \cong \mathbb{P}^{13}$ defined by 
$\left[x_0:\dots:x_6:y\right] \mapsto \left[z_0:\dots :z_{13} \right]$
where
\begin{center}
$z_0 = x_6^2$,
$z_1 = x_0^2+x_2^2$,   
$z_2 = x_1^2+x_4^2$, 
$z_3 = x_3^2+x_5^2$,
$z_4 = (x_0+x_2)x_6$,
$z_5 = (x_1+x_4)x_6$, 
$z_6 = (x_3+x_5)x_6$, 
$z_7 = x_0x_1+x_2x_4$, 
$z_8 = x_2x_3+x_0x_5$, 
$z_9 = x_1x_3+x_4x_5$,
$z_{10} = (x_0-x_2)y$,
$z_{11} = (x_1-x_4)y$,
$z_{12} = (x_3-x_5)y$,
$z_{13} = y^2.$
\end{center}
By using Macaulay2, we find that the threefold $W_{P}^{13}$ has ideal generated by the following $42$ quadratic polynomials 
\begin{center}
${z}_{4} {z}_{5}-2 {z}_{0} {z}_{6}-{z}_{2} {z}_{6}+{z}_{5}{z}_{9}, \quad
{z}_{5}^{2}-{z}_{6}^{2}-{z}_{6} {z}_{7}+{z}_{5} {z}_{8},\quad 
2 {z}_{0}{z}_{5}+{z}_{3} {z}_{5}-{z}_{4} {z}_{6}-{z}_{6}{z}_{9},$

${z}_{4}^{2}-{z}_{6}^{2}-{z}_{6} {z}_{7}+{z}_{4} {z}_{9}, \quad
{z}_{4}{z}_{5}-2 {z}_{0} {z}_{6}-{z}_{1} {z}_{6}+{z}_{4} {z}_{8},\quad
-2 {z}_{0}{z}_{5}-{z}_{1} {z}_{5}+{z}_{4} {z}_{6}+{z}_{4} {z}_{7},$

$2 {z}_{0}{z}_{4}+{z}_{3} {z}_{4}-{z}_{5} {z}_{6}-{z}_{6} {z}_{8},\quad
2 {z}_{0}{z}_{4}+{z}_{2} {z}_{4}-{z}_{5} {z}_{6}-{z}_{5} {z}_{7},\quad
4{z}_{0}^{2}-{z}_{4}^{2}-{z}_{5}^{2}+{z}_{6} {z}_{7},$

${z}_{5}{z}_{10}-{z}_{4} {z}_{11}+2 {z}_{0} {z}_{12},\quad
-{z}_{6} {z}_{10}+2 {z}_{0}{z}_{11}-{z}_{4} {z}_{12},\quad
2 {z}_{0} {z}_{10}-{z}_{6} {z}_{11}+{z}_{5}{z}_{12},$

$2 {z}_{0} {z}_{4}-2 {z}_{5} {z}_{6}+2 {z}_{0} {z}_{9},\quad
2 {z}_{0}{z}_{5}-2 {z}_{4} {z}_{6}+2 {z}_{0} {z}_{8},\quad
-2 {z}_{4} {z}_{5}+2 {z}_{0}{z}_{6}+2 {z}_{0} {z}_{7},$

$2 {z}_{0} {z}_{3}+{z}_{4}^{2}+{z}_{5}^{2}-2{z}_{6}^{2}-{z}_{6} {z}_{7},\quad
2 {z}_{0}{z}_{2}+{z}_{4}^{2}-{z}_{5}^{2}-{z}_{6} {z}_{7},\quad
2 {z}_{0}{z}_{1}-{z}_{4}^{2}+{z}_{5}^{2}-{z}_{6} {z}_{7},$

${z}_{12}^{2}+2 {z}_{0}{z}_{13}-{z}_{3} {z}_{13},\quad
{z}_{11} {z}_{12}+{z}_{4} {z}_{13}-{z}_{9}{z}_{13}, \quad
{z}_{10} {z}_{12}-{z}_{5} {z}_{13}+{z}_{8} {z}_{13},$

${z}_{4}{z}_{10}-{z}_{5} {z}_{11}+{z}_{7} {z}_{12},\quad
{z}_{11}^{2}+2 {z}_{0}{z}_{13}-{z}_{2} {z}_{13},\quad
{z}_{10} {z}_{11}+{z}_{6} {z}_{13}-{z}_{7}{z}_{13},$

$-{z}_{5} {z}_{10}+{z}_{9} {z}_{11}-{z}_{2} {z}_{12},\quad
-{z}_{4}{z}_{10}+{z}_{8} {z}_{11}-{z}_{6} {z}_{12},\quad
-{z}_{6} {z}_{10}+{z}_{3}{z}_{11}-{z}_{9} {z}_{12},$

${z}_{10}^{2}+2 {z}_{0} {z}_{13}-{z}_{1}{z}_{13},\quad
{z}_{9} {z}_{10}-{z}_{5} {z}_{11}+{z}_{6} {z}_{12},\quad
{z}_{8}{z}_{10}-{z}_{4} {z}_{11}+{z}_{1} {z}_{12},$

${z}_{7} {z}_{10}-{z}_{1}{z}_{11}+{z}_{4} {z}_{12},\quad
{z}_{3} {z}_{10}-{z}_{6} {z}_{11}+{z}_{8}{z}_{12},\quad
{z}_{2} {z}_{10}-{z}_{7} {z}_{11}+{z}_{5} {z}_{12},$

$-{z}_{4}{z}_{5}+2 {z}_{0} {z}_{6}-{z}_{3} {z}_{6}+{z}_{8} {z}_{9},\quad
2 {z}_{0}{z}_{5}-{z}_{2} {z}_{5}-{z}_{4} {z}_{6}+{z}_{7} {z}_{9},\quad
2 {z}_{0}{z}_{4}-{z}_{5} {z}_{7}-{z}_{6} {z}_{8}+{z}_{1} {z}_{9},$

$2 {z}_{0}{z}_{4}-{z}_{1} {z}_{4}-{z}_{5} {z}_{6}+{z}_{7} {z}_{8},\quad
-{z}_{1}{z}_{5}+{z}_{4} {z}_{6}+{z}_{2} {z}_{8}-{z}_{6} {z}_{9},$

$2 {z}_{4}{z}_{5}-2 {z}_{0} {z}_{6}-{z}_{1} {z}_{6}-{z}_{2} {z}_{6}+{z}_{3}{z}_{7},\quad
{z}_{2} {z}_{3}+{z}_{5}^{2}-{z}_{6} {z}_{7}-{z}_{9}^{2},$

${z}_{1}{z}_{3}+{z}_{4}^{2}-{z}_{6} {z}_{7}-{z}_{8}^{2},\quad
{z}_{1}{z}_{2}+{z}_{4}^{2}+{z}_{5}^{2}-{z}_{6}^{2}-{z}_{6} {z}_{7}-{z}_{7}^{2}.$
\end{center}
Furthermore, $W_{P}^{13}$ has the following five singular points
$$P_1 := \pi (v_1) = 
\left[1: 2: 2: 2: 2: 2: 2: 2: 2: 2: 0: 0: 0: 0 \right],$$
$$P_2 := \pi (v_2) = 
\left[1: 2: 2: 2: 2: -2: -2: -2: -2: 2: 0: 0: 0: 0\right],$$
$$P_3 := \pi (v_3) = 
\left[1: 2: 2: 2: -2: 2: -2: -2: 2: -2: 0: 0: 0: 0\right],$$
$$P_4 := \pi (v_4) = 
\left[1: 2: 2: 2: -2: -2: 2: 2: -2: -2: 0: 0: 0: 0\right],$$
$$P_5 := \pi (v) = 
\left[0 : 0 : 0 : 0 :  0 :  0 :  0 :  0 :  0 :  0 : 0 : 0 : 0 : 1\right].$$

\begin{proposition}\label{prop:cones1234p13}
If $i=1,2,3,4$, the tangent cone $TC_{P_i} W_{P}^{13}$ to $W_{P}^{13}$ at the point $P_{i}$ is a cone over a Veronese surface.
\end{proposition}
\begin{proof}
Let us consider the following change of coordinates of $\mathbb{P}^{13}$
$$z_0 = w_0, \quad
z_i = w_i + 2w_0, \quad 
z_{j} = w_{j}, \quad i=1,\dots , 9,\, j=10,\dots , 13.$$
With respect to the new system of coordinates $\left[w_0:\dots :w_{13}\right]$ of $\mathbb{P}^{13}$, the point $P_1$ has coordinates $\left[1: 0: \dots : 0 \right]$ and
can be viewed as the origin of the open affine set $U_0 := \{w_0 \ne 0\}$.
The ideal of the tangent cone $TC_{P_1}(W_{P}^{13}\cap U_0)$ is generated by the minimal degree homogeneous parts of all the polynomials in the ideal of $W_{P}^{13} \cap U_0$.
Via Macaulay, $TC_{P_1}(W_{P}^{13}\cap U_0)$ is found to have ideal generated by 
\begin{center}
$-9 {w}_{1}+8 {w}_{7}+8 {w}_{8}-4 {w}_{9},\quad $
$-9 {w}_{2}+8{w}_{7}-4 {w}_{8}+8 {w}_{9}, \quad$
$-9 {w}_{3}-4 {w}_{7}+8 {w}_{8}+8 {w}_{9},$

$-9{w}_{4}+2 {w}_{7}+2 {w}_{8}-{w}_{9},\quad$
$-9 {w}_{5}+2 {w}_{7}-{w}_{8}+2{w}_{9},\quad$
$-9 {w}_{6}-{w}_{7}+2 {w}_{8}+2{w}_{9},$

${w}_{10}-{w}_{11}+{w}_{12},\quad $
$9 {w}_{11} {w}_{12}+2 {w}_{7}{w}_{13}+2 {w}_{8} {w}_{13}-10 {w}_{9} {w}_{13},$

$2 {w}_{7} {w}_{11}-10{w}_{8} {w}_{11}+2 {w}_{9} {w}_{11}-10 {w}_{7} {w}_{12}+2 {w}_{8}{w}_{12}+2 {w}_{9} {w}_{12},$

$6 {w}_{7} {w}_{11}-6 {w}_{8} {w}_{11}-18{w}_{9} {w}_{11}+6 {w}_{7} {w}_{12}-6 {w}_{8} {w}_{12}+18 {w}_{9}{w}_{12},$

$9 {w}_{12}^{2}+4 {w}_{7} {w}_{13}-8{w}_{8} {w}_{13}-8 {w}_{9} {w}_{13}, \quad 
9 {w}_{11}^{2}-8 {w}_{7} {w}_{13}+4{w}_{8} {w}_{13}-8 {w}_{9} {w}_{13},$

${w}_{7}^{2}-2 {w}_{7} {w}_{8}+{w}_{8}^{2}-2 {w}_{7} {w}_{9}-2{w}_{8} {w}_{9}+{w}_{9}^{2}.$
\end{center}
Hence $TC_{P_1}W_{P}^{13}$ is a cone with vertex at $P_1$ over a Veronese surface in 
\begin{center}
$\{ w_0=0,\,  -9 {w}_{1}+8 {w}_{7}+8 {w}_{8}-4 {w}_{9} = 0, \, -9 {w}_{2}+8{w}_{7}-4 {w}_{8}+8 {w}_{9}=0,$

$-9 {w}_{3}-4 {w}_{7}+8 {w}_{8}+8 {w}_{9}=0,\, -9{w}_{4}+2 {w}_{7}+2 {w}_{8}-{w}_{9}=0,$

$-9 {w}_{5}+2 {w}_{7}-{w}_{8}+2{w}_{9}=0,\, -9 {w}_{6}-{w}_{7}+2 {w}_{8}+2{w}_{9}=0,\, {w}_{10}-{w}_{11}+{w}_{12}=0 \}\cong \mathbb{P}^{5}.$
\end{center}
Similar analysis for the points $P_2$, $P_3$ and $P_4$.
\end{proof}

\begin{theorem}\label{thm:cone5p13}
The tangent cone $TC_{P_5} W_{P}^{13}$ to $W_{P}^{13}$ at the point $P_{5}$ is a cone over a reducible quintic surface $M_5$, which is given by the union of five planes $\pi_{0},\pi_{1},\pi_{2},\pi_{3},\pi_{4}$, such that the four planes $\pi_{1},\pi_{2},\pi_{3},\pi_{4}$ intersect the plane $\pi_0$ along the four edges of a quadrilateral.
\begin{figure}[h]
\centering
\includegraphics[scale=0.5]{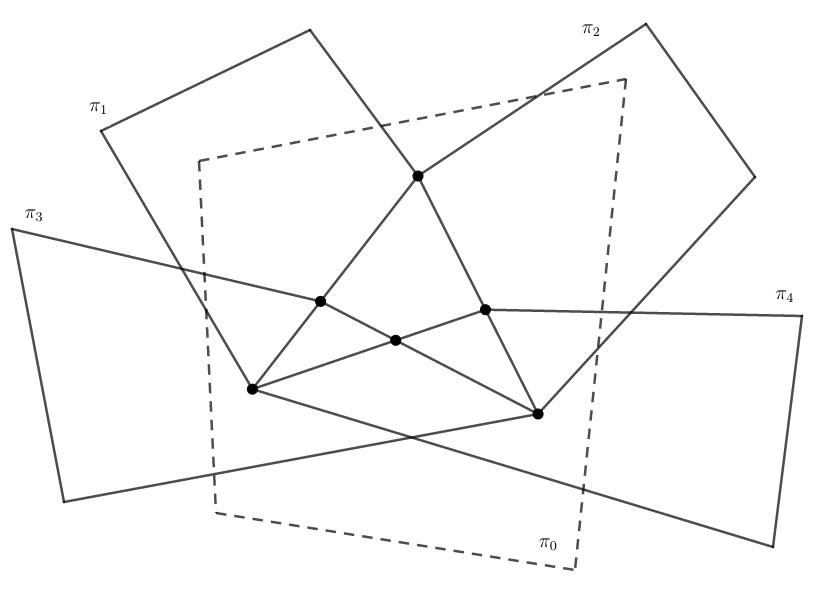}
\caption{The reducible quintic surface $M_5\subset \mathbb{P}^{6}$ given by the union of five planes $\pi_{0}$, $\pi_{1}$, $\pi_{2}$, $\pi_{3}$, $\pi_{4}$, which intersect as in the statement of Theorem~\ref{thm:cone5p13}.}
\label{fig:quinticM}
\end{figure}
\end{theorem}
\begin{proof}
The point $P_5$ can be viewed as the origin of the open affine set given by $U_{13} := \{z_{13} \ne 0\}$. The ideal of the tangent cone $TC_{P_5}(W_{P}^{13}\cap U_{13})$ is generated by the minimal degree homogeneous parts of all the polynomials in the ideal of $W_{P}^{13} \cap U_{13}$. 
By using Macaulay2, we find that $TC_{P_5}(W_{P}^{13}\cap U_{13})$ has ideal generated by the following polynomials
\begin{center}
${z}_{6}-{z}_{7},\quad {z}_{5}-{z}_{8},\quad {z}_{4}-{z}_{9},\quad 
{z}_{2}-{z}_{3},\quad {z}_{1}-{z}_{3}, \quad 2{z}_{0}- {z}_{3}, \quad {z}_{9}{z}_{10}-{z}_{8} {z}_{11}+{z}_{7} {z}_{12},$

${z}_{8} {z}_{10}-{z}_{9}{z}_{11}+{z}_{3} {z}_{12}, \quad 
{z}_{7} {z}_{10}-{z}_{3} {z}_{11}+{z}_{9}{z}_{12},\quad
{z}_{3} {z}_{10}-{z}_{7} {z}_{11}+{z}_{8}{z}_{12},$

${z}_{8}^{2}-{z}_{9}^{2},\quad {z}_{7}^{2}-{z}_{9}^{2}, \quad {z}_{3}^{2}-{z}_{9}^{2}, \quad 
{z}_{7} {z}_{8}-{z}_{3} {z}_{9},\quad
{z}_{3}{z}_{8}-{z}_{7} {z}_{9},\quad
{z}_{3} {z}_{7}-{z}_{8}{z}_{9}.$
\end{center}
Hence $TC_{P_5}W_{P}^{13}$ is a cone with vertex at $P_5$ over a surface $M_5$ contained in 
the $\mathbb{P}^6$ defined by $\{z_{13}=0,\, {z}_{6}={z}_{7},\, {z}_{5}={z}_{8},\, {z}_{4}={z}_{9},\, 
{z}_{2}={z}_{3},\, {z}_{1}={z}_{3}, \, 2{z}_{0}= {z}_{3}\}.$
This surface $M_5$, an idea of which is given in Figure~\ref{fig:quinticM}, is the union of the following five planes:
\begin{center}
$\pi_{0} := \{z_i = 0 | i\ne 10,11\},$

$\pi_{1} := \{
2 {z}_{0}={z}_{1}={z}_{2}={z}_{3}={z}_{4}={z}_{5}={z}_{6}={z}_{7}={z}_{8}={z}_{9},\,$ \\
${z}_{10}={z}_{11}-{z}_{12},\, {z}_{13}=0\},$

$\pi_{2} := \{
2 {z}_{0}={z}_{1}={z}_{2}={z}_{3}={z}_{4}=-{z}_{5}=-{z}_{6}=-{z}_{7}=-{z}_{8}={z}_{9},\,$\\ 
${z}_{10}={z}_{12}-{z}_{11},\, {z}_{13}=0\},$

$\pi_{3} := \{
2 {z}_{0}={z}_{1}={z}_{2}={z}_{3}=-{z}_{4}={z}_{5}=-{z}_{6}=-z_7=z_8=-{z}_{9},\,$\\
${z}_{10}=-{z}_{11}-{z}_{12},\, {z}_{13}=0 \},$

$\pi_{4} := \{
2 {z}_{0}={z}_{1}={z}_{2}={z}_{3}=-{z}_{4}=-{z}_{5}={z}_{6}={z}_{7}=-{z}_{8}=-{z}_{9},\,$\\
${z}_{10}={z}_{11}+{z}_{12},\, {z}_{13}=0\}.$
\end{center}
\end{proof}

By Proposition~\ref{prop:cones1234p13} and Theorem~\ref{thm:cone5p13} we have that $W_P^{13}$ has five non-similar points. For completeness, let us find their configuration.
Let $l_{i,j}$ be the line joining the singular points $P_i$ and $P_j$ for $1\le i<j\le 5$. 
Thanks to Macaulay2 we find that the lines $l_{1,5}$, $l_{2,5}$, $l_{3,5}$, $l_{4,5}$ are contained in $W_{P}^{13}$, while the lines $l_{1,2}$, $l_{1,3}$, $l_{1,4}$, $l_{2,3}$, $l_{2,4}$, $l_{3,4}$ are not. Hence the five singular points $P_1$, $P_2$, $P_3$, $P_4$, $P_5$ of $W_{P}^{13}$ are associated as in 
Table~\ref{tab:proKLM}
of Appendix~\ref{app:congifuration}.

\begin{remark}\label{rem:singpro13NOCM2}
Let $bl: \operatorname{Bl}_{P_i=1,\dots 5}\mathbb{P}^{13} \to \mathbb{P}^{13}$ be the blow-up of $\mathbb{P}^{13}$ at the five singular points of $W_P^{13}$ and let $\widetilde{W}$ be the strict transform of $W_P^{13}$. Then $\widetilde{W}$ intersects the exceptional divisor $bl^{-1}(P_5)$ along a surface isomorphic to $M_5$, which has six singular points locally given by the intersection of three planes of $\mathbb{P}^4$, such that two of them intersect the third along two lines and intersect each other at a point which is intersection of these two lines. Therefore $\widetilde{W}$ is not a desingularization of $W_P^{13}$, since there are six singular points infinitely near to $P_5$.
Therefore it is not enough to blow-up $W_P^{13}$ at $P_1$, $P_2$, $P_3$, $P_4$, $P_5$ to solve the singularities of $W_P^{13}$, as instead implicitly assumed by Fano and explicitly by Conte-Murre (see \cite[\S 3.10]{CoMu85}).

\end{remark}

\section{The KLM-EF 3-fold of genus 9}\label{sec:KLM}

Let us study the Enriques-Fano threefold described in (see \cite[\S 13]{KLM11}). We refer to \S~\ref{code:KLM} of Appendix~\ref{app:code} for the computational techniques we will use.
Let $\mathcal{S}$ be the linear system of the sextic surfaces of $\mathbb{P}^{3}_{\left[t_0:\dots : t_3\right]}$ having double points along the six edges $l_{ij}:=\{t_i=t_j=0\}$ of a fixed tetrahedron $T:=\{t_0t_1t_2t_3=0\}$, for $0\le i<j \le 3$. Let us denote the vertices of $T$ by $v_i:=\{t_k=0|k\ne i\}$, for $0\le i \le 3$. Then $\mathcal{S}$ defines a rational map
$\nu_{\mathcal{S}} : \mathbb{P}^{3} \dashrightarrow \mathbb{P}^{13}_{\left[w_0,\dots , w_{13}\right]}$
whose image is the Enriques-Fano threefold $W_F^{13}=W_{BS}^{13}$ studied in \S~\ref{subsec:bayle13} (see Theorem~\ref{thm:B13=F13}). 

Let us \textit{fix} a general element $\Sigma\in \mathcal{S}$ and let us take its image $S:=\nu_{\mathcal{S}}(\Sigma)$.
Then there exists a hyperplane $H_{12}\cong \mathbb{P}^{12}\subset \mathbb{P}^{13}$ such that $S=W_{F}^{13}\cap H_{12}$. By the generality of $\Sigma$, we may assume that $H_{12}$ does not pass through the singular points $P_1$, $P_2$, $P_3$, $P_4$, $P_1'$, $P_2'$, $P_3'$, $P_4'$, and so that
$H_{12}=\{\sum_{i=0}^{13}a_iw_i=0\}$,
where $a_0,a_1,a_3,a_4,a_9,a_{10},a_{12},a_{13}$ are not equal to zero. In particular, we may suppose $a_0 = 1$.
Let us blow-up $\mathbb{P}^3$ at the vertices of $T$: we obtain a smooth threefold $Y'$ and a birational morphism $bl' : Y'\to \mathbb{P}^{3}$ with exceptional divisors
$E_i := (bl')^{-1}(v_i)$, for $0\le i \le 3$. 
If $H$ denotes the pullback on $Y'$ of the hyperplane class on $\mathbb{P}^{3}$,
the strict transform of $\Sigma$ on $Y'$ is linearly equivalent to $6H-3\sum_{i=0}^{3}E_i$.
Let us blow-up $Y'$ along the strict transforms $\widetilde{l}_{ij}$ of the edges of $T$, for $0\le i<j\le 3$: we obtain a smooth threefold $Y''$ and a birational morphism $bl'' : Y'' \to Y'$ with exceptional divisors 
$F_{ij}:=(bl'')^{-1}(\widetilde{l}_{ij})$, for $0\le i<j \le 3$.
Let $\Sigma''$ be the strict transform on $Y''$ of $\Sigma$. Then $\Sigma'' \sim 6H-3\sum_{i=0}^3 \widetilde{E}_i-2\sum_{0\le i < j \le 3} F_{ij}$, where $\widetilde{E}_i$ denotes the strict transform of $E_i$, for $0\le i\le 3$, and $H$ denotes the pullback $bl''^* H$, by abuse of notation.
Let $\nu'' : Y'' \dashrightarrow W_{F}^{13}\subset \mathbb{P}^{13}$ be the birational map defined by the linear system $|\mathcal{O}_{Y''}(\Sigma'')|$ and let us take $E_{3} := \nu''(F_{23} \cap \Sigma'')$.
The hyperplane sections of $W_F^{13}\subset \mathbb{P}^{13}$ containing $\nu''(F_{23})$ correspond to the divisors on $Y''$ linearly equivalent to
$6H - 3\sum_{i=0}^3\widetilde{E}_i- 3F_{23}- \sum_{\substack{ 0\le i<j \le 3 \\ (i,j)\ne (2,3)}}2F_{ij}.$
Since
$$\left\langle \nu''(F_{23}) \right\rangle = \{w_5=w_6=w_7=w_8=w_9=w_{10}=w_{11}=w_{12}=w_{13}=0\}\cong \mathbb{P}^4,$$
then we have that
$\left\langle E_3 \right\rangle = H_{12} \cap \left\langle \nu''(F_{23}) \right\rangle \cong \mathbb{P}^{3}$, and so that $E_3 = S \cap \left\langle E_3 \right\rangle$, which is defined by the equations
\begin{center}
${w}_{0}+a_1{w}_{1}+a_2 {w}_{2}+a_3 {w}_{3}+a_4{w}_{4} = 0,$

${w}_{5}={w}_{6}={w}_{7}={w}_{8}={w}_{9}={w}_{10}={w}_{11}={w}_{12}={w}_{13}=0,$

${w}_{1} {w}_{3}+a_1{w}_{1} {w}_{4}+a_2 {w}_{2} {w}_{4}+a_3 {w}_{3}{w}_{4}+a_4 {w}_{4}^{2} =0,$ 

${w}_{2}^{2}+a_1{w}_{1} {w}_{4}+a_2 {w}_{2} {w}_{4}+a_3{w}_{3} {w}_{4}+a_4 {w}_{4}^{2}=0.$
\end{center}
Thus, $E_3$ is a quartic elliptic curve, since it is the complete intersection of two quadric surfaces of $\left\langle E_3 \right\rangle \cong \mathbb{P}^3$. 
If $\rho_{\left\langle E_3 \right\rangle} : \mathbb{P}^{13} \dashrightarrow \mathbb{P}^{9}$ denotes the projection of $\mathbb{P}^{13}$ from
$\left\langle E_3 \right\rangle \cong \mathbb{P}^3$, 
then 
$W_{KLM}^9:=\rho_{\left\langle E_3 \right\rangle}(W_F^{13}) \subset \mathbb{P}^9$
is the KLM-EF 3-fold of genus $9$.

\begin{remark}
For the construction of $W_{KLM}^9$, we have \textit{fixed} a general sextic $\Sigma\in\mathcal S$. The hyperplane sections of $W_{KLM}^9\subset \mathbb{P}^9$ correspond to the hyperplane sections of $W_F^{13}\subset \mathbb{P}^{13}$ containing $E_3$, which are images via $\nu_{\mathcal{S}}$ of the sextic surfaces $R$ of $\mathcal{S}$ which are tangent to $\Sigma$ along the two branches of $\Sigma$ intersecting at $l_{23}$. Then $W_{KLM}^9$ is the image of $\mathbb{P}^3$ via the rational map defined by the sublinear system $\mathcal{R}\subset \mathcal{S}$ of these sextic surfaces $R$. 
\end{remark}

By using Macaulay2 we find that the projection map is given by
$$\begin{tikzcd}
\left[w_0 : w_1:w_2:w_3:w_4:w_5:w_6:w_7:w_8:w_9:w_{10}:w_{11}:w_{12} : w_{13} \right] \arrow[d, mapsto, "\rho_{\left\langle E_3 \right\rangle}"]\\
 \left[ w_0+a_1w_1+a_2w_2+a_3w_3+a_4w_4:w_5:w_6:w_7:w_8:w_9:w_{10}:w_{11}:w_{12}:w_{13}\right]
\end{tikzcd}$$
and that the ideal of
$W_{KLM}^9 \subset \mathbb{P}^{9}_{\left[z_0:z_1:z_2:z_3:z_4:z_5:z_6:z_7:z_8 :z_9\right]}$ is generated by the following $16$ polynomials of degree $2$ or $3$ 
\begin{center}
${z}_{6} {z}_{8}-{z}_{5} {z}_{9},\quad {z}_{3} {z}_{8}-{z}_{2}{z}_{9},\quad {z}_{7}^{2}-{z}_{5} {z}_{9}, \quad {z}_{4} {z}_{7}-{z}_{2}{z}_{9},\quad {z}_{3} {z}_{7}-{z}_{1} {z}_{9},\quad {z}_{2} {z}_{7}-{z}_{1}{z}_{8},$ 

${z}_{4} {z}_{6}-{z}_{1} {z}_{9},\quad {z}_{2} {z}_{6}-{z}_{1}{z}_{7},\quad {z}_{4} {z}_{5}-{z}_{1} {z}_{8}, \quad {z}_{3} {z}_{5}-{z}_{1}{z}_{7}, \quad {z}_{2} {z}_{3}-{z}_{1} {z}_{4},$

$\quad$

${z}_{1} {z}_{2}+a_1{z}_{1} {z}_{3}+a_2{z}_{1} {z}_{4}+a_3 {z}_{2} {z}_{4}+a_4 {z}_{3} {z}_{4}-{z}_{0}{z}_{7},$

$\quad$

${z}_{2}^{2} {z}_{9}+a_1{z}_{1} {z}_{4} {z}_{9}+a_2{z}_{2} {z}_{4} {z}_{9}+a_3 {z}_{4}^{2} {z}_{8}+a_4 {z}_{4}^{2} {z}_{9}-{z}_{0} {z}_{8} {z}_{9},$

$\quad$
      
${z}_{1}^{2}{z}_{9}+a_1{z}_{3}^{2} {z}_{6}+a_2 {z}_{1} {z}_{3} {z}_{9}+a_3 {z}_{1} {z}_{4}{z}_{9}+a_4 {z}_{3}^{2} {z}_{9}-{z}_{0} {z}_{6} {z}_{9},$

$\quad$

${z}_{2}^{2} {z}_{5}+a_3 {z}_{2}^{2}{z}_{8}-a_2a_3 {z}_{2} {z}_{4} {z}_{8}-{z}_{0} {z}_{5} {z}_{8}+a_2 {z}_{0}
{z}_{7} {z}_{8}-a_1^2{z}_{1}^{2} {z}_{9}+(a_4-a_2^2-a_1a_3) {z}_{2}^{2} {z}_{9}++2a_1^2a_2 {z}_{1}
{z}_{3} {z}_{9}+a_1(2a_2^2-a_4) {z}_{1} {z}_{4} {z}_{9}+(2a_1a_2a_3-a_2a_4) {z}_{2} {z}_{4} {z}_{9}+2a_1a_2a_4
{z}_{3} {z}_{4} {z}_{9}++a_1{z}_{0} {z}_{5} {z}_{9}-a_1a_2 {z}_{0} {z}_{7}{z}_{9},$

$\quad$

${z}_{1}^{2} {z}_{5}+a_1{z}_{1}^{2} {z}_{6}-{z}_{0} {z}_{5} {z}_{6}+a_2{z}_{1}^{2} {z}_{7}+(a_4-a_1a_2){z}_{1}^{2} {z}_{9}-a_3^2 {z}_{2}^{2} {z}_{9}+a_1a_2a_3 {z}_{1}{z}_{3} {z}_{9}-a_3(a_4-a_2^2) {z}_{1} {z}_{4} {z}_{9}+a_2a_3^2{z}_{2} {z}_{4} {z}_{9}+a_2a_3a_4{z}_{3} {z}_{4} {z}_{9}+a_3 {z}_{0} {z}_{5} {z}_{9}-a_2a_3 {z}_{0} {z}_{7}{z}_{9}.$
\end{center}
Let us take the images of the eight quadruple points of $W_F^{13}$, by denoting them, by abuse of notation, in the following way
$$P_1 :=\rho_{\left\langle E_3 \right\rangle} (P_1') = \{z_k=0 | k\ne 5\},\,\,
P_2 :=\rho_{\left\langle E_3 \right\rangle} (P_2') = \{z_k=0 | k\ne 9\}$$
$$P_3 :=\rho_{\left\langle E_3 \right\rangle} (P_3) = \{z_k=0 | k\ne 6\},\,\,
P_4 :=\rho_{\left\langle E_3 \right\rangle} (P_4) = \{z_k=0 | k\ne 8\},$$
$$P_5 :=\rho_{\left\langle E_3 \right\rangle} (P_1) =\rho_{\left\langle E_3 \right\rangle} (P_2) =\rho_{\left\langle E_3 \right\rangle} (P_3') =\rho_{\left\langle E_3 \right\rangle} (P_4')= \{z_k=0 | k\ne 0\}.$$


\begin{proposition}\label{prop:cones1234KLM}
If $i=1,2,3,4$, the tangent cone $TC_{P_i} W_{KLM}^{9}$ to $W_{KLM}^{9}$ at the point $P_{i}$ is a cone over a Veronese surface.
\end{proposition}
\begin{proof}
Each point $P_i$, for $i=1,2,3,4$, can be viewed as the origin of the open affine set $U_{j(i)} := \{z_{j(i)} \ne 0\}$, where $j(1)=5$, $j(2)=9$, $j(3)=6$, $j(4)=8$. The ideal of the tangent cone $TC_{P_i}(W_{KLM}^{9}\cap U_{j(i)})$ is generated by the minimal degree homogeneous parts of all the polynomials in the ideal of $W_{KLM}^{9} \cap U_{j(i)}$. 
Thanks to Macaulay2, we obtain the following tangent cones.

$TC_{P_1}(W_{KLM}^{9}\cap U_5)$ has ideal generated by 
${z}_{9}$, ${z}_{4}$, ${z}_{3}$, ${z}_{7}^{2}-{z}_{6}{z}_{8}$, ${z}_{2} {z}_{7}-{z}_{1} {z}_{8}$, ${z}_{2} {z}_{6}-{z}_{1}{z}_{7}$, ${z}_{2}^{2}-{z}_{0} {z}_{8}$, ${z}_{1} {z}_{2}-{z}_{0}{z}_{7}$, ${z}_{1}^{2}-{z}_{0} {z}_{6}.$
Hence $TC_{P_1}W_{KLM}^{9}$ is a cone with vertex $P_1$ over a Veronese surface in $\{ z_i = 0 | i = 3,4,5,9\}\cong \mathbb{P}^{5}$.

$TC_{P_2}(W_{KLM}^{9}\cap U_9)$ has ideal generated by 
${z}_{5}$, ${z}_{2}$, ${z}_{1}$, ${z}_{7}^{2}-{z}_{6}{z}_{8}$, ${z}_{4} {z}_{7}-{z}_{3} {z}_{8}$, ${z}_{4} {z}_{6}-{z}_{3}{z}_{7}$, 
$a_4 {z}_{4}^{2}-{z}_{0} {z}_{8}$,  $a_4 {z}_{3} {z}_{4}-{z}_{0}{z}_{7}$, $a_4 {z}_{3}^{2}-{z}_{0} {z}_{6}.$
Hence $TC_{P_2}W_{KLM}^{9}$ is a cone with vertex $P_2$ over a Veronese surface in $\{ z_i = 0 | i = 1,2,5,9 \}\cong \mathbb{P}^{5}$.

$TC_{P_3}(W_{KLM}^{9}\cap U_6)$ has ideal generated by 
${z}_{8}$, ${z}_{4}$, ${z}_{2}$, ${z}_{7}^{2}-{z}_{5}{z}_{9}$, ${z}_{3} {z}_{7}-{z}_{1} {z}_{9}$, ${z}_{3} {z}_{5}-{z}_{1}{z}_{7}$, $a_1{z}_{3}^{2}-{z}_{0} {z}_{9}$, $a_1{z}_{1} {z}_{3}-{z}_{0}{z}_{7}$, $a_1{z}_{1}^{2}-{z}_{0} {z}_{5}.$
Hence $TC_{P_3}W_{KLM}^{9}$ is a cone with vertex $P_3$ over a Veronese surface in $\{ z_i = 0 | i = 2,4,6,8 \}\cong\mathbb{P}^{5}$.

$TC_{P_4}(W_{KLM}^{9}\cap U_8)$ has ideal generated by 
${z}_{6}$, ${z}_{3}$, ${z}_{1}$, ${z}_{7}^{2}-{z}_{5}{z}_{9}$, ${z}_{4} {z}_{7}-{z}_{2} {z}_{9}$, ${z}_{4} {z}_{5}-{z}_{2}{z}_{7}$, $a_3 {z}_{4}^{2}-{z}_{0} {z}_{9}$, $a_3 {z}_{2} {z}_{4}-{z}_{0}{z}_{7}$, $a_3 {z}_{2}^{2}-{z}_{0} {z}_{5}.$
Hence $TC_{P_4}W_{KLM}^{9}$ is a cone with vertex $P_4$ over a Veronese surface in $\{ z_i = 0 | i = 1,3,6,8 \}\cong \mathbb{P}^{5}$.
\end{proof}

\begin{theorem}\label{thm:cone5KLM}
The tangent cone $TC_{P_5} W_{KLM}^{9}$ to $W_{KLM}^{9}$ at the point $P_{5}$ is a cone over a reducible sextic surface $M_6 \subset \mathbb{P}^{7} \subset \mathbb{P}^{9}$, which is given by the union of four planes $\pi_{1}$, $\pi_{2}$, $\pi_{1}'$, $\pi_{2}'$ and a quadric surface $Q \subset \mathbb{P}^{3} \subset \mathbb{P}^{7}$. In particular each one of the planes $\pi_{1}$, $\pi_{2}$, $\pi_{1}'$, $\pi_{2}'$ intersects the quadric surface $Q$ respectively along a line $l_{1}$, $l_{2}$, $l_{1}'$, $l_{2}'$, where $l_i$ is disjoint from $l_i'$, for $i=1,2$. In the other cases the intersections of two of these lines identify four points, which are
$q_{1,2} := l_1 \cap l_2$, $q_{1,2'} := l_1 \cap l_2'$, $q_{1',2} := l_1' \cap l_2$, $q_{1',2'} := l_1' \cap l_2'$.
\end{theorem}
\begin{proof}
The point $P_5$ can be viewed as the origin of the open affine set given by $U_0 := \{z_0 \ne 0\}$. The ideal of the tangent cone $TC_{P_5}(W_{KLM}^{9}\cap U_0)$ is generated by the minimal degree homogeneous parts of all the polynomials in the ideal of $W_{KLM}^{9} \cap U_0$. 
By using Macaulay2 one finds that $TC_{P_5}(W_{KLM}^{9}\cap U_0)$ has ideal generated by the following polynomials
$${z}_{7},\,{z}_{8} {z}_{9},\,{z}_{6} {z}_{9},\,{z}_{5}{z}_{9},\,{z}_{2} {z}_{9},\,{z}_{1} {z}_{9},\,
{z}_{6} {z}_{8},\,{z}_{5}{z}_{8},\,{z}_{3} {z}_{8},\,{z}_{1} {z}_{8},\,{z}_{5} {z}_{6},\,
{z}_{4}{z}_{6},\,{z}_{2} {z}_{6},\,{z}_{4} {z}_{5},\,{z}_{3} {z}_{5},$$
$${z}_{2}{z}_{3}-{z}_{1} {z}_{4}.$$
Hence $TC_{P_5}W_{KLM}^{9}$ is a cone with vertex $P_5$ over a surface $M_6$ contained in the $\mathbb{P}^{7}$ given by $\{ z_i = 0 | i = 0,7\}$. The surface $M_6$ is the union of four planes $\pi_{1}, \pi_{2}, \pi_{1}', \pi_{2}'$ and a quadric surface $Q$, where
\begin{center}
$\pi_{1} := \{z_i = 0 | i=0,1,2,5,6,7,8 \}, \quad \pi_{2} := \{z_i = 0 | i=0,1,3,5,6,7,9 \},$

$\pi_{1}' := \{z_i = 0 | i=0,3,4,6,7,8,9\}, \quad \pi_{2}' := \{z_i = 0 | i=0,2,4,5,7,8,9 \},$

$Q := \{z_i = 0 | i=0,5,6,7,8,9\} \cap \{{z}_{2}{z}_{3}-{z}_{1} {z}_{4}=0\}.$
\end{center}
We obtain the same situation as the one described in Theorem~\ref{thm:cone5p17}, and so a sextic surface $M_6$ as in Figure~\ref{fig:sexticM}.
\end{proof}




By Proposition~\ref{prop:cones1234KLM} and Theorem~\ref{thm:cone5KLM} we have that $P_1$, $P_2$, $P_3$, $P_4$, $P_5$ are non-similar singular points of $W_{KLM}^{9}$. For completeness, let us find their configuration.
Let $l_{i,j}$ be the line joining the singular points $P_i$ and $P_j$ for $1\le i<j \le 5$. Then we have
$l_{1,2}=\{z_k=0| k\ne 5,9\}$, $l_{1,3}=\{z_k=0| k\ne 5,6\}$, $l_{1,4}=\{z_k=0| k\ne 5,8\}$, $l_{1,5}=\{z_k=0| k\ne 0,5\}$, $l_{2,3}=\{z_k=0| k\ne 6,9\}$, $l_{2,4}=\{z_k=0| k\ne 8,9\}$, $l_{2,5}=\{z_k=0| k\ne 0,9\}$, $l_{3,4}=\{z_k=0| k\ne 6,8\}$, $l_{3,5}=\{z_k=0| k\ne 0,6\}$, $l_{4,5}=\{z_k=0| k\ne 0,8\}$.  
By looking at the ideal of $W_{KLM}^9\subset \mathbb{P}^9$, we deduce that the lines $l_{1,3}$, $l_{1,4}$, $l_{1,5}$, $l_{2,3}$, $l_{2,4}$, $l_{2,5}$, $l_{3,5}$, $l_{4,5}$ are contained in $W_{KLM}^{9}$, while $l_{1,4}$ and $l_{2,3}$ are not.
So the five singular points $P_1$, $P_2$, $P_3$, $P_4$, $P_5$ of $W_{KLM}^{9}$ are associated as in 
Table~\ref{tab:proKLM}
of Appendix~\ref{app:congifuration}.

\section{Projective normality}\label{subsec:normality}

Some authors define an \textit{Enriques-Fano threefold} just as a threefold satisfying the following assumption (see for example \cite[Definition 1.3]{KLM11}).

\begin{assumption}[\textbf{*}]
Let $W\subset \mathbb{P}^{N}$ be a non-degenerate threefold whose general hyperplane section $S$ is an Enriques surface and such that $W$ is not a cone over $S$. 
\end{assumption}

If the pair $(W,\mathcal{L}:=|\mathcal{O}_{W}(S)|)$ satisfies Assumption (*), it is enough to take its \textit{normalization} $\nu : W^{\nu} \to W$ to obtain an Enriques-Fano threefold in the general sense, that is $(W^{\nu}, \nu^* \mathcal{L})$. Indeed, an 
element of $\nu^* \mathcal{L}$ is ample, since it is the pullback of a very ample divisor of $\mathcal{L}$ via the finite birational morphism $\nu : W^{\nu} \to W$ (see \cite[Theorem 1.2.13]{Laz}).
Moreover if $(W^\nu, \nu^* \mathcal{L})$ were a (polarized) generalized cone,
$W^\nu$ would contain a $3$-dimensional family of curves of degree $1$ with respect to the given polarization
such that they all pass through a point:
thus, $W\subset \mathbb{P}^N$ would be the union of lines through a point,
contradicting Assumption (*).

Furthermore, we observe that if a pair $(X, \mathcal{L})$ is an Enriques-Fano threefold such that the elements of $\mathcal{L}$ are very ample divisors on $X$, then $\mathcal{L}$ defines an embedding $\phi_{\mathcal{L}} : X \hookrightarrow \mathbb{P}^p$ and $\phi_{\mathcal{L}}(X)\subset \mathbb{P}^p$ is a threefold satisfying Assumption (*).

An example of ``Enriques-Fano threefold'' 
in the sense of 
Assumption (*) is the KLM-EF 3-fold $W_{KLM}^{9} \subset \mathbb{P}^9$: 
instead of proving the normality of this threefold, Knutsen-Lopez-Mu\~{n}oz study properties of its normalization (see \cite[Proposition 13.1]{KLM11}). 
We will see below that the KLM-EF 3-fold actually is (projectively) normal.

Also the rational F-EF 3-folds $W_F^{p=6,7,9,13}\subset \mathbb{P}^p$ 
are ``Enriques-Fano threefold'' in the sense of
Assumption (*): 
indeed,
their normality is unproved, even if Fano assumed normality at the beginning of his work. The normality of the non-rational F-EF 3-fold $W_F^4$ is unproved too; however it does not exactly satisfy Assumption (*), since its hyperplane sections are not Enriques surfaces, but their minimal desingularizations are (see \cite[p.275]{CoDo89}). 
We will see below that the rational F-EF 3-folds of genus $7$, $9$ and $13$ actually are (projectively) normal.

Instead the BS-EF 3-folds and the P-EF 3-folds are normal
by construction, 
since they are quotient of normal threefolds under the action of a finite group defined by a certain involution with a finite number of fixed points (see \cite[Proposition 2.15]{Dre04}).
In particular, the BS-EF 3-folds with very ample hyperplane sections satisfy Assumption (*) in the projective space in which they are embedded, while the other eight BS-EF 3-folds are Enriques-Fano threefolds satisfying exactly the abstract definition. Furthermore, as we saw in \S~\ref{subsec:pro17},~\ref{subsec:pro13} the P-EF 3-folds $W_P^{p=13,17}$ can be embedded in $\mathbb{P}^{p}$ and they also satisfy Assumption (*).

\begin{definition}\label{def:Reye}
Let $\mathcal{R}$ be a $3$-dimensional linear system of quadric sufaces of $\mathbb{P}^3$. Let us suppose that $\mathcal{R}$ is sufficiently general, i.e. $\mathcal{R}$ is base point free and, if $l$ is a double line for $Q \in \mathcal{R}$, then $Q$ is the unique quadric in $\mathcal{R}$ containing $l$. A \textit{Reye congruence} is a surface obtained as the set
$$\{ l \in \mathbb{G}(1,3) | l \text{ is contained in a pencil contained in } \mathcal{R} \},$$ 
where $\mathbb{G}(1,3)$ denotes the Grassmannian variety of lines in $\mathbb{P}^3$. 
\end{definition}

\begin{theorem}\label{thm:projnorm}
Let $W\subset \mathbb{P}^{N}$ be a threefold satisfying Assumption (*). If $S\subset \mathbb{P}^{N-1}$ is linearly normal and if either $N\ge 7$ or $N = 6$ and $S$ is not a Reye congruence, then $h^1(\mathcal{O}_{W})=0$ and $W\subset \mathbb{P}^N$ is projectively normal.
\end{theorem}
\begin{proof}
Since the case where $N=6$ and $S$ is a Reye congruence is excluded, we have that $S\subset \mathbb{P}^{N-1}$ is projectively normal (see \cite[Theorem 1.1]{GLM02}). Thus, by using the arguments of \cite[Lemmas 1.5,1.6,1.7]{CoMu87} (which are inspired by the ones of \cite[pp. 10-11]{Epema}), we obtain that $h^1(\mathcal{O}_{W})=0$ and that $W\subset \mathbb{P}^{N}$ is projectively normal.
\end{proof}

\begin{proposition}\label{prop:linnormS}
Let $W\subset \mathbb{P}^{N}$ be a threefold satisfying Assumption (*). If $W\subset \mathbb{P}^N$ is linearly normal and $h^1(\mathcal{O}_{W})=0$, then $S\subset \mathbb{P}^{N-1}$ is linearly normal.
\end{proposition}
\begin{proof}
We have to show that $h^0(\mathcal{O}_{S}(1))= h^0(\mathcal{O}_{\mathbb{P}^{N-1}}(1))=N$.
This follows by the following exact sequence
$$0 \to \mathcal{O}_{W} \to \mathcal{O}_{W}(1) \to \mathcal{O}_{S}(1) \to 0,$$
since $h^0(\mathcal{O}_W)=1$, $h^1(\mathcal{O}_W)=0$ and $h^0(\mathcal{O}_W(1))=h^0(\mathcal{O}_{\mathbb{P}^N}(1))=N+1$ by hypothesis.
\end{proof} 

\begin{corollary}\label{cor:projnormW}
Let $W\subset \mathbb{P}^{N}$ be a threefold satisfying Assumption (*). If $W\subset \mathbb{P}^N$ is linearly normal and $h^1(\mathcal{O}_{W})=0$, then $W\subset \mathbb{P}^{N}$ is projectively normal (except when $N = 6$ and $S$ is a Reye congruence).
\end{corollary}
\begin{proof}
See Theorem~\ref{thm:projnorm} and Proposition~\ref{prop:linnormS}.
\end{proof} 

\begin{proposition}\label{prop:linnorm}
Let $W\subset \mathbb{P}^{p}$ be a threefold satisfying Assumption (*) such that $p$ is the genus of a curve section of $W$. Then $W\subset \mathbb{P}^p$ and $S\subset \mathbb{P}^{p-1}$ are linearly normal.
\end{proposition}
\begin{proof}
By Riemann-Roch on $S$ we obtain $h^0(\mathcal{O}_{S}(1))=p$.
From $W\subset \mathbb{P}^p$ we have that $h^0(\mathcal{O}_{W}(1))\ge p+1$. On the other hand, from the following exact sequence
$$0 \to \mathcal{O}_{W} \to \mathcal{O}_{W}(1) \to \mathcal{O}_{S}(1) \to 0$$
one gets $h^0(\mathcal{O}_{W}(1))\le p+1$
and hence equality holds.
\end{proof} 

\begin{corollary}\label{cor:projnormKnown}
The following Enriques-Fano threefolds are projectively normal:
$$W_{KLM}^9\subset \mathbb{P}^9, \quad W_F^{p=7,9,13}\subset \mathbb{P}^p, \quad W_{BS}^{p=7,8,9,10,13} \xhookrightarrow{\phi_{\mathcal{L}}} \mathbb{P}^p, \quad W_{P}^{p=13,17}\subset \mathbb{P}^p.$$
\end{corollary}
\begin{proof}
See Theorem~\ref{thm:projnorm} and Proposition~\ref{prop:linnorm}.
\end{proof}

We cannot say the same for the F-EF 3-fold $W_F^{6}\subset \mathbb{P}^6$, since its hyperplane sections are Reye congruences (see \cite[Proposition 3]{Co83} and \cite[\S 3]{Fa38}).
As for the BS-EF 3-fold $W_{BS}^{6}\xhookrightarrow{\phi_{\mathcal{L}}} \mathbb{P}^6$, Macaulay2 enables us to find that its hyperplane section $S\subset \mathbb{P}^5$ is not contained in quadric hypersurfaces of $\mathbb{P}^5$ (see Code~\ref{code:bayle6} of Appendix~\ref{app:code}): this is equivalent to saying that $S\subset \mathbb{P}^{5}$ is projectively normal (use Riemann-Roch and see \cite[Theorem 1.1]{GLM02}), and so we have that
$W_{BS}^6\subset \mathbb{P}^6$ is projectively normal too (see Theorem~\ref{thm:projnorm}).
Thus, we obtain Theorem~\ref{thm:projnormKnown}.

\appendix

\section{Macaulay2 codes}\label{app:code}

In the following we will collect the input codes used in Macaulay2 for the computational analysis of this paper. 
We will essentially use the package \textit{Cremona} of Staglian\`{o} (see \cite{Sta18}).
For more information visit the website
\begin{center}
\href{http://www2.macaulay2.com/Macaulay2/doc/Macaulay2-1.12/share/doc/Macaulay2/Cremona/html/}{http://www2.macaulay2.com/Macaulay2/doc/Macaulay2-1.12/share/doc/Macaulay2/Cremona/html/}
\end{center}

\subsection{Computational analysis of the BS-EF 3-fold of genus 6}\label{code:bayle6} 
\begin{scriptsize}
\begin{verbatim}
Macaulay2, version 1.11
with packages: ConwayPolynomials, Elimination, IntegralClosure, InverseSystems,
               LLLBases, PrimaryDecomposition, ReesAlgebra, TangentCone
i1 : needsPackage "Cremona";
i2 : PP3xPP3 = ZZ/10000019[x_0,x_1,x_2,x_3]**ZZ/10000019[y_0,y_1,y_2,y_3];
i3 : X = ideal{ x_0*y_0-7*x_1*y_1+4*x_2*y_2+2*x_3*y_3, 
     x_0*y_0-6*x_1*y_1+2*x_2*y_2+3*x_3*y_3, x_0*y_0-x_1*y_1-7*x_2*y_2+7*x_3*y_3};
i4 : PP9 = ZZ/10000019[Z_0..Z_9];
i5 : phi = rationalMap map(PP3xPP3,PP9,matrix{{x_0*y_0,x_1*y_1,x_2*y_2,x_3*y_3,x_0*y_1+x_1*y_0, 
     x_0*y_2+x_2*y_0, x_0*y_3+x_3*y_0, x_1*y_2+x_2*y_1, x_1*y_3+x_3*y_1, x_2*y_3+x_3*y_2}});
i6 : (dim(image phi) -1, degree(image phi)) == (6,10)
i7 : image phi == 
     ideal{-2*Z_1*Z_5*Z_6+Z_4*Z_6*Z_7+Z_4*Z_5*Z_8-2*Z_0*Z_7*Z_8+4*Z_0*Z_1*Z_9-Z_4^2*Z_9,
     -2*Z_2*Z_4*Z_6+Z_5*Z_6*Z_7+4*Z_0*Z_2*Z_8-Z_5^2*Z_8+Z_4*Z_5*Z_9-2*Z_0*Z_7*Z_9,
     -4*Z_1*Z_2*Z_6+Z_6*Z_7^2+2*Z_2*Z_4*Z_8-Z_5*Z_7*Z_8+2*Z_1*Z_5*Z_9-Z_4*Z_7*Z_9,
     -2*Z_3*Z_4*Z_5+4*Z_0*Z_3*Z_7-Z_6^2*Z_7+Z_5*Z_6*Z_8+Z_4*Z_6*Z_9-2*Z_0*Z_8*Z_9,
     -4*Z_1*Z_3*Z_5+2*Z_3*Z_4*Z_7-Z_6*Z_7*Z_8+Z_5*Z_8^2+2*Z_1*Z_6*Z_9-Z_4*Z_8*Z_9,
     -4*Z_2*Z_3*Z_4+2*Z_3*Z_5*Z_7+2*Z_2*Z_6*Z_8-Z_6*Z_7*Z_9-Z_5*Z_8*Z_9+Z_4*Z_9^2,
     -4*Z_1*Z_2*Z_3+Z_3*Z_7^2+Z_2*Z_8^2-Z_7*Z_8*Z_9+Z_1*Z_9^2,
     -4*Z_0*Z_2*Z_3+Z_3*Z_5^2+Z_2*Z_6^2-Z_5*Z_6*Z_9+Z_0*Z_9^2,
     -4*Z_0*Z_1*Z_3+Z_3*Z_4^2+Z_1*Z_6^2-Z_4*Z_6*Z_8+Z_0*Z_8^2, 
     -4*Z_0*Z_1*Z_2+Z_2*Z_4^2+Z_1*Z_5^2-Z_4*Z_5*Z_7+Z_0*Z_7^2}
i8 : phiX = ideal{Z_2-Z_3,Z_1-Z_3,Z_0-Z_3,
     -2*Z_1*Z_5*Z_6+Z_4*Z_6*Z_7+Z_4*Z_5*Z_8-2*Z_0*Z_7*Z_8+4*Z_0*Z_1*Z_9-Z_4^2*Z_9,
     -2*Z_2*Z_4*Z_6+Z_5*Z_6*Z_7+4*Z_0*Z_2*Z_8-Z_5^2*Z_8+Z_4*Z_5*Z_9-2*Z_0*Z_7*Z_9,
     -4*Z_1*Z_2*Z_6+Z_6*Z_7^2+2*Z_2*Z_4*Z_8-Z_5*Z_7*Z_8+2*Z_1*Z_5*Z_9-Z_4*Z_7*Z_9,
     -2*Z_3*Z_4*Z_5+4*Z_0*Z_3*Z_7-Z_6^2*Z_7+Z_5*Z_6*Z_8+Z_4*Z_6*Z_9-2*Z_0*Z_8*Z_9,
     -4*Z_1*Z_3*Z_5+2*Z_3*Z_4*Z_7-Z_6*Z_7*Z_8+Z_5*Z_8^2+2*Z_1*Z_6*Z_9-Z_4*Z_8*Z_9,
     -4*Z_2*Z_3*Z_4+2*Z_3*Z_5*Z_7+2*Z_2*Z_6*Z_8-Z_6*Z_7*Z_9-Z_5*Z_8*Z_9+Z_4*Z_9^2,
     -4*Z_1*Z_2*Z_3+Z_3*Z_7^2+Z_2*Z_8^2-Z_7*Z_8*Z_9+Z_1*Z_9^2, 
     -4*Z_0*Z_2*Z_3+Z_3*Z_5^2+Z_2*Z_6^2-Z_5*Z_6*Z_9+Z_0*Z_9^2,
     -4*Z_0*Z_1*Z_3+Z_3*Z_4^2+Z_1*Z_6^2-Z_4*Z_6*Z_8+Z_0*Z_8^2, 
     -4*Z_0*Z_1*Z_2+Z_2*Z_4^2+Z_1*Z_5^2-Z_4*Z_5*Z_7+Z_0*Z_7^2};
i9 : (dim oo -1, degree oo, oo == phi(X) ) == (3, 10, true)
i10 : H6 = ideal{Z_2-Z_3,Z_1-Z_3,Z_0-Z_3};
i11 : PP6 = ZZ/10000019[w_0..w_6];
i12 : inclusion = rationalMap map(PP6,PP9,matrix{{w_0,w_0,w_0,w_0,w_1,w_2,w_3,w_4,w_5,w_6}});
i13 : image oo == H6
i14 : pigreca = phi*(rationalMap map(PP9,PP6, sub(matrix inverseMap(inclusion||H6), PP9) ))
i15 : pigreca(X) == inclusion^*(phiX)
i16 : WB6 = ideal{-2*w_0*w_2*w_3+w_1*w_3*w_4+w_1*w_2*w_5-2*w_0*w_4*w_5+4*w_0^2*w_6-w_1^2*w_6,
      -2*w_0*w_1*w_3+w_2*w_3*w_4+4*w_0^2*w_5-w_2^2*w_5+w_1*w_2*w_6-2*w_0*w_4*w_6,
      -4*w_0^2*w_3+w_3*w_4^2+2*w_0*w_1*w_5-w_2*w_4*w_5+2*w_0*w_2*w_6-w_1*w_4*w_6,
      -2*w_0*w_1*w_2+4*w_0^2*w_4-w_3^2*w_4+w_2*w_3*w_5+w_1*w_3*w_6-2*w_0*w_5*w_6,
      -4*w_0^2*w_2+2*w_0*w_1*w_4-w_3*w_4*w_5+w_2*w_5^2+2*w_0*w_3*w_6-w_1*w_5*w_6,
      -4*w_0^2*w_1+2*w_0*w_2*w_4+2*w_0*w_3*w_5-w_3*w_4*w_6-w_2*w_5*w_6+w_1*w_6^2,
      -4*w_0^3+w_0*w_4^2+w_0*w_5^2-w_4*w_5*w_6+w_0*w_6^2, 
      -4*w_0^3+w_0*w_2^2+w_0*w_3^2-w_2*w_3*w_6+w_0*w_6^2,
      -4*w_0^3+w_0*w_1^2+w_0*w_3^2-w_1*w_3*w_5+w_0*w_5^2, 
      -4*w_0^3+w_0*w_1^2+w_0*w_2^2-w_1*w_2*w_4+w_0*w_4^2};
i17 : WB6 == pigreca(X)
i18 : (dim ooo -1, degree ooo) == (3, 10)
i19 : P1 = ideal{w_1-2*w_0,w_2-2*w_0,w_3-2*w_0,w_4-2*w_0,w_5-2*w_0,w_6-2*w_0};
i20 : P2 = ideal{w_1+2*w_0,w_2+2*w_0,w_3+2*w_0,w_4-2*w_0,w_5-2*w_0,w_6-2*w_0};
i21 : P3 = ideal{w_1+2*w_0,w_2-2*w_0,w_3-2*w_0,w_4+2*w_0,w_5+2*w_0,w_6-2*w_0};
i22 : P4 = ideal{w_1-2*w_0,w_2+2*w_0,w_3+2*w_0,w_4+2*w_0,w_5+2*w_0,w_6-2*w_0};
i23 : P5 = ideal{w_1-2*w_0,w_2+2*w_0,w_3-2*w_0,w_4+2*w_0,w_5-2*w_0,w_6+2*w_0};
i24 : P6 = ideal{w_1+2*w_0,w_2-2*w_0,w_3+2*w_0,w_4+2*w_0,w_5-2*w_0,w_6+2*w_0};
i25 : P7 = ideal{w_1+2*w_0,w_2+2*w_0,w_3-2*w_0,w_4-2*w_0,w_5+2*w_0,w_6+2*w_0};
i26 : P8 = ideal{w_1-2*w_0,w_2-2*w_0,w_3+2*w_0,w_4-2*w_0,w_5+2*w_0,w_6+2*w_0};
i27 : -- let us see if the lines lij joining the points Pi and Pj
      -- are contained in the threefold WB6
      l12 = ideal{(toMap(saturate(P1*P2),1,1)).matrix};
i28 : (l12 + WB6 == l12) == true
i29 : l13 = ideal{(toMap(saturate(P1*P3),1,1)).matrix};
i30 : (l13 + WB6 == l13) == true
i31 : l14 = ideal{(toMap(saturate(P1*P4),1,1)).matrix};
i32 : (l14 + WB6 == l14) == true
i33 : l15 = ideal{(toMap(saturate(P1*P5),1,1)).matrix};
i34 : (l15 + WB6 == l15) == true
i35 : l16 = ideal{(toMap(saturate(P1*P6),1,1)).matrix};
i36 : (l16 + WB6 == l16) == true
i37 : l17 = ideal{(toMap(saturate(P1*P7),1,1)).matrix};
i38 : (l17 + WB6 == l17) == true
i39 : l18 = ideal{(toMap(saturate(P1*P8),1,1)).matrix};
i40 : (l18 + WB6 == l18) == true
i41 : -- etc...
      -- let us now change the coordinates of PP6
      -- in order to have P1 = [1:0...0]
      PP6'=ZZ/10000019[z_0..z_6];
i42 : W' = sub(WB6, {(gens PP6)_0 => (gens PP6')_0, (gens PP6)_1 =>(gens PP6')_1+2*(gens PP6')_0, 
      (gens PP6)_2 =>(gens PP6')_2+2*(gens PP6')_0, (gens PP6)_3 =>(gens PP6')_3+2*(gens PP6')_0, 
      (gens PP6)_4 =>(gens PP6')_4+2*(gens PP6')_0, (gens PP6)_5 =>(gens PP6')_5+2*(gens PP6')_0, 
      (gens PP6)_6 =>(gens PP6')_6+2*(gens PP6')_0})
i43 : W'U0 = sub(oo, {(gens PP6')_0  => 1})
i44 : ConeP1 = sub(tangentCone oo, {(gens PP6')_0 => (gens PP6)_0, 
      (gens PP6')_1 =>(gens PP6)_1-2*(gens PP6)_0, 
      (gens PP6')_2 =>(gens PP6)_2-2*(gens PP6)_0, (gens PP6')_3 =>(gens PP6)_3-2*(gens PP6)_0, 
      (gens PP6')_4 =>(gens PP6)_4-2*(gens PP6)_0, (gens PP6')_5 =>(gens PP6)_5-2*(gens PP6)_0, 
      (gens PP6')_6 =>(gens PP6)_6-2*(gens PP6)_0})
i45 : degree oo == 4
i46 : -- similarly for P2,P3,P4,P5,P6,P7,P8
      -- we observe now that WB6 is not contained in a quadric hypersurface of PP6
      rationalMap toMap(WB6,2,1)
i47 : -- let us also see that a general hyperplane section S is not 
      -- contained in a quadric hypersurface of PP5, where
      -- S = ideal{random(1,PP6)}+WB6
      -- for example:
      S = ideal{w_0-w_1+72*w_2-13*w_3+4*w_4+8*w_5+35*w_6}+WB6
i48: PP5 = ZZ/10000019[t_0..t_5]
i49 : inc = rationalMap map(PP5,PP6,matrix{{t_0-72*t_1+13*t_2-4*t_3-8*t_4-35*t_5,
            t_0,t_1,t_2,t_3,t_4,t_5}})
i50 : image oo == ideal{S_0}
i51 : inc^*S
i52 : (dim oo -1, degree oo) == (2, 10)
i53 : toMap(ooo,2,1)
\end{verbatim}
\end{scriptsize}

\subsection{Computational analysis of the BS-EF 3-fold of genus 7}\label{code:bayle7} 
\begin{scriptsize}
\begin{verbatim}
Macaulay2, version 1.11
with packages: ConwayPolynomials, Elimination, IntegralClosure, InverseSystems,
               LLLBases, PrimaryDecomposition, ReesAlgebra, TangentCone
i1 : needsPackage "Cremona";
i2 : PP1xPP1xPP1xPP1 = 
     ZZ/10000019[x_0,x_1]**ZZ/10000019[y_0,y_1]**ZZ/10000019[z_0,z_1]**ZZ/10000019[t_0,t_1];
i3 : a0001=1;
i4 : a0010=1;
i5 : a0100=1;
i6 : a1000=1;
i7 : a1110=1;
i8 : a1101=1;
i9 : a1011=1;
i10 : a0111=1;
i11 : X = ideal{a0001*x_0*y_0*z_0*t_1+a0010*x_0*y_0*z_1*t_0+a0100*x_0*y_1*z_0*t_0+
      a1000*x_1*y_0*z_0*t_0+a1110*x_1*y_1*z_1*t_0+a1101*x_1*y_1*z_0*t_1+
      a1011*x_1*y_0*z_1*t_1+a0111*x_0*y_1*z_1*t_1};
i12 : phi = rationalMap map(PP1xPP1xPP1xPP1, ZZ/10000019[w_0..w_7], 
      matrix(PP1xPP1xPP1xPP1,{{x_1*y_1*z_1*t_1, x_1*y_0*z_0*t_1, x_0*y_0*z_1*t_1, 
      x_1*y_0*z_1*t_0, x_0*y_0*z_0*t_0, x_0*y_1*z_1*t_0, 
      x_1*y_1*z_0*t_0, x_0*y_1*z_0*t_1}}));
i13 : WB7 = phi(X);
i14 : (dim oo -1, degree oo) == (3,12)
i15 : PP7 = ring WB7;
i16 : P1 = ideal{w_1, w_2, w_3, w_4, w_5, w_6, w_7};
i17 : P2 = ideal{w_0, w_2, w_3, w_4, w_5, w_6, w_7};
i18 : P3 = ideal{w_0, w_1, w_3, w_4, w_5, w_6, w_7};
i19 : P4 = ideal{w_0, w_1, w_2, w_4, w_5, w_6, w_7};
i20 : P1' = ideal{w_0, w_1, w_2, w_3, w_5, w_6, w_7};
i21 : P2' = ideal{w_0, w_1, w_2, w_3, w_4, w_6, w_7};
i22 : P3' = ideal{w_0, w_1, w_2, w_3, w_4, w_5, w_7};
i23 : P4' = ideal{w_0, w_1, w_2, w_3, w_4, w_5, w_6};
i24 : -- let us see if the lines lij joining the points Pi and Pj
      -- are contained in the threefold WB7
      l12 = ideal{(toMap(saturate(P1*P2),1,1)).matrix};
i25 : (l12 + WB7 == l12) == true
i26 : l13 = ideal{(toMap(saturate(P1*P3),1,1)).matrix};
i27 : (l13 + WB7 == l13) == true
i28 : l14 = ideal{(toMap(saturate(P1*P4),1,1)).matrix};
i29 : (l14 + WB7 == l14) == true
i30 : l11' = ideal{(toMap(saturate(P1*P1'),1,1)).matrix};
i31 : (l11' + WB7 == l11') == false
i32 : l12' = ideal{(toMap(saturate(P1*P2'),1,1)).matrix};
i33 : (l12' + WB7 == l12') == true
i34 : l13' = ideal{(toMap(saturate(P1*P3'),1,1)).matrix};
i35 : (l13' + WB7 == l13') == true
i36 : l14' = ideal{(toMap(saturate(P1*P4'),1,1)).matrix};
i37 : (l14' + WB7 == l14') == true
i38 : -- etc...      
      sub(WB7, {(gens PP7)_0=>1});
i39 : ConeP1 = tangentCone oo
i40 : degree oo == 4
i41 : sub(WB7, {(gens PP7)_1=>1});
i42 : ConeP2 = tangentCone oo
i43 : degree oo == 4
i44 : -- etc.. similarly for P3,P4,P5,P1',P2',P3',P4'\end{verbatim}
\end{scriptsize}

\subsection{Computational analysis of the BS-EF 3-fold of genus 9}\label{code:bayle9} 
\begin{scriptsize}
\begin{verbatim}
Macaulay2, version 1.11
with packages: ConwayPolynomials, Elimination, IntegralClosure, InverseSystems,
               LLLBases, PrimaryDecomposition, ReesAlgebra, TangentCone
i1 : needsPackage "Cremona";
i2 : PP5 = ZZ/10000019[x_0, x_1, x_2, y_3, y_4, y_5];
i3 : s1 = x_0^2-3*x_1^2+2*x_2^2;
i4 : s2 = 3*x_0^2-8*x_1^2+5*x_2^2;
i5 : r1 = 3*y_3^2-8*y_4^2+5*y_5^2;
i6 : r2 = y_3^2-3*y_4^2+2*y_5^2;
i7 : X = ideal{s1+r1, s2+r2};
i8 : (dim oo -1, degree oo) == (3,4)
i9 : PP11 = ZZ/10000019[Z_0..Z_11];
i10 : phi = rationalMap map(PP5, PP11, matrix(PP5,{{x_0^2, x_1^2, x_2^2, x_0*x_1, 
      x_0*x_2, x_1*x_2, y_3^2, y_4^2, y_5^2, y_3*y_4, y_3*y_5, y_4*y_5}}));
i11 : phi(X)
i12 : (dim oo -1, degree oo) == (3,16)
i13 : H9 = ideal{ooo_0, ooo_1}
i14 : phi(X) + H9 == phi(X)
i15 : PP9 = ZZ/10000019[w_0..w_9];
i16 : inclusion = rationalMap map(PP9,PP11, matrix(PP9,{{w_0+21*w_4-55*w_5+34*w_6, 
       w_0+8*w_4-21*w_5+13*w_6, w_0, w_1, w_2, w_3, w_4, w_5, w_6, w_7, w_8, w_9 }}));
i17 : image oo == H9
i18 : WB9 = inclusion^* (phi(X));
i19 : (dim oo -1, degree oo) == (3, 16)
i20 : rationalMap map(PP11,PP9, sub(matrix inverseMap(inclusion||H9), PP11))
i21 : pigreca = phi* oo
i23 : fixedPlanex = associatedPrimes (X+ideal{x_0,x_1,x_2});
i24 : fixedPlaney = associatedPrimes (X+ideal{y_3,y_4,y_5});
i25 : P1 = inclusion^* phi(fixedPlaney#0);
i26 : P4 = inclusion^* phi(fixedPlaney#1);
i27 : P2 = inclusion^* phi(fixedPlaney#2);
i28 : P3 = inclusion^* phi(fixedPlaney#3);
i29 : P1' = inclusion^* phi(fixedPlanex#0);
i30 : P4' = inclusion^* phi(fixedPlanex#1);
i31 : P2' = inclusion^* phi(fixedPlanex#2);
i32 : P3' = inclusion^* phi(fixedPlanex#3);
i33 : -- let us see if the lines lij joining the points Pi and Pj
      -- are contained in the threefold WB9
      l12 = ideal{(toMap(saturate(P1*P2),1,1)).matrix};
i34 : (l12 + WB9 == l12) == false
i35 : l13 = ideal{(toMap(saturate(P1*P3),1,1)).matrix};
i36 : (l13 + WB9 == l13) == false
i37 : l14 = ideal{(toMap(saturate(P1*P4),1,1)).matrix};
i38 : (l14 + WB9 == l14) == false
i39 : l11' = ideal{(toMap(saturate(P1*P1'),1,1)).matrix};
i40 : (l11' + WB9 == l11') == true
i41 : l12' = ideal{(toMap(saturate(P1*P2'),1,1)).matrix};
i42 : (l12' + WB9 == l12') == true
i43 : l13' = ideal{(toMap(saturate(P1*P3'),1,1)).matrix};
i44 : (l13' + WB9 == l13') == true
i45 : l14' = ideal{(toMap(saturate(P1*P4'),1,1)).matrix};
i46 : (l14' + WB9 == l14') == true
i47 : -- etc..
      proj1 = rationalMap toMap(P2,1,1);
i48 : proj2 = rationalMap toMap(proj1(P3),1,1);
i49 : proj3 = rationalMap toMap(proj2(proj1(P4)),1,1);
i50 : proj4 = rationalMap toMap(proj3(proj2(proj1(P2'))),1,1);
i51 : proj5 = rationalMap toMap(proj4(proj3(proj2(proj1(P3')))),1,1);
i52 : proj6 = rationalMap toMap(proj5(proj4(proj3(proj2(proj1(P4'))))),1,1);
i53 : proj = proj1*proj2*proj3*proj4*proj5*proj6;
i54 : proj(WB9)
i55 : PP3 = ring oo;
i56 : isBirational( proj|WB9 )
i57 : septics = rationalMap map( PP3, PP9, matrix(inverseMap( proj|WB9 )));
i58 : time image oo == WB9
i59 : comp = associatedPrimes(ideal septics)
i60 : l3' = comp#0;
i61 : l2' = comp#1;
i62 : r21 = comp#2;
i63 : r11 = comp#3;
i64 : r31 = comp#4;
i65 : l1' = comp#5;
i66 : r23 = comp#6;
i67 : r13 = comp#7;
i68 : r33 = comp#8;
i69 : r22 = comp#9;
i70 : r12 = comp#10;
i71 : r32 = comp#11;
i72 : l1 = comp#12;
i73 : l2 = comp#13;
i74 : l3 = comp#14;
i75 : -- trihedron T' :
      f1' = ideal{(gens PP3)_3};
i76 : f2' = ideal{(gens PP3)_1+(gens PP3)_3};
i77 : f3' = ideal{(gens PP3)_2+(gens PP3)_3};
i78 : f1'+f2' == l3'
i79 : f1'+f3' == l2'
i80 : f2'+f3' == l1'
i81 : v'= saturate(f1'+f2'+f3')
i82 : -- trihedron T :
      f1 = ideal{(gens PP3)_0-55*(gens PP3)_1+34*(gens PP3)_2};
i83 : f2 = ideal{(gens PP3)_0  - 21*(gens PP3)_1  +13*(gens PP3)_2};
i84 : f3 = ideal{(gens PP3)_0};
i85 : f1+f2 == l3
i86 : f1+f3 == l2
i87 : f2+f3 == l1
i88 : v = saturate(l1+l2+l3)
i89 : r11 == f1+f1'
i90 : r12 == f1+f2'
i91 : r13 == f1+f3'
i92 : r21 == f2+f1'
i93 : r22 == f2+f2'
i94 : r23 == f2+f3'
i95 : r31 == f3+f1'
i96 : r32 == f3+f2'
i97 : r33 == f3+f3'
i98 : -- general septic surface of the linear system :       
      K = septics^* ideal{random(1,PP9)};
i99 : (dim oo -1, degree oo) == (2,7)
i100 : -- K has double point along l1,l2,l3,l1',l2',l3' :
       (minors(1,jacobian(K))+l1 == l1) == true
i101 : (minors(1,jacobian(K))+l2 == l2) == true
i102 : (minors(1,jacobian(K))+l3 == l3) == true
i103 : (minors(1,jacobian(K))+l1' == l1') == true
i104 : (minors(1,jacobian(K))+l2' == l2') == true
i105 : (minors(1,jacobian(K))+l3' == l3') == true
i106 : -- K has triple point at v and v' :
       (minors(1,jacobian(jacobian(K)))+minors(1,jacobian(K))+v == v) == true
i107 : (minors(1,jacobian(jacobian(K)))+minors(1,jacobian(K))+v' == v') == true
i108 : -- remark
       septics(f1) == P2
i109 : septics(f1') == P2'
i110 : septics(f2) == P3
i111 : septics(f2') == P3'
i112 : septics(f3) == P4
i113 : septics(f3') == P4'
\end{verbatim}
\end{scriptsize}

\subsection{Computational analysis of the BS-EF 3-fold of genus 13}\label{code:bayle13} 
\begin{scriptsize}
\begin{verbatim}
Macaulay2, version 1.11
with packages: ConwayPolynomials, Elimination, IntegralClosure, InverseSystems,
               LLLBases, PrimaryDecomposition, ReesAlgebra, TangentCone
i1 : needsPackage "Cremona";
i2 : PP1x = ZZ/10000019[x_0,x_1];
i3 : PP1y = ZZ/10000019[y_0,y_1];
i4 : PP1z = ZZ/10000019[z_0,z_1];
i5 : X = PP1x ** PP1y ** PP1z;
i6 : use X;
i7 : pigreca = rationalMap map(X, ZZ/10000019[w_0..w_13], matrix{{x_0^2*y_0^2*z_0^2, 
     x_0^2*y_0^2*z_1^2, x_0^2*y_0*y_1*z_0*z_1, x_0^2*y_1^2*z_0^2, x_0^2*y_1^2*z_1^2, 
     x_0*x_1*y_0^2*z_0*z_1, x_0*x_1*y_0*y_1*z_0^2, x_0*x_1*y_0*y_1*z_1^2, 
     x_0*x_1*y_1^2*z_0*z_1, x_1^2*y_0^2*z_0^2, x_1^2*y_0^2*z_1^2, x_1^2*y_0*y_1*z_0*z_1, 
     x_1^2*y_1^2*z_0^2, x_1^2*y_1^2*z_1^2}});
i8 : WB13 = image pigreca;
i9 : (dim oo -1, degree oo) == (3, 24)
i10 : PP13 = ring WB13;
i11 : P1 = pigreca(ideal{x_1,y_0,z_0});
i12 : P2 = pigreca(ideal{x_1,y_1,z_1});
i13 : P3 = pigreca(ideal{x_0,y_1,z_0});
i14 : P4 = pigreca(ideal{x_0,y_0,z_1});
i15 : P1' = pigreca(ideal{x_0,y_1,z_1});
i16 : P2' = pigreca(ideal{x_0,y_0,z_0});
i17 : P3' = pigreca(ideal{x_1,y_0,z_1});
i18 : P4' = pigreca(ideal{x_1,y_1,z_0});
i19 : -- let us see if the lines lij joining the points Pi and Pj
      -- are contained in the threefold WB13
      l12 = ideal{(toMap(saturate(P1*P2),1,1)).matrix};
i20 : (l12 + WB13 == l12) == false
i21 : l13 = ideal{(toMap(saturate(P1*P3),1,1)).matrix};
i22 : (l13 + WB13 == l13) == false
i23 : l14 = ideal{(toMap(saturate(P1*P4),1,1)).matrix};
i24 : (l14 + WB13 == l14) == false
i25 : l11' = ideal{(toMap(saturate(P1*P1'),1,1)).matrix};
i26 : (l11' + WB13 == l11') == false
i27 : l12' = ideal{(toMap(saturate(P1*P2'),1,1)).matrix};
i28 : (l12' + WB13 == l12') == true
i29 : l13' = ideal{(toMap(saturate(P1*P3'),1,1)).matrix};
i30 : (l13' + WB13 == l13') == true
i31 : l14' = ideal{(toMap(saturate(P1*P4'),1,1)).matrix};
i32 : (l14' + WB13 == l14') == true
i33 : -- etc..
      PP3 = ZZ/10000019[t_0..t_3];
i34 : sexties = rationalMap map(PP3,PP13, matrix{{t_0*t_1^3*t_2*t_3, t_0^2*t_1^2*t_2^2,
      t_0^2*t_1^2*t_2*t_3, t_0^2*t_1^2*t_3^2, t_0^3*t_1*t_2*t_3, t_0*t_1^2*t_2^2*t_3,
      t_0*t_1^2*t_2*t_3^2, t_0^2*t_1*t_2^2*t_3, t_0^2*t_1*t_2*t_3^2, t_1^2*t_2^2*t_3^2,
      t_0*t_1*t_2^3*t_3, t_0*t_1*t_2^2*t_3^2, t_0*t_1*t_2*t_3^3, t_0^2*t_2^2*t_3^2}});
i35 : WF13 = image oo
i36 : (WF13 == WB13) == true
\end{verbatim}
\end{scriptsize}

\subsection{Computational analysis of the BS-EF 3-fold of genus 8}\label{code:bayle8} 
\begin{scriptsize}
\begin{verbatim}
Macaulay2, version 1.11
with packages: ConwayPolynomials, Elimination, IntegralClosure, InverseSystems,
               LLLBases, PrimaryDecomposition, ReesAlgebra, TangentCone
i1 : needsPackage "Cremona";
i2 : PP4 = ZZ/10000019[x_0..x_4];
i3 : Q = ideal{x_0^2-x_1^2 -x_2^2+x_3^2};
i4 : R = ideal{2*x_0^2-x_1^2-3*x_2^2+2*x_3^2};
i5 : fixedconic1 = ideal{x_2,x_3,x_4^2-R_0};
i6 : fixedconic2 = ideal{x_0,x_1,x_4^2+R_0};
i7 : four = associatedPrimes (fixedconic1+Q)
i8 : p1 = four#0;
i9 : p2 = four#1;
i10 : p1' = four#2;
i11 : p2' = four#3;
i12 : four' = associatedPrimes (fixedconic2+Q)
i13 : p3 = four'#0;
i14 : p4 = four'#1;
i15 : p3' = four'#2;
i16 : p4' = four'#3;
i17 : PP9 = ZZ/10000019[z_0..z_9];
i18 : phi = rationalMap map(PP4,PP9,matrix{{x_4^2*x_0+x_0*R_0,x_4^2*x_1+x_1*R_0,
      x_4^2*x_2-x_2*R_0,x_4^2*x_3-x_3*R_0,x_4*x_0^2,x_4*x_1^2,
      x_4*x_2^2,x_4*x_3^2,x_4*x_0*x_1,x_4*x_2*x_3}});
i19 : phiY = phi(Q);
i20 : H8 = ideal{phiY_0}
i21 : PP8 = ZZ/10000019[w_0,w_1,w_2,w_3,w_4,w_5,w_6,w_7,w_8];
i22 : inclusion = rationalMap map(PP8,PP9, 
                  matrix(PP8,{{w_0,w_1,w_2,w_3,w_4+w_5-w_6,w_4,w_5,w_6,w_7,w_8}}));
i23 : H8 == image inclusion
i24 : WB8 = inclusion^* phiY;
i25 : (dim oo -1, degree oo) == (3,14)
i26 : P1 = inclusion^* phi(p1)
i27 : P2 = inclusion^* phi(p2)
i28 : P3 = inclusion^* phi(p3)
i29 : P4 = inclusion^* phi(p4)
i30 : P1' = inclusion^* phi(p1')
i31 : P2' = inclusion^* phi(p2')
i32 : P3' = inclusion^* phi(p3')
i33 : P4' = inclusion^* phi(p4')
i34 : -- let us see if the lines lij joining the points Pi and Pj
      -- are contained in the threefold WB8
      l12 = ideal{(toMap(saturate(P1*P2),1,1)).matrix};
i35 : (l12 + WB8 == l12) == true
i36 : l13 = ideal{(toMap(saturate(P1*P3),1,1)).matrix};
i37 : (l13 + WB8 == l13) == true
i38 : l14 = ideal{(toMap(saturate(P1*P4),1,1)).matrix};
i39 : (l14 + WB8 == l14) == true
i40 : l11' = ideal{(toMap(saturate(P1*P1'),1,1)).matrix};
i41 : (l11' + WB8 == l11') == false
i42 : l12' = ideal{(toMap(saturate(P1*P2'),1,1)).matrix};
i43 : (l12' + WB8 == l12') == false
i44 : l13' = ideal{(toMap(saturate(P1*P3'),1,1)).matrix};
i45 : (l13' + WB8 == l13') == true
i46 : l14' = ideal{(toMap(saturate(P1*P4'),1,1)).matrix};
i47 : (l14' + WB8 == l14') == true
i48 : -- etc...   
      proj1 = rationalMap toMap(P1,1,1);
i49 : proj1' = rationalMap toMap(proj1(P1'),1,1);
i50 : proj2 = rationalMap toMap(proj1'(proj1(P2)),1,1);
i51 : proj3 = rationalMap toMap(proj2(proj1'(proj1(P3))),1,1);
i52 : proj3' = rationalMap toMap(proj3(proj2(proj1'(proj1(P3')))),1,1);
i53 : proj = proj1*proj1'*proj2*proj3*proj3'
i54 : isBirational(proj | WB8)
i55 : PP3 = target proj;
i56 : septies = rationalMap map( PP3, PP8, matrix(inverseMap(proj|WB8)) )
i57 : image oo == WB8
i58 : baseL = associatedPrimes ideal septies
i59 : e0= baseL#0;
i60 : l1= baseL#1;
i61 : l2= baseL#2;
i62 : s1= baseL#3;
i63 : s2= baseL#4;
i64 : l2'= baseL#5;
i65 : l1'= baseL#6;
i66 : l0= baseL#7;
i67 : r1= baseL#8;
i68 : t1= baseL#9;
i69 : r2= baseL#10;
i70 : t2= baseL#11;
i71 : C= baseL#12;
i72 : v = saturate(l1+l2+l0);
i73 : q1 = saturate(l1+r1+s1+e0+l2')
i74 : q2 = saturate(l2+r2+s2+e0+l1')
i75 : ar = saturate(r1+r2+l0)
i76 : as = saturate(s1+s2+l0)
i77 : at = saturate(t1+t2+l0)
i78 : a1 = saturate(l1+t1)
i79 : a2 = saturate(l2+t2)
i80 : b1 = saturate(r1+t1+C)
i81 : b2 = saturate(r2+t2+C)
i82 : c1 = saturate(s1+t1)
i83 : c2 = saturate(s2+t2)
i84 : q1' = saturate(l1'+t1)
i85 : q2' = saturate(l2'+t2)
i86 : -- general septic surface of the linear system :       
      N = septies^* ideal{random(1,PP8)};
i87 : (dim oo -1, degree oo) == (2, 7)
i88 : -- N is double along l0,l1,l2,l1',l2',C
      (minors(1,jacobian(N))+ l1 == l1) == true
i89 : (minors(1,jacobian(N))+ l2 == l2) == true
i90 : (minors(1,jacobian(N))+ l2' == l2') == true
i91 : (minors(1,jacobian(N))+ l1' == l1') == true
i92 : (minors(1,jacobian(N))+ l0 == l0) == true
i93 : (minors(1,jacobian(N))+ C == C) == true
i94 : -- N is triple at v
      (minors(1,jacobian(jacobian(N)))+minors(1,jacobian(N))+ v == v) == true
i95 : -- N is quadruple at q1 and q2
      (minors(1,jacobian(jacobian(jacobian(N))))+minors(1,jacobian(jacobian(N)))+
      minors(1,jacobian(N))+ q1 == q1) == true
i96 : (minors(1,jacobian(jacobian(jacobian(N))))+minors(1,jacobian(jacobian(N)))+
       minors(1,jacobian(N))+ q2 == q2) == true
\end{verbatim}
\end{scriptsize}

\subsection{Computational analysis of the BS-EF 3-fold of genus 10}\label{code:bayle10} 
\begin{scriptsize}
\begin{verbatim}
Macaulay2, version 1.11
with packages: ConwayPolynomials, Elimination, IntegralClosure, InverseSystems,
               LLLBases, PrimaryDecomposition, ReesAlgebra, TangentCone
i1 : needsPackage "Cremona";
i2 : PP2=ZZ/10000019[u_0,u_1,u_2];
i3 : PP6 = ZZ/10000019[x_0,x_1,x_2,x_3,x_4,x_5,x_6];
i4 : cubics3points = rationalMap map(PP2, PP6 , matrix{{u_1^2*u_2, 
     u_1*u_2^2, u_0^2*u_2,u_0*u_2^2, u_0^2*u_1,u_0*u_1^2, u_0*u_1*u_2}});
i5 : S6 = image cubics3points
i6 : PP1 = ZZ/10000019[y_0,y_1];
i7 : PP1xPP6= PP1 ** PP6;
i8 : pr2 = rationalMap(PP1xPP6,PP6, matrix{{x_0,x_1,x_2,x_3,x_4,x_5,x_6}});
i9 : PP10 = ZZ/10000019[w_0..w_10];
i10 : phi = rationalMap map(PP1xPP6,PP10, matrix{{y_0^2*x_6,y_0^2*x_0+y_0^2*x_2,
      y_0^2*x_1+y_0^2*x_4,y_0^2*x_3+y_0^2*x_5,y_1^2*x_6,y_1^2*x_0+y_1^2*x_2,
      y_1^2*x_1+y_1^2*x_4,y_1^2*x_3+y_1^2*x_5,y_0*y_1*x_0-y_0*y_1*x_2,
      y_1*y_0*x_1-y_1*y_0*x_4,y_1*y_0*x_3-y_1*y_0*x_5}});
i11 : PP1xS6 = pr2^* S6;
i12 : WB10 = phi(PP1xS6);
i13 : (dim WB10 -1, degree WB10) == (3,18)
i14 : ideal{WB10_0,WB10_1,2*WB10_2,WB10_3,WB10_4,2*WB10_5,WB10_6,WB10_7,WB10_8,
      2*WB10_9,WB10_10,WB10_11,WB10_12,2*WB10_13,WB10_14,2*WB10_15,2*WB10_16,
      4*WB10_17,4*WB10_18,4*WB10_19}
i15 : oo == WB10
i16 : P1 = ideal{w_0,w_1,w_2,w_3,w_5-2*w_4,w_6-2*w_4,w_7-2*w_4,w_8,w_9,w_10};
i17 : P2 = ideal{w_0,w_1,w_2,w_3,w_5-2*w_4,w_6+2*w_4,w_7+2*w_4,w_8,w_9,w_10};
i18 : P3 = ideal{w_0,w_1,w_2,w_3,w_5+2*w_4,w_6-2*w_4,w_7+2*w_4,w_8,w_9,w_10};
i19 : P4 = ideal{w_0,w_1,w_2,w_3,w_5+2*w_4,w_6+2*w_4,w_7-2*w_4,w_8,w_9,w_10};
i20 : P1' = ideal{w_1-2*w_0,w_2-2*w_0,w_3-2*w_0,w_4,w_5,w_6,w_7,w_8,w_9,w_10};
i21 : P2' = ideal{w_1-2*w_0,w_2+2*w_0,w_3+2*w_0,w_4,w_5,w_6,w_7,w_8,w_9,w_10};
i22 : P3' = ideal{w_1+2*w_0,w_2-2*w_0,w_3+2*w_0,w_4,w_5,w_6,w_7,w_8,w_9,w_10};
i23 : P4' = ideal{w_1+2*w_0,w_2+2*w_0,w_3-2*w_0,w_4,w_5,w_6,w_7,w_8,w_9,w_10};
i24 : -- let us see if the lines lij joining the points Pi and Pj
      -- are contained in the threefold WB10
      l12 = ideal{(toMap(saturate(P1*P2),1,1)).matrix};
i25 : (l12 + WB10 == l12) == true
i26 : l13 = ideal{(toMap(saturate(P1*P3),1,1)).matrix};
i27 : (l13 + WB10 == l13) == true
i28 : l14 = ideal{(toMap(saturate(P1*P4),1,1)).matrix};
i29 : (l14 + WB10 == l14) == true
i30 : l11' = ideal{(toMap(saturate(P1*P1'),1,1)).matrix};
i31 : (l11' + WB10 == l11') == true
i32 : l12' = ideal{(toMap(saturate(P1*P2'),1,1)).matrix};
i33 : (l12' + WB10 == l12') == false
i34 : l13' = ideal{(toMap(saturate(P1*P3'),1,1)).matrix};
i35 : (l13' + WB10 == l13') == false
i36 : l14' = ideal{(toMap(saturate(P1*P4'),1,1)).matrix};
i37 : (l14' + WB10 == l14') == false
i38 : -- etc...
      proj1 = rationalMap toMap(P1,1,1);
i39 : proj2 = rationalMap toMap(proj1(P2),1,1);
i40 : proj3 = rationalMap toMap(proj2(proj1(P3)),1,1);
i41 : proj4 = rationalMap toMap(proj3(proj2(proj1(P4))),1,1);
i42 : proj1' = rationalMap toMap(proj4(proj3(proj2(proj1(P1')))),1,1);
i43 : proj2' = rationalMap toMap(proj1'(proj4(proj3(proj2(proj1(P2'))))),1,1);
i44 : proj3' = rationalMap toMap(proj2'(proj1'(proj4(proj3(proj2(proj1(P3')))))),1,1);
i45 : proj = proj1*proj2*proj3*proj4*proj1'*proj2'*proj3'
i46 : isBirational(proj | WB10)
i47 : PP3 = target proj;
i48 : sexties = rationalMap map( PP3, PP10, matrix(inverseMap(proj|WB10)) ) 
i49 : image oo == WB10
i50 : baseL = associatedPrimes ideal sexties
i51 : l23 = baseL#0
i52 : r1 = baseL#1
i53 : l12 = baseL#2
i54 : r3 = baseL#3
i55 : l13 = baseL#4
i56 : r2 = baseL#5
i57 : l02 = baseL#6
i58 : l03 = baseL#7
i59 : l01 = baseL#8
i60 : v1 = baseL#9
i61 : v2 = baseL#10
i62 : v3 = baseL#11
i63 : f0 =ideal{(gens PP3)_0};
i64 : f1 =ideal{(gens PP3)_1+(gens PP3)_2+(gens PP3)_3};
i65 : f2=ideal{(gens PP3)_1-(gens PP3)_2+(gens PP3)_3};
i66 : f3 =ideal{(gens PP3)_1+(gens PP3)_2-(gens PP3)_3};
i67 : plane = ideal{(gens PP3)_1-(gens PP3)_2-(gens PP3)_3};
i68 : l12 == f1+f2
i69 : l13 == f1+f3
i70 : l23 == f2+f3
i71 : l01 == f0+f1
i72 : l02 == f0+f2
i73 : l03 == f0+f3
i74 : r1 == plane+f1
i75 : r2 == plane+f2
i76 : r3 == plane+f3
i77 : v0 = f1+f2+f3+plane
i78 : v1 == f0+f2+f3
i79 : v2 == f0+f1+f3
i80 : v3 == f0+f1+f2
i81 : q1 = saturate(l01+r1)
i82 : q2 = saturate(l02+r2)
i83 : q3 = saturate(l03+r3)
i84 : -- general element of the linear system defining sexties :
       M = sexties^* ideal{random(1,PP10)};
i85 : (dim oo -1, degree oo)
i86 : -- M has double points along r1,r2,r3 :     
      (minors(1,jacobian(M))+r1 == r1) == true
i87 : (minors(1,jacobian(M))+r2 == r2) == true
i88 : (minors(1,jacobian(M))+r3 == r3) == true
i89 : -- M has triple points at v1,v2,v3  : 
      (minors(1,jacobian(jacobian(M)))+minors(1,jacobian(M))+ v1 == v1) == true
i90 : (minors(1,jacobian(jacobian(M)))+minors(1,jacobian(M))+ v2 == v2) == true 
i91 : (minors(1,jacobian(jacobian(M)))+minors(1,jacobian(M))+ v3 == v3) == true 
i92 : -- v0 is a quadruple point of M :
      (minors(1,jacobian(jacobian(jacobian(M))))+minors(1,jacobian(jacobian(M)))+
       minors(1,jacobian(M))+ v0 == v0) == true
\end{verbatim}
\end{scriptsize}

\subsection{Computational analysis of the P-EF 3-fold of genus 17}\label{code:pro17}

\begin{scriptsize}
\begin{verbatim}
Macaulay2, version 1.11
with packages: ConwayPolynomials, Elimination, IntegralClosure, InverseSystems,
               LLLBases, PrimaryDecomposition, ReesAlgebra, TangentCone
i1 : needsPackage "Points";
i2 : needsPackage "Cremona";
i3 : PP1 = ZZ/10000019[u_0,u_1];
i4 : PP1'= ZZ/10000019[v_0,v_1];
i5 : P1P1 = PP1 ** PP1';
i6 : PP9 = ZZ/10000019[y_{0,0},y_{0,1},y_{0,2},y_{1,0},y_{1,1},
                       y_{1,2},y_{2,0},y_{2,1},y_{2,2},x];
i7 : antiCanonicalEmbeddingP = rationalMap map(P1P1,PP9, matrix{{u_1^2*v_1^2,
     u_1^2*v_0*v_1, u_1^2*v_0^2,u_1*u_0*v_1^2,u_1*u_0*v_0*v_1,u_1*u_0*v_0^2,
     u_0^2*v_1^2,u_0^2*v_0*v_1,u_0^2*v_0^2,0}});
i8 : P = image oo
i9 : (dim P -1, degree P) == (2,8)
i10 : numgens P
i11 : V = ideal{P_1,P_2,P_3,P_4,P_5,P_6,P_7,P_8,P_9,P_10,P_11,P_12,
                P_13,P_14,P_15,P_16,P_17,P_18,P_19,P_20}
i12 : (dim V -1, degree V) == (3, 8)
i13 : PP29 = ZZ/10000019[Z_0..Z_29];
i14 : phi = rationalMap map(PP9, PP29, matrix(PP9, {{y_{1,1}^2, y_{0,0}^2, y_{0,2}^2, 
      y_{2,0}^2, y_{2,2}^2, x^2, y_{0,1}^2, y_{1,0}^2, y_{1,2}^2, y_{2,1}^2, y_{0,1}*x, 
      y_{1,0}*x, y_{1,2}*x, y_{2,1}*x, y_{0,0}*y_{1,1}, y_{0,2}*y_{1,1}, y_{2,0}*y_{1,1}, 
      y_{2,2}*y_{1,1}, y_{0,1}*y_{1,0}, y_{0,1}*y_{1,2}, y_{1,0}*y_{2,1}, y_{1,2}*y_{2,1}, 
      y_{0,0}*y_{0,2}, y_{0,0}*y_{2,0}, y_{0,2}*y_{2,2}, y_{2,0}*y_{2,2}, y_{0,1}*y_{2,1}, 
      y_{0,0}*y_{2,2}, y_{0,2}*y_{2,0}, y_{1,0}*y_{1,2} }}));
i15 : phi(V)
i16 : H17 = ideal{Z_18-Z_14, Z_19-Z_15, Z_20-Z_16, Z_21-Z_17, Z_22-Z_6, Z_23-Z_7, 
      Z_24-Z_8, Z_25-Z_9, Z_26-Z_0, Z_27-Z_0, Z_28-Z_0, Z_29-Z_0};
i17 : phi(V) + H17 == phi(V)
i18 : PP17=ZZ/10000019[z_0..z_17];
i19 : inclusion = rationalMap map(PP17, PP29, matrix(PP17, {{z_0,z_1,z_2,z_3,z_4,z_5,z_6,z_7,
      z_8,z_9,z_10,z_11,z_12,z_13,z_14,z_15,z_16,z_17, z_14,z_15,z_16,z_17,z_6,z_7,z_8,
      z_9,z_0,z_0,z_0,z_0 }}));
i20 : image oo == H17
i21 : WP17 = inclusion^* (phi(V))
i22 : (dim oo -1, degree oo) == (3, 32)
i23 : pigreca = rationalMap map(PP9,PP17, matrix(PP9, {{y_{1,1}^2, y_{0,0}^2, y_{0,2}^2, 
      y_{2,0}^2, y_{2,2}^2, x^2, y_{0,1}^2, y_{1,0}^2, y_{1,2}^2, y_{2,1}^2, y_{0,1}*x, 
      y_{1,0}*x, y_{1,2}*x, y_{2,1}*x, y_{0,0}*y_{1,1}, y_{0,2}*y_{1,1}, 
      y_{2,0}*y_{1,1}, y_{2,2}*y_{1,1} }}));
i24 : pigreca(V) == WP17
i25 : sub(WP17, {(gens PP17)_1=>1});
i26 : ConeP1 = tangentCone oo
i27 : degree oo == 4
i28 : sub(WP17, {(gens PP17)_2=>1});
i29 : ConeP2 = tangentCone oo
i30 : degree oo == 4
i31 : sub(WP17, {(gens PP17)_3=>1});
i32 : ConeP3 = tangentCone oo
i33 : degree oo == 4
i34 : sub(WP17, {(gens PP17)_4=>1});
i35 : ConeP4 = tangentCone oo
i36 : degree oo == 4
i37 : sub(WP17, {(gens PP17)_5=>1});
i38 : ConeP5 = tangentCone oo
i39 : degree oo == 6
i40 : M6 = ConeP5+ideal{(gens PP17)_5}
i41 : time irredCompM6 = associatedPrimes M6;
i42 : plane1 = irredCompM6#0
i43 : plane2 = irredCompM6#1
i44 : plane1' = irredCompM6#2
i45 : plane2' = irredCompM6#3
i46 : Q = irredCompM6#4
i47 : line1 = Q+plane1;
i48 : line1' = Q+plane1';
i49 : line2 = Q+plane2;
i50 : line2' = Q+plane2';
i51 : (dim(line1+line1')-1) == -1
i52 : (dim(line2+line2')-1) == -1
i53 : q12 = saturate(line1+line2)
i54 : q12' = saturate(line1+line2')
i55 : q1'2 = saturate(line1'+line2)
i56 : q1'2' = saturate(line1'+line2')
\end{verbatim}
\end{scriptsize}

\subsection{Computational analysis of the P-EF 3-fold of genus 13}\label{code:pro13}
\begin{scriptsize}
\begin{verbatim}
Macaulay2, version 1.11
with packages: ConwayPolynomials, Elimination, IntegralClosure, InverseSystems,
               LLLBases, PrimaryDecomposition, ReesAlgebra, TangentCone
i1 : needsPackage "Cremona";
i2 : PP2=ZZ/10000019[u_0,u_1,u_2];
i3 : a1 = ideal{u_1,u_2};
i4 : a2 = ideal{u_0,u_2};
i5 : a3 = ideal{u_0,u_1};
i6 : cubics3points = rationalMap toMap(saturate(a1*a2*a3),3,1);
i7 : DelPezzo6ic = image cubics3points
i8 : (dim DelPezzo6ic -1, degree DelPezzo6ic)
i9 : PP6 = ring DelPezzo6ic;
i10 : PP7 = ZZ/10000019[x_0,x_1,x_2,x_3,x_4,x_5,x_6,y];
i11 : inclusion = rationalMap map(PP6,PP7, matrix{{(gens PP6)_0,(gens PP6)_1,
      (gens PP6)_2,(gens PP6)_3,(gens PP6)_4,(gens PP6)_5,(gens PP6)_6,0}});
i12 : S6 = inclusion(DelPezzo6ic)
i13 : numgens S6 == 10
i14 : V = ideal{S6_1,S6_2,S6_3,S6_4,S6_5,S6_6,S6_7,S6_8,S6_9};
i15 : (dim V -1, degree V) == (3, 6)
i16 : F1 = ideal{(gens PP7)_0+(gens PP7)_2, (gens PP7)_1+(gens PP7)_4,
      (gens PP7)_3+(gens PP7)_5, (gens PP7)_6};
i17 : v = (associatedPrimes (F1+V))#0 
i18 : oo == ideal{x_0,x_1,x_2,x_3,x_4,x_5,x_6}      
i19 : F2 = ideal{(gens PP7)_0-(gens PP7)_2, (gens PP7)_1-(gens PP7)_4, 
           (gens PP7)_3-(gens PP7)_5, y};
i20 : F2intV = associatedPrimes saturate(F2+V);
i21 : v1 = F2intV#0
i22 : v2 = F2intV#3
i23 : v3 = F2intV#2
i24 : v4 = F2intV#1
i25 : PP13 = ZZ/10000019[z_0..z_13];
i26 : pigreco = rationalMap map(PP7,PP13, matrix{{x_6^2, x_0^2+x_2^2, x_1^2+x_4^2, x_3^2+x_5^2,    
      (x_0+x_2)*x_6, (x_1+x_4)*x_6, (x_3+x_5)*x_6, x_0*x_1+x_2*x_4, x_2*x_3+x_0*x_5,
      x_1*x_3+x_4*x_5, (x_0-x_2)*y, (x_1-x_4)*y, (x_3-x_5)*y, y^2}});
i27 : PP19 = ZZ/10000019[Z_0..Z_19];
i28 : phi = rationalMap map(PP7,PP19,matrix{{x_6^2, x_0^2+x_2^2, x_1^2+x_4^2,
      x_3^2+x_5^2, (x_0+x_2)*x_6, (x_1+x_4)*x_6, (x_3+x_5)*x_6, x_0*x_1+x_2*x_4, 
      x_2*x_3+x_0*x_5, x_1*x_3+x_4*x_5, (x_0-x_2)*y, (x_1-x_4)*y, (x_3-x_5)*y, y^2, 
      2*x_0*x_2, 2*x_1*x_4, 2*x_3*x_5, x_4*x_3+x_1*x_5, x_0*x_3+x_2*x_5, x_1*x_2+x_0*x_4}});
i29 : phi(V)
i30 : phiV = sub(phi(V), {Z_14 => 2*Z_0,Z_15 => 2*Z_0,Z_16 => 2*Z_0, Z_19 => Z_6,
      Z_18 => Z_5, Z_17 => Z_4})
i31 : PP13' = ZZ/10000019[Z_0..Z_13];
i32 : ideal(submatrix(gens (sub(ooo, PP13')), {6..47}))
i33 : WP13 = sub(oo, { (gens PP13')_0=>(gens PP13)_0, (gens PP13')_1=>(gens PP13)_1, 
      (gens PP13')_2=>(gens PP13)_2, (gens PP13')_3=>(gens PP13)_3, 
      (gens PP13')_4=>(gens PP13)_4, (gens PP13')_5=>(gens PP13)_5, 
      (gens PP13')_6=>(gens PP13)_6, (gens PP13')_7=>(gens PP13)_7,
      (gens PP13')_8=>(gens PP13)_8, (gens PP13')_9=>(gens PP13)_9, 
      (gens PP13')_10=>(gens PP13)_10, (gens PP13')_11=>(gens PP13)_11, 
      (gens PP13')_12=>(gens PP13)_12, (gens PP13')_13=>(gens PP13)_13 })
i34 : (dim oo -1, degree oo) == (3, 24)
i35 : WP13 == pigreco(V)
i36 : P1 = ideal{z_1 -2*z_0,z_2 -2*z_0,z_3 -2*z_0,z_4 -2*z_0,z_5 -2*z_0,
       z_6 -2*z_0,z_7 -2*z_0,z_8 -2*z_0,z_9 -2*z_0,z_10,z_11,z_12,z_13};
i37 : P2 = ideal{z_1 -2*z_0,z_2 -2*z_0,z_3 -2*z_0,z_4 -2*z_0,z_5 +2*z_0,
      z_6 +2*z_0,z_7 +2*z_0,z_8 +2*z_0,z_9 -2*z_0,z_10,z_11,z_12,z_13};
i38 : P3 = ideal{z_1 -2*z_0,z_2 -2*z_0,z_3 -2*z_0,z_4 +2*z_0,z_5 -2*z_0,
      z_6 +2*z_0,z_7 +2*z_0,z_8 -2*z_0,z_9 +2*z_0,z_10,z_11,z_12,z_13};
i39 : P4 = ideal{z_1 -2*z_0,z_2 -2*z_0,z_3 -2*z_0,z_4 +2*z_0,z_5 +2*z_0,
      z_6 -2*z_0,z_7 -2*z_0,z_8 +2*z_0,z_9 +2*z_0,z_10,z_11,z_12,z_13};
i40 : P5 = pigreco(v);
i41 : -- let us see if the lines lij joining the points Pi and Pj
      -- are contained in the threefold WP13
      l12 = ideal{(toMap(saturate(P1*P2),1,1)).matrix};
i42 : (l12 + WP13 == l12 ) == false
i43 : l13 = ideal{(toMap(saturate(P1*P3),1,1)).matrix};
i44 : (l13 + WP13 == l13) == false
i45 : l14 = ideal{(toMap(saturate(P1*P4),1,1)).matrix};
i46 : (l14 + WP13 == l14 ) == false
i47 : l15 = ideal{(toMap(saturate(P1*P5),1,1)).matrix};
i48 : (l15 + WP13 == l15) == true
i49 : l23 = ideal{(toMap(saturate(P2*P3),1,1)).matrix};
i50 : (l23 + WP13 == l23) == false
i51 : l24 = ideal{(toMap(saturate(P2*P4),1,1)).matrix};
i52 : (l24 + WP13 == l24) == false
i53 : l25 = ideal{(toMap(saturate(P2*P5),1,1)).matrix};
i54 : (l25 + WP13 == l25) == true
i55 : l34 = ideal{(toMap(saturate(P3*P4),1,1)).matrix};
i56 : (l34 + WP13 == l34) == false
i57 : l35 = ideal{(toMap(saturate(P3*P5),1,1)).matrix};
i58 : (l35 + WP13 == l35) == true
i59 : l45 = ideal{(toMap(saturate(P4*P5),1,1)).matrix};
i60 : (l45 + WP13 == l45) == true
i61 : W' = sub(WP13, {(gens PP13)_0=>(gens PP13')_0,
      (gens PP13)_1=>(gens PP13')_1+2*(gens PP13')_0,
      (gens PP13)_2=>(gens PP13')_2+2*(gens PP13')_0,
      (gens PP13)_3=>(gens PP13')_3+2*(gens PP13')_0,
      (gens PP13)_4=>(gens PP13')_4+2*(gens PP13')_0,
      (gens PP13)_5=>(gens PP13')_5+2*(gens PP13')_0, 
      (gens PP13)_6=>(gens PP13')_6+2*(gens PP13')_0,
      (gens PP13)_7=>(gens PP13')_7+2*(gens PP13')_0, 
      (gens PP13)_8=>(gens PP13')_8+2*(gens PP13')_0,
      (gens PP13)_9=>(gens PP13')_9+2*(gens PP13')_0, 
      (gens PP13)_10=>(gens PP13')_10, (gens PP13)_11=>(gens PP13')_11, 
      (gens PP13)_12=>(gens PP13')_12, (gens PP13)_13=>(gens PP13')_13});
i62 : W'U0 = sub(oo, {(gens PP13')_0  => 1});
i63 : tangentCone W'U0 == ideal{-9*Z_1+8*Z_7+8*Z_8-4*Z_9, -9*Z_2+8*Z_7-4*Z_8+8*Z_9,
      -9*Z_3-4*Z_7+8*Z_8+8*Z_9, -9*Z_4+2*Z_7+2*Z_8-Z_9, -9*Z_5+2*Z_7-Z_8+2*Z_9, 
      -9*Z_6-Z_7+2*Z_8+2*Z_9, Z_10-Z_11+Z_12, 9*Z_12^2-(-4*Z_7+8*Z_8+8*Z_9)*Z_13, 
      9*Z_11^2-(8*Z_7-4*Z_8+8*Z_9)*Z_13, 9*Z_11*Z_12+(2*Z_7+2*Z_8-10*Z_9)*Z_13, 
      (2*Z_7-10*Z_8+2*Z_9)*Z_11+(-10*Z_7+2*Z_8+2*Z_9)*Z_12, 
      (6*Z_7-6*Z_8-18*Z_9)*Z_11+(6*Z_7-6*Z_8+18*Z_9)*Z_12, 
      Z_7^2-2*Z_7*Z_8+Z_8^2-2*Z_7*Z_9-2*Z_8*Z_9+Z_9^2}
i64 : degree (tangentCone W'U0) == 4 
i65 : sub(WP13, {(gens PP13)_13=>1});
i66 : ConeP5 = tangentCone oo
i67 : degree oo == 5
i68 : TC0W'U13 = ideal{ z_6-z_7, z_5-z_8, z_4-z_9, z_2-z_3, z_1-z_3, 2*z_0-z_3,
      z_9*z_10-z_8*z_11+z_7*z_12, z_8*z_10-z_9*z_11+z_3*z_12,
      z_7*z_10-z_3*z_11+z_9*z_12, z_3*z_10-z_7*z_11+z_8*z_12,
      z_8^2-z_9^2, z_7*z_8-z_3*z_9, z_3*z_8-z_7*z_9,
      z_7^2-z_9^2, z_3*z_7-z_8*z_9, z_3^2-z_9^2 }
i69 : (ConeP5 == oo ) == true
i70 : M5 = ConeP5+ideal{(gens PP13)_13}
i71 : (dim oo -1, degree oo) == (2, 5)
i72 : irredCompM5 = associatedPrimes M5;
i73 : plane0=irredCompM5#0
i74 : plane1=irredCompM5#1
i75 : plane2=irredCompM5#2
i76 : plane3=irredCompM5#3
i77 : plane4=irredCompM5#4
i78 : (dim(plane0+plane1)-1, degree (plane0+plane1)) == (1,1)
i79 : (dim(plane0+plane2)-1, degree (plane0+plane2)) == (1,1)
i80 : (dim(plane0+plane3)-1, degree (plane0+plane3)) == (1,1)
i81 : (dim(plane0+plane4)-1, degree (plane0+plane4)) == (1,1)
i82 : (dim(plane1+plane2)-1, degree (plane1+plane2)) == (0,1)
i83 : (dim(plane1+plane3)-1, degree (plane1+plane3)) == (0,1)
i84 : (dim(plane1+plane4)-1, degree (plane1+plane4)) == (0,1)
i85 : (dim(plane2+plane3)-1, degree (plane2+plane3)) == (0,1)
i86 : (dim(plane2+plane4)-1, degree (plane2+plane4)) == (0,1)
i87 : (dim(plane3+plane4)-1, degree (plane3+plane4)) == (0,1)
\end{verbatim}
\end{scriptsize}

\subsection{Computational analysis of the KLM-EF 3-fold of genus 9}\label{code:KLM}
\begin{scriptsize}
\begin{verbatim}
Macaulay2, version 1.11
with packages: ConwayPolynomials, Elimination, IntegralClosure, InverseSystems,
               LLLBases, PrimaryDecomposition, ReesAlgebra, TangentCone
i1 : needsPackage "Cremona";
i2 : PP3 = ZZ/10000019[t_0..t_3];
i3 : PP13 = ZZ/10000019[w_0..w_13];
i4 : sexties = rationalMap map(PP3,PP13, matrix{{t_0*t_1^3*t_2*t_3, t_0^2*t_1^2*t_2^2,
     t_0^2*t_1^2*t_2*t_3, t_0^2*t_1^2*t_3^2, t_0^3*t_1*t_2*t_3, t_0*t_1^2*t_2^2*t_3,
     t_0*t_1^2*t_2*t_3^2, t_0^2*t_1*t_2^2*t_3, t_0^2*t_1*t_2*t_3^2, t_1^2*t_2^2*t_3^2,
     t_0*t_1*t_2^3*t_3, t_0*t_1*t_2^2*t_3^2, t_0*t_1*t_2*t_3^3, t_0^2*t_2^2*t_3^2}});
i5 : WF13 = image sexties
i6 : (dim WF13 -1, degree WF13) == (3, 24)
i7 : P1 = ideal{w_0,w_1,w_2,w_3,w_5,w_6,w_7,w_8,w_9,w_10,w_11,w_12,w_13};
i8 : tangentCone sub(WF13, {(gens PP13)_4=>1})
i9 : degree oo == 4
i10 : P2 = ideal{w_1,w_2,w_3,w_4,w_5,w_6,w_7,w_8,w_9,w_10,w_11,w_12,w_13};
i11 : tangentCone sub(WF13, {(gens PP13)_0=>1})
i12 : degree oo == 4
i13 : P3 = ideal{w_0,w_1,w_2,w_3,w_4,w_5,w_6,w_7,w_8,w_9,w_11,w_12,w_13};
i14 : tangentCone sub(WF13, {(gens PP13)_10=>1})
i15 : degree oo == 4
i16 : P4 = ideal{w_0,w_1,w_2,w_3,w_4,w_5,w_6,w_7,w_8,w_9,w_10,w_11,w_13};
i17 : tangentCone sub(WF13, {(gens PP13)_12=>1})
i18 : degree oo == 4
i19 : P1' = ideal{w_0,w_1,w_2,w_3,w_5,w_4,w_6,w_7,w_8,w_10,w_11,w_12,w_13};
i20 : tangentCone sub(WF13, {(gens PP13)_9=>1})
i21 : degree oo == 4
i22 : P2' = ideal{w_0,w_1,w_2,w_3,w_4,w_5,w_6,w_7,w_8,w_9,w_10,w_11,w_12};
i23 : tangentCone sub(WF13, {(gens PP13)_13=>1})
i24 : degree oo == 4
i25 : P3' = ideal{w_0,w_1,w_2,w_4,w_5,w_6,w_7,w_8,w_9,w_10,w_11,w_12,w_13};
i26 : tangentCone sub(WF13, {(gens PP13)_3=>1})
i27 : degree oo == 4
i28 : P4' = ideal{w_0,w_2,w_3,w_4,w_5,w_6,w_7,w_8,w_9,w_10,w_11,w_12,w_13};
i29 : tangentCone sub(WF13, {(gens PP13)_1=>1})
i30 : degree oo == 4
i31 : J = jacobian((map sexties).matrix);
i32 : JJ = jacobian(J);
i33 : JJl23 = sub(JJ,{(gens PP3)_2=> 0, (gens PP3)_3 =>0})
i34 : SPANnuF23 = ideal{w_5,w_6,w_7,w_8,w_9,w_10,w_11,w_12,w_13};
i35 : -- H12 = ideal{random(1,PP13)};
      -- for example
      H12 = ideal{w_0+11*w_1+2*w_2+3*w_3+5*w_4+4*w_5+6*w_6-7*w_7-8*w_8-9*w_9+
      10*w_10-11*w_11+12*w_12+13*w_13};
i36 : S = H12+WF13;
i37 : E3 = saturate(S+SPANnuF23)
i38 : (dim oo -1, degree oo, genus oo) == (1, 4, 1)
i39 : SPANE3 = ideal{E3_0,E3_1,E3_2,E3_3,E3_4,E3_5,E3_6,E3_7,E3_8,E3_9};
i40 : PP9 = ZZ/10000019[z_0..z_9]; 
i41 : projE3 = rationalMap map(PP13,PP9, matrix{{SPANE3_9,SPANE3_8,SPANE3_7,SPANE3_6,
      SPANE3_5,SPANE3_4,SPANE3_3,SPANE3_2,SPANE3_1,SPANE3_0}})
i42 : KLM = projE3(WF13)
i43 : (dim oo -1, degree oo) == (3, 16)
i44 : isBirational((projE3|WF13)||KLM) == true
i45 : sub(KLM, {(gens PP9)_5=>1});
i46 : Conep1 = tangentCone oo
i47 : degree oo
i48 : sub(KLM, {(gens PP9)_9=>1});
i49 : Conep2 = tangentCone oo
i50 : degree oo
i51 : sub(KLM, {(gens PP9)_6=>1});
i52 : Conep3 = tangentCone oo
i53 : degree oo
i54 : sub(KLM, {(gens PP9)_8=>1});
i55 : Conep4 = tangentCone oo
i56 : degree oo
i57 : sub(KLM, {(gens PP9)_0=>1});
i58 : Conep5 = tangentCone oo
i59 : degree oo
i60 : M6 = Conep5+ideal{(gens PP9)_0}
i61 : irredCompM6 = associatedPrimes M6;
i62 : plane1 = irredCompM6#0
i63 : plane2 = irredCompM6#1
i64 : plane2' = irredCompM6#2
i65 : plane1' = irredCompM6#3
i66 : Q = irredCompM6#4
i67 : line1 = Q+plane1;
i68 : line1' = Q+plane1';
i69 : line2 = Q+plane2;
i70 : line2' = Q+plane2';
i71 : dim(line1+line1')-1 == -1
i72 : dim(line2+line2')-1 == -1
i73 : q12 = saturate(line1+line2)
i74 : q12' = saturate(line1+line2')
i75 : q1'2 = saturate(line1'+line2)
i76 : q1'2' = saturate(line1'+line2')
\end{verbatim}
\end{scriptsize}

\newpage

\section{Configurations of the singularities of some Enriques-Fano threefolds of genus 6, 7, 8, 9, 10, 13, 17}\label{app:congifuration}

In this Appendix we graphically represent the configurations of the singular points of the F-EF 3-folds $W_F^{p=6,7,9,13}$, the BS-EF 3-folds $W_{BS}^{p=6,7,8,9,10,13}$, the P-EF 3-folds $W_P^{p=13,17}$ and the KLM-EF 3-fold $W_{KLM}^9$.

\begin{table}[ht]\scriptsize
\begin{tabular}{|l|l|} 
\hline
  & \\ 
$W_{BS}^{6}$, $W_{F}^{6}$ & $W_{BS}^{7}$, $W_{F}^{7}$\\ 
\includegraphics[scale=0.35]{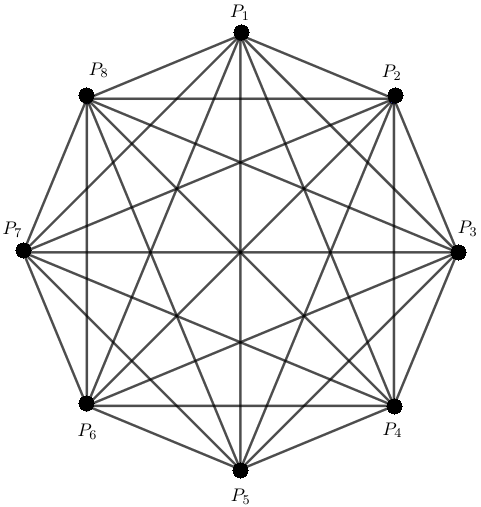} & \includegraphics[scale=0.35]{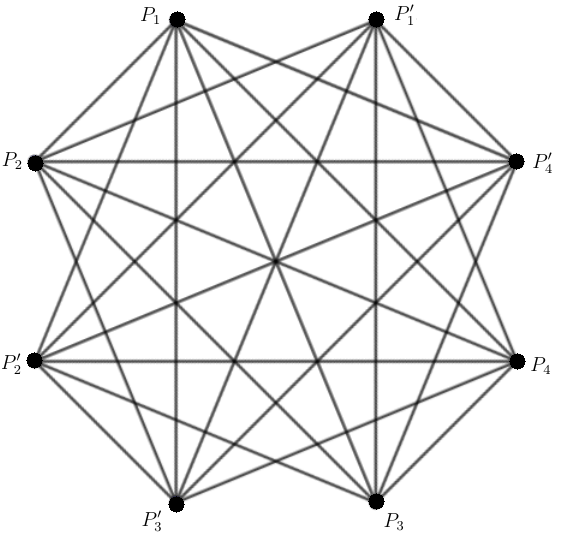}  \\
\hline
\hline
  & \\ 
$W_{BS}^{9}$, $W_{F}^{9}$ & $W_{BS}^{13}$, $W_{F}^{13}$ \\ 
\includegraphics[scale=0.35]{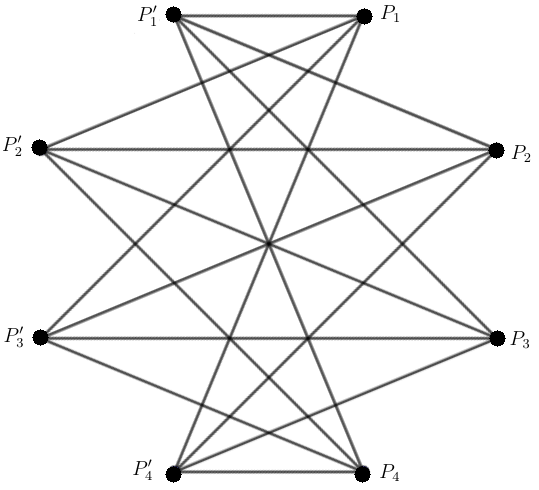} & \includegraphics[scale=0.45]{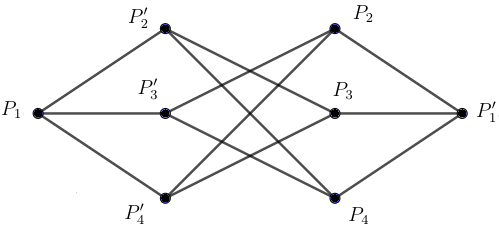}\\
\hline
\end{tabular} 
\caption{\label{tab:BS67913}\scriptsize{Configurations of the eight quadruple points of the Enriques-Fano threefolds $W_{F}^{p}\subset \mathbb{P}^p$ and $W_{BS}^{p} \xhookrightarrow{\phi_{\mathcal{L}}} \mathbb{P}^p$, for $p=6,7,9,13$.}}
\end{table}

\begin{table}[ht]\scriptsize
\begin{tabular}{|l|l|} 
\hline
  & \\ 
$W_{BS}^{8}$ & $W_{BS}^{10}$\\ 
\includegraphics[scale=0.45]{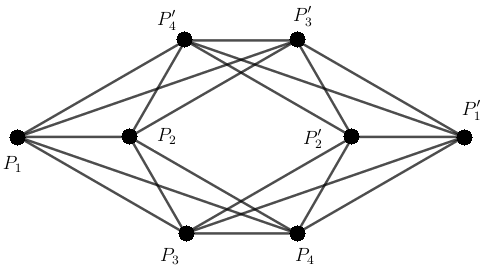} & \includegraphics[scale=0.45]{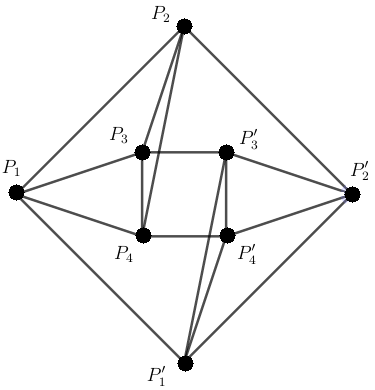}\\
\hline
\end{tabular} 
\caption{\label{tab:BS810}\scriptsize{Configurations of the eight quadruple points of the Enriques-Fano threefolds $W_{BS}^{p} \xhookrightarrow{\phi_{\mathcal{L}}} \mathbb{P}^p$, for $p=8,10$.}}
\end{table}

\begin{table}[ht]\scriptsize
\begin{tabular}{|l|l|} 
\hline
  & \\ 
$W_{P}^{13}$, $W_{P}^{17}$,  & $W_{KLM}^{9}$\\ 
\includegraphics[scale=0.45]{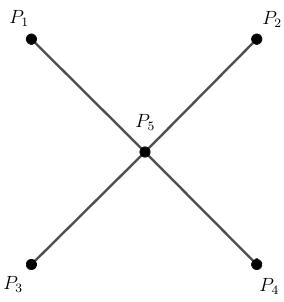} & \includegraphics[scale=0.45]{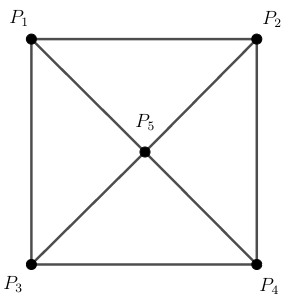}\\
\hline
\end{tabular} 
\caption{\label{tab:proKLM}\scriptsize{Configurations of the five singular points of the Enriques-Fano threefolds $W_{P}^{13}$, $W_{P}^{17}$ and $W_{KLM}^{9}$.}}
\end{table}

\end{document}